\documentclass[leqno,11pt]{amsart}
\usepackage{amssymb, amsmath}
\usepackage[usenames,dvipsnames]{xcolor}
\usepackage{color,ulem,textcomp}

\numberwithin{equation}{section}
\usepackage{amsmath}

\usepackage{enumerate}
\usepackage{graphicx}
\usepackage{cancel}
\usepackage{dsfont}

\theoremstyle{definition}
\theoremstyle{plain}
\newtheorem{thm}{Theorem}
\newtheorem{prop}{Proposition}[section]
\newtheorem{lem}[prop]{Lemma}
\newtheorem{Cor}[prop]{Corollary}

 \newtheorem{Rmk}[prop]{Remark} 

\def\bC {\mathbb{C}}

 \def\N {\mathbb{N}}
\def\R {\mathbb{R}}
\def\Sp {\mathbb{S}}

\def\T {\mathbb{T}}
\def\Z {\mathbb{Z}}

\def\b {{\beta}}

\def\eps {{\varepsilon}}
\def\e {{\varepsilon}}

\newcommand{\Div}{\operatorname{div}}

\newcommand{\ba}{\begin{aligned}}
\newcommand{\ea}{\end{aligned}}

\newcommand{\be}{\begin{equation}}
\newcommand{\ee}{\end{equation}}

\numberwithin{equation}{section}

\begin{document}
\author{{\sc Isabelle Gallagher}}
\address{DMA, \'Ecole normale sup\'erieure, CNRS, PSL   University, 75005 Paris, France
and UFR de math\'ematiques, Universit\'e Paris-Diderot, Sorbonne Paris-Cit\'e, 75013 Paris, France
E-mail: {\tt isabelle.gallagher@ens.fr}}
\author{{\sc Isabelle Tristani}}
\address{D\'epartement de math\'ematiques et applications, \'Ecole normale sup\'erieure, CNRS, PSL   University, 45 rue d'Ulm, 75005 Paris, France
E-mail: {\tt isabelle.tristani@ens.fr}}
\title[]{On the convergence of  smooth solutions from     Boltzmann   to     Navier-Stokes}
\begin{abstract}
In this work, we are interested in the link between strong solutions of the Boltzmann and the Navier-Stokes equations. To justify this connection, our main idea is to use information on the limit system (for instance the fact that the Navier-Stokes equations are globally wellposed in two space dimensions or when the data are small). In particular we prove   that the life span of the solutions to the rescaled Boltzmann equation is bounded from below by that of the Navier-Stokes system.   We deal with general initial data in the whole space in dimensions 2 and 3, and also with well-prepared data in the case of periodic boundary conditions. 
\end{abstract}
\maketitle


\section{Introduction}
In this paper, we are interested in the link between the Boltzmann and Navier-Stokes equations. The problem of deriving hydrodynamic equations from the Boltzmann equation   goes back to Hilbert~\cite{Hilbert} and can be seen as an intermediate step in the problem of deriving macroscopic equations from microscopic ones, the final goal being to obtain a unified description of gas dynamics including all the different scales of description. The first justifications of this type of limit (mesoscopic to macroscopic equations) were formal and based on asymptotic expansions,  given by Hilbert~\cite{Hilbert} and Chapman-Enskog~\cite{ChapEns}. Later on, Grad introduced a new formal method to derive hydrodynamic equations from the Boltzmann equation in~\cite{Gradhydro} called the moments method. 

The first convergence proofs based on asymptotic expansions were given by Caflisch~\cite{Caflisch} for the compressible Euler equation. The idea here was to justify the limit up to the first singular time for the limit equation. In this setting, let us also mention the paper by Lachowicz~\cite{Lachowicz} in which more general initial data are treated and also the paper by De Masi, Esposito and Lebowitz~\cite{DeMasi-Esposito-Lebowitz} in which roughly speaking, it is proved that in the torus, if the Navier-Stokes equation has a smooth solution on some interval $[0,T_*]$, then there also exists a solution to the rescaled Boltzmann equation on this interval of time. Our main theorem is actually reminiscent of this type of result, also in the spirit of~\cite{BMN,CDGG,Grenier,Schochet}: we try to use information on the limit system (for instance the fact that the Navier-Stokes equations are globally wellposed in two space dimensions) to obtain results on the life span of solutions to the rescaled Boltzmann equation. We would like to emphasize here that in our result, if the solution to the limit equation is global (regardless of its size), then, we are able to construct a global solution to the Boltzmann equation, which is not the case in the aforementioned result. Moreover, we treat both the case of the torus and of the whole space. 

Let us also briefly recall some convergence proofs based on spectral analysis, in the framework of strong solutions close to   equilibrium introduced by Grad~\cite{Grad} and Ukai~\cite{Ukai} for the Boltzmann equation. They  go back to Nishida~\cite{Nishida} for the compressible Euler equation (this is a local in time result) and this type of proof was also developed for the incompressible Navier-Stokes equation by Bardos and Ukai~\cite{Bardos-Ukai} in the case of smooth global solutions in three space dimensions, the initial velocity field being taken small. These results use the description of the spectrum of the linearized Boltzmann equation performed by Ellis and Pinsky in~\cite{Ellis-Pinsky}. In~\cite{Bardos-Ukai}, Bardos and Ukai only treat the case of the whole space, with a smallness assumption on the initial data which allows them to work with global solutions in time. In our result, no smallness assumption is needed and we can thus treat the case of non global in time solutions to the Navier-Stokes equation. We would also like to emphasize that Bardos and Ukai also deal with the case of ill-prepared data but their  result is not strong up to~$t=0$ contrary to the present work  (where the strong convergence holds in an averaged sense in time).

More recently, Briant in~\cite{Briant}  and Briant, Merino-Aceituno and Mouhot in~\cite{Briant-Mouhot-Merino} obtained  convergence to equilibrium results    for the rescaled Boltzmann equation uniformly in the rescaling parameter using hypocoercivity and ``enlargement methods", that enabled them to weaken the assumptions on the data down to Sobolev spaces with polynomial weights.

Finally, let us mention that this problem has been extensively studied in the framework of weak solutions, the goal being to obtain solutions for the fluid models from   renormalized solutions introduced by Di Perna and Lions in~\cite{DiPerna-Lions} for the Boltzmann equation. We shall not make an extensive presentation of this program as it is out of the realm of this study, but let us mention that it was started by Bardos, Golse and Levermore at the beginning of the nineties in~\cite{bgl1,bgl2} and was continued by those  authors, Saint-Raymond, Masmoudi, Lions among others. We mention here  a (non exhaustive) list of papers which are part of this program: see~\cite{GSR1,GSR2,Levermore-Masmoudi,lionsmasmoudi,Saint-Raymond}.

\subsection{The models}
We start by introducing the Boltzmann equation which models the evolution of a rarefied gas through the evolution of the density of particles $f=f(t,x,v)$ which depends on time $t \in \R^+$, position $x \in \Omega$ and velocity $v \in \R^d$ when only binary collisions are taken into account. We take $\Omega $ to be the~$d$-dimensional unit periodic box~${\mathbb T}^d$  (in which case the functions we shall consider will be assumed to be mean free) or the whole space~$  \R^d$~in dimension $2$ or $3$. We  focus here  on hard-spheres   collisions (our proof should be adaptable to the case of hard potentials with cut-off). The Boltzmann equation reads:
$$
\partial_t f + v \cdot \nabla_x f = {1 \over \eps} Q(f,f)
$$
where $\eps$ is the Knudsen number which is the inverse of the average number of collisions for each particle per unit time and $Q$ is the Boltzmann collision operator. It is defined as
$$
Q(g,f):=\int_{\R^d \times \Sp^{d-1}} |v-v_*| \left[g'_*f' - g_* f \right] \, d\sigma \, dv_* \,.
$$ 
Here and below, we are using the shorthand notations $f=f(v)$, $g_*=g(v_*)$, $f'=f(v')$ and $g'_*=g(v'_*)$. In this expression,  $v'$,~$v'_*$ and~$v$, $v_*$ are the velocities of a pair of particles before and after collision.  More precisely we parametrize
the  solutions to the conservation of momentum and energy (which are the physical laws of elastic collisions): 
$$
v+v_*=v'+v'_* \, ,
$$
$$
|v|^2+|v_*|^2=|v'|^2+ |v'_*|^2\, ,
$$
so that the pre-collisional velocities are given by 
$$
v':=\frac{v+v_*}{2} + \frac{|v-v_*|}{2}\, \sigma \,, \quad v'_*:=\frac{v+v_*}{2} - \frac{|v-v_*|}{2}\, \sigma\, , \quad \sigma \in \Sp^{d-1} \, .
$$
Taking $\eps$ small has the effect of enhancing the role of collisions and thus when $\eps \to 0$, in view of Boltzmann~$H$-theorem, the solution looks more and more like a local thermodynamical equilibrium. 
As suggested in previous works~\cite{bgl1}, we consider the following rescaled Boltzmann equation in which an additional dilatation of the macroscopic  time scale has been done in order to be able to reach the Navier-Stokes equation in the limit:
\begin{equation}\label{scaledboltzmann}
\partial_t f^\e + \frac1\e v \cdot \nabla_x f^\e = \frac1{\e^2} Q(f^\e,f^\e) \quad \mbox{in} \quad \R^+ \times  \Omega \times {\mathbb R}^d \, .
\end{equation}
It is a well-known fact that global equilibria of the Boltzmann equation are local Maxwellians in velocity. In what follows, we only consider the following global normalized Maxwellian defined by
$$
M(v):=\frac1{(2\pi)^\frac d2} e^{-\frac{|v|^2}2} \, .
$$
To relate the Boltzmann equation to the incompressible Navier-Stokes equation, we look at equation~\eqref{scaledboltzmann} under the following linearization of order $\eps$:
$$
f^\e(t,x,v)= M(v) + \e M^\frac12(v) g^\e(t,x,v)\, .
$$
Let us recall that taking $\eps$ small in this linearization corresponds to taking a small Mach number, which enables one to get in the limit the incompressible Navier-Stokes equation. If~$f^\e$ solves~(\ref{scaledboltzmann}) then equivalently~$g^\e$ solves 
\begin{equation}\label{eqgeps}
\partial_t g^\e + \frac1\e v \cdot\nabla_x g^\e = \frac1{\e^2}  Lg^\e +  \frac1\e\Gamma(g^\e,g^\e) \quad \mbox{in} \quad \R^+ \times  \Omega\times {\mathbb R}^d\, 
\end{equation}
with
\begin{equation} \label{defoperator}
\begin{aligned}
 Lh & :=M^{-\frac12}\big(
Q(M,M^\frac12 h) + Q(M^\frac12 h,M)
\big)\\
 \mbox{and} \quad \Gamma (h_1,h_2)  &:=\frac12M^{-\frac12}
\big(
Q(M^\frac12 h_1,M^\frac12 h_2) + Q(M^\frac12 h_2,M^\frac12 h_1)\big)\, .
\end{aligned}
\end{equation}
In the following we shall denote by~$
\Pi_{L}   $
  the orthogonal projector onto $\operatorname{Ker} L$ . It is well-known that 
$$
\operatorname{Ker} L = \hbox{Span} \big(M^\frac12, v_1M^\frac12, \dots, v_dM^\frac12, |v|^2M^\frac12\big) \,.
$$ 
Appendix~\ref{appendixbotlz} collects a number of well-known results on the Cauchy problem for~(\ref{eqgeps}).

\subsection{Notation}  
Before stating the convergence result, let us define the functional setting we shall be working with.  For any real number~$\ell\geq 0$, the space~$H^\ell_x$ (which we sometimes denote by $H^\ell$ or~$H^\ell(\Omega)$) is the space of functions defined on~$\Omega$ such that
$$
\|f\|_{H^\ell_x}^2:=\int_{\R^d} \langle \xi\rangle^{2\ell} |\widehat f (\xi)|^2\, d\xi <\infty \quad \text{if} \quad \Omega = \R^d \, ,
$$
or
$$
\|f\|_{H^\ell_x}^2:=\sum_{\xi \in \Z^d} \langle \xi\rangle^{2\ell} |\widehat f (\xi)|^2 <\infty \quad \text{if} \quad \Omega = \T^d \,,
$$
where~$\widehat f $ is the Fourier transform of~$f$ in $x$ with dual variable $\xi$ and where
$$
\langle \xi\rangle^{2 }:= (1+|\xi|)^2\,.
$$
We shall sometimes note~${\mathcal F}_x f$ for~$\widehat f$. We also recall the definition of homogeneous Sobolev spaces (which are Hilbert spaces for~$s<d/2$), defined through the norms
$$
 \|f\|_{ \dot H^s(\R^d)}^2:= \int_{\R^d} | \xi |^{2s}   |\widehat f(\xi)|^2 \, d\xi  \quad \mbox{and} \quad    \|f\|_{ \dot H^s(\T^d))}^2:=\sum_{\xi \in \Z^d} | \xi |^{2s}   |\widehat f(\xi)|^2 \, .
 $$ 
 In the case when~$ \Omega = \T^d$ we further make the assumption that the functions under study are mean free.
 Note that for mean free functions defined on~$\T^d$,  homogeneous and inhomogeneous norms are equivalent.
 We also define~$W^{\ell,\infty}_x$  (or~$W^{\ell,\infty}$ or~$W^{\ell,\infty}(\Omega)$)  the space of functions defined on~$\Omega$ such that
$$
\|f\|_{W^{\ell,\infty}_x} :=\sum_{|\alpha| \le \ell} \sup_{x \in \Omega}|\partial^\alpha_x f(x)|<\infty \,  ,
$$

We set, for any real number~$k$
$$
L^{\infty,k}_v :=\Big\{f = f(v) \, / \, \langle v\rangle^k f \in L^\infty (\R^d)\Big\}
$$
endowed with the norm
$$
\|f\|_{L^{\infty,k}_v}:=\sup_{v\in\R^d} \,   \langle v\rangle^k |f(v) | \,.
$$
The following spaces will be of constant use:
\begin{equation} \label{Xellk}
X^{\ell,k}:=\Big\{f = f(x,v) \, / \,\|f(\cdot,v)\|_{H^\ell_x}\in L^{\infty,k}_v, \, \, \sup_{|v| \ge R} \langle v\rangle^k \|f(\cdot,v)\|_{H^\ell_x} \xrightarrow[R \to \infty]{}0\Big\}
\end{equation}
{(note that the $R\to\infty$ property included in this definition is here to ensure the continuity property of the semi-group generated by the non homogeneous linearized Boltzmann operator~\cite{Ukai})} and we set
$$
\|f\|_{\ell,k}:=\sup_{v\in\R^d} \,  \langle v\rangle^k \big\|f(\cdot,v)\big\|_{H^\ell_x}\,.
$$

\subsection{Main result} Let us now present our main result, which
states that the  hydrodynamical limit of~(\ref{scaledboltzmann}) as~$\e$ goes to zero is   the Navier-Stokes-Fourier system associated with the Boussinesq equation   which writes
\begin{equation}\label{NSF}
\left\{
\begin{aligned}
\partial_t u + u \cdot \nabla u- \mu_1 \Delta u &= - \nabla p \\
\partial_t \theta  + u \cdot \nabla \theta- \mu_2 \Delta \theta &= 0 \\
\mbox{div}\, u & = 0 \\
 \nabla(\rho+\theta) &= 0 \, .
\end{aligned}
\right.
\end{equation}
In this system~$\theta$ (the temperature),~$\rho$ (the density) and~$p$ (the pressure) are scalar unknowns and~$u$ (the velocity) is a~$d$-component unknown  vector field. The pressure can actually be eliminated from the equations by applying to the momentum equation the projector~${\mathbb P}$ onto the space of divergence free vector fields. This projector is bounded over~$H^\ell_x$ for all~$\ell$, and in~$L^p_x$ for all~$1<p<\infty$.  To define the viscosity coefficients, let us introduce the two unique functions $\Phi$ (which is a matrix function) and $\Psi$ (which is a vectorial function) orthogonal to~$\operatorname{Ker} L$ such that 
$$
M^{-\frac12}L\big(M^{\frac12} \Phi\big) =  {|v|^2\over d } {\rm{Id}} -v\otimes v 
\quad \text{and} \quad 
M^{-\frac12}L\big(M^{\frac12} \Psi\big) =v\Big(\frac{d+2}{2}-{|v|^2\over 2}\Big) \, .
$$
The viscosity coefficients are then defined by
$$ 
\begin{aligned}
\mu_1 :=\frac{1}{(d-1)(d+2)}\int\Phi : L\big(M^{\frac12}\Phi\big) M^{\frac12} \, dv \quad \text{and} \quad
\mu_2  :=\frac{2}{d(d+2)} \int \Psi \cdot L\big(M^{\frac12}\Psi\big) M^{\frac12} \, dv  \, .
\end{aligned}
$$
Before stating our main results, let us mention that  Appendix~\ref{appendixNS}
provides some useful results on the Cauchy problem for~(\ref{NSF}).  
\begin{thm}\label{mainthmwp}
 Let~$\ell>d/2$ and~$k>d/2 +1$ be given and consider~$(\rho_{\rm in},u_{\rm in},\theta_{\rm in})$ in~$ H^\ell
 (\Omega)$ {if~$\Omega \neq \R^2$ and in $ H^\ell
 (\Omega) \cap L^1(\Omega)$ if $\Omega = \R^2$}. If $\Omega= \T^d$, we furthermore assume that~$\rho_{\rm in},u_{\rm in},\theta_{\rm in}$ are mean free. Define   
\begin{equation}\label{relationbarnotbar}
\bar \rho_{\rm in}:= \frac2{d+2} \rho_{\rm in}- \frac{d}{d+2} \theta_{\rm in} \, , \quad \bar u_{\rm in} = \mathbb{P} u_{\rm in}, \quad \bar \theta_{\rm in} := - \bar \rho_{\rm in} \,.
\end{equation}
 Let~$(\rho,u,\theta)$
be the unique solution to~{\rm(\ref{NSF})} associated with the initial data~$(\bar \rho_{\rm in},\bar u_{\rm in},\bar \theta_{\rm in})$ on a time interval~$[0,T]$. Set
 \begin{equation}\label{defg0dataperiodi}
\bar g_{\rm in}(x,v):=M^\frac12(v)\Big(\bar \rho_{\rm in} (x) + \bar u_{\rm in}(x) \cdot v + \frac12(|v|^2- d) \bar \theta_{\rm in} (x)\Big)\, ,
\end{equation}
and define on~$[0,T] \times \Omega \times \R^d$
\begin{equation}\label{defg0solution}
g
(t,x,v) := M^\frac12(v)\Big(\rho (t,x) + u (t,x) \cdot v + \frac12(|v|^2- d) \theta(t,x) \Big)\, .
\end{equation}
\noindent $\bullet$ $ $ {\it The well prepared case: } 
Assume~$\Omega = \T^d$ or~$\R^d$, $d= 2,3$.  There is~$\e_0>0$ such that for all~$\eps \le \eps_0$  there is a unique solution~$g^\e$ to~{\rm(\ref{eqgeps})} in~{$L^\infty([0,T],X^{\ell,k})$} with initial data~$\bar g_{\rm in}$, and it satisfies 
\begin{equation}\label{limgepsg0wp}
\lim_{\eps\to 0} \,   \big\|
g^\eps-g
\big\|_{L^\infty([0,T],X^{\ell,k})} = 0 \, .
\end{equation}
Moreover, if the solution $(\rho,u,\theta)$ to~\eqref{NSF} is defined on $\R^+$,  then~$\eps_0$ depends only on the initial data and not on~$T$ and there holds
$$
\lim_{\eps\to 0}  \big\|
 g^\eps-g  \big\|_{L^\infty(\R^+,X^{\ell,k})} = 0 \, .
$$ 
\noindent $\bullet$ $ $ {\it The ill prepared case: }  Assume~$\Omega = \R^d$, $d= 2,3$.  For   all initial data~$g_{\rm in}$ in~$X^{\ell,k}$ satisfying
$$
\begin{aligned}
\rho_{\rm in}(x)= \int_{\R^d} g_{{\rm in}} (x,v)M^\frac12 (v) \, dv \, , \quad  u_{\rm in} (x)= \int_{\R^d} v\,  g_{{\rm in}} (x,v)M^\frac12 (v) \, dv\, , \\
 \theta_{\rm in} (x)= {1 \over d} \int_{\R^d} (|v|^2-d) g_{{\rm in}} (x,v)M^\frac12 (v) \, dv \, ,
\end{aligned}
$$
there is~$\e_0>0$ such that  for all~$\e\leq \e_0$ there is a unique solution~$g^\e$ to~{\rm(\ref{eqgeps})} in~{$L^\infty([0,T],X^{\ell,k})$} with initial data~$g_{\rm in}$.
It satisfies for all~$p>2/(d-1)$
\begin{equation}\label{limgepsg0ip}
\lim_{\eps\to 0} \,   \big\|
g^\eps-g
\big\|_{L^\infty([0,T],X^{\ell,k}) + {L^p(\R^+,L^{\infty,k}_v(W^{\ell,\infty}_x{+ H^\ell_x})(\R^d ))}} = 0 \, .
\end{equation}
Moreover, if the solution $(\rho,u,\theta)$ to~\eqref{NSF} is defined on $\R^+$, then~$\eps_0$ depends only on the initial data and not on~$T$ and there holds
$$
\lim_{\eps\to 0} \,  \big\|
 g^\eps-g 
\big\|_{L^\infty(\R^+,X^{\ell,k}) +L^p(\R^+,L^{\infty,k}_v(W^{\ell,\infty}_x{+ H^\ell_x})(\R^d )) }= 0 \, .
$$ 
 \end{thm}
 Notice that the last assumption (that  the solution $(\rho,u,\theta)$ to~\eqref{NSF} is defined on $\R^+$) always holds when~$d=2$ and is also known to hold for small data in dimension 3 or without any smallness assumption in some cases (see examples in~\cite{cg1} in the periodic case, \cite{cg} in the whole space for instance): see Appendix~\ref{appendixNS}
 for more on~(\ref{NSF}).
  \begin{Rmk} 
We choose   initial data for~{\rm(\ref{eqgeps})}
 which does not depend on~$\eps$, but it is easy
 to modify the proof if the initial data is a family
 depending on~$\eps$, as long as it is compact in~$X^{\ell,k}$. 
   \end{Rmk}
  \begin{Rmk} 
{In the case of $\R^2$, we have made the additional assumption that our initial data lie in $L^1(\Omega)$. Actually, it would be enough to suppose that the projection onto the kernel of~$L$ is in $L^1(\Omega)$.}
  \end{Rmk}
  \begin{Rmk} 
Let us mention that if we work with smooth data, we can obtain a rate of convergence of $\eps^\frac12$ in~\eqref{limgepsg0wp} and~\eqref{limgepsg0ip} -- which is probably not the optimal rate.
  \end{Rmk}
  
 \begin{Rmk}
  As noted in~\cite{luzhang}, the original solution to the Boltzmann equation, constructed as~$f^\e(t,x,v)= M(v) + \e M^\frac12(v) g^\e(t,x,v)$, is nonnegative under our assumptions, as soon as the initial data is nonnegative (which is an assumption that can be made in the statement of Theorem~\ref{mainthmwp}).  
  \end{Rmk} 

 The proof of the theorem mainly relies on a fixed point argument, which enables us to prove that the equation satisfied by the difference~$h^\e$ between the solution~$g^\e$  of the Boltzmann equation and its expected limit~$g$  does have a solution (which is arbitrarily small) as long as~$g$ exists. In order to develop this fixed point argument, we have to filter the unknown~$h^\e$  by some well chosen exponential function which depends on the solution of the Navier-Stokes-Fourier equation. This enables us to obtain a contraction estimate. Let us also point out that the analysis of the operators that appear in the equation on~$h^\e$ is akin to the one made by Bardos and Ukai~\cite{Bardos-Ukai} and it relies heavily on the Ellis and Pinsky decomposition~\cite{Ellis-Pinsky}.
In the case of ill-prepared data,  the  fixed point argument needs some adjusting. 
Indeed the linear propagator consists in two classes of operators, one of which vanishes identically when applied to well-prepared case, and in general decays to zero in an averaged sense in time due to   dispersive properties.
 Consequently, we choose to  apply the fixed point theorem not to~$h^\e$ but to the difference between~$h^\e$ and those dispersive-type remainder terms.  This induces some additional terms to estimate, which turn out to  be harmless thanks to their dispersive nature.

 \noindent\textbf{Acknowledgments.} The authors thank Fran\c{c}ois Golse for his valuable advice. The second author thanks the ANR EFI:  ANR-17-CE40-0030. 
 
  \section{Main steps of the proof  of  Theorem~\ref{mainthmwp}}
\subsection{Main reductions}
Given~$g_{\rm in} \in X^{{\ell},k}$, the classical Cauchy theory on the Boltzmann equation recalled in Appendix~\ref{appendixbotlz} states that there is a time~$T^\eps$ and a  unique solution~$g^\e $ in~$ \mathcal{C}^0([0,T^\e],X^{{\ell},k})$ to~{\rm(\ref{eqgeps})} associated with the data~$g_{\rm in}$.  The proof of  Theorem~\ref{mainthmwp}   consists   in proving that the life span of~$g^\eps$ is actually at least that of the limit system~(\ref{NSF}) by proving the convergence result~(\ref{limgepsg0wp}).  Our proof is based on a fixed point argument of the following type.
\begin{lem}
\label{cacciopoli+}
  Let~$X$ be a Banach space, let~${\mathcal L}$ be a  continuous linear map
from~$X$ to~$X$,  and let~${\mathcal B}$ be a bilinear map from~$X\times X$ to~$X$.
Let us define
\[
\|{\mathcal L}\|  := \sup_{\|x\|=1} \|{\mathcal L}x\|\quad\hbox{and}\quad
\|{\mathcal B}\| := \sup_{\|x\|=\|y\|=1} \|{\mathcal B}(x,y)\| \, .
\]
If~$\|{\mathcal L}\| <1$, then for any~$x_{0}$ in~$X$ such that
\[
\|x_{0}\|_{X}< \frac{(1-\|{\mathcal L}\| )^2} {4\|{\mathcal B}\| } 
\]
the equation
\[
x=x_{0}+{\mathcal L}x+{\mathcal B}(x,x)
\]
has a unique solution in the ball of center~$0$ and
radius~$\displaystyle \frac {1-\|{\mathcal L}\| }{2\|{\mathcal B}\|} $  and there is a constant~$C_0$ such that
$$
\|x\|\le C _0\|x_0\| \, .
$$
\end{lem}
We are now going to give a formulation of the problem which falls within this framework. To this end, let us introduce the integral formulation of~(\ref{eqgeps})
 \begin{equation}\label{integralformulation}
g^\e(t) =U^\eps(t)  g _{{\rm in}}    +    \Psi^\e(t) \big(g^\e  ,g^\e\big) 
\end{equation}
where~$U^\eps(t) $ denotes the semi-group associated with~$\displaystyle - \frac1\e v \cdot\nabla_x   + \frac1{\e^2}  L  $ (see~\cite{Ukai, Bardos-Ukai} as well as  Appendix~\ref{specBoltz}) and where
\begin{equation}\label{defPsiepsh1h2}
 \Psi^\e(t)  (f_1,f_2) :=\frac1\e \int_0^t U^\e (t-t') \Gamma \big(f_1(t'),f_2(t')\big) \, dt' \, ,
\end{equation}
with~$\Gamma$ defined in~(\ref{defoperator}). It follows from the results and notations recalled in Appendix~\ref{specBoltz}
 (in particular Remark~\ref{defPsi}) that given~$\bar  g_{{\rm in}} \in X^{{\ell},k}$   of the form~(\ref{defg0dataperiodi})
the function~$g$ defined in~(\ref{defg0solution}) satisfies
$$
g(t) =U (t)\bar    g_{{\rm in}}    +    \Psi(t)  (g  ,g  ) \, .
$$
It will be useful in the following to assume that~$   g_{{\rm in}} $  and~$ \bar  g_{{\rm in}} $ are  as smooth and decaying as necessary in~$x$. So we consider   families~$ ( \rho^\eta_{{\rm in}}, u^\eta_{{\rm in}},\theta^\eta_{{\rm in}})_{\eta \in (0,1)} $ in the Schwartz class~$ {\mathcal S}_x$, as well as~$  ( g^\eta_{{\rm in}})_{\eta \in (0,1)} $ and~$  (\bar  g^\eta_{{\rm in}})_{\eta \in (0,1)} $ related by
\begin{equation}\label{bargeta}
\bar g^\eta_{\rm in}(x,v)=M^\frac12(v)\Big(\bar \rho^\eta_{\rm in} (x) + \bar u^\eta_{\rm in}(x) \cdot v + \frac12(|v|^2- d) \bar \theta^\eta_{\rm in} (x)\Big)
\end{equation}
with~$ (\bar \rho^\eta_{{\rm in}}, \bar u^\eta_{{\rm in}},\bar\theta^\eta_{{\rm in}}) $ defined by notation~(\ref{relationbarnotbar}), with
\begin{equation}\label{bargetageta}
\begin{aligned}
\rho^\eta_{\rm in}(x)= \int_{\R^d} g^\eta_{{\rm in}} (x,v)M^\frac12 (v) \, dv \, , \quad  u^\eta_{\rm in} (x)= \int_{\R^d} v\,  g^\eta_{{\rm in}} (x,v)M^\frac12 (v) \, dv\, , \\
 \theta^\eta_{\rm in} (x)= {1 \over d} \int_{\R^d} (|v|^2-d) g^\eta_{{\rm in}} (x,v)M^\frac12 (v) \, dv \, ,
\end{aligned}
\end{equation}
and such that
\begin{equation} \label{deltaineta}
\begin{aligned}
 \forall\, \eta \in (0,1) \, , \quad g^\eta_{{\rm in}} \, , \bar  g^\eta_{{\rm in}}\in {\mathcal S}_{x,v} \quad \mbox{and} \quad  \|    \delta^\eta_{{\rm in}}  \|_{\ell,k} +\|  \bar\delta^\eta_{{\rm in}}  \|_{\ell,k} \leq \eta \, , \\
    \mbox{with} \quad \delta^\eta_{{\rm in}}  := g^\eta_{{\rm in}}  -   g _{{\rm in}} 
\, \quad    \mbox{and} \quad \bar   \delta^\eta_{{\rm in}}  :=\bar   g^\eta_{{\rm in}}  -  \bar   g _{{\rm in}}  \,.
\end{aligned}
\end{equation}
If $\Omega = \R^2$, we furthermore assume, recalling that~$( \rho_{{\rm in}}, u_{{\rm in}},\theta_{{\rm in}})$ belong to~$H^\ell_x \cap L^1_x$, that  
\begin{equation} \label{deltainetaR2}
\|    \delta^\eta_{{\rm in}}  \|_{L^2_vL^1_x} \le \eta \, .
\end{equation}
Thanks to the stability of the Navier-Stokes-Fourier equation recalled in Appendix~\ref{appendixNS}
we know that
\begin{equation}\label{eqgeta}
g^\eta(t) :=U (t)\bar    g^\eta_{{\rm in}}    +    \Psi(t)  (g^\eta  ,g^\eta  )
\end{equation}
satisfies  
\begin{equation}\label{estimatebardeltaepseta0}
\lim_{\eta\to0}\big\|g^\eta - g\big\|_{L^\infty([0,T],X^{\ell,k})}=0 \, ,
\end{equation}
uniformly in~$T $ if the solution~$g$ is global.
Moreover
setting
\begin{equation}\label{defgepseta}
g^{\eps,\eta} :=g^\e + \delta^{\eps,\eta}  \, , \quad  \delta^{\eps,\eta} (t) :=  U^\eps(t)   \delta^\eta_{{\rm in}}
\end{equation}
there holds
\begin{equation}\label{eqgepseta}
g^{\eps,\eta}(t) =U^\eps(t)  g^\eta _{{\rm in}}    +    \Psi^\e(t) \big(g^{\eps,\eta}-  \delta^{\eps,\eta}  ,g^{\eps,\eta}-   \delta^{\eps,\eta} \big) \, .
\end{equation}
Thanks to~(\ref{deltainetaR2}) and the continuity  of~$U^\eps(t) $   recalled in Lemma~\ref{lem:Uepscont} we know that
  \begin{equation}\label{estimatedeltaepseta0}
 \| \delta^{\eps,\eta}   \|_{L^\infty(\R^+,{X}^{\ell,k})}\lesssim \eta
\end{equation} 
hence with~(\ref{estimatebardeltaepseta0}) it is enough to prove the convergence results~(\ref{limgepsg0wp}) and~(\ref{limgepsg0ip}) with~$ g^\eps $ and~$g$ respectively replaced by~$ g^{\eps,\eta} $ and~$ g^{\eta} $  (the parameter~$\eta$ will be converging to zero uniformly in~$\e$).  Indeed we have the following inequality
 $$
\|g^\eps-g\|_{L^\infty ([0,T],{X}^{\ell,k})} \le \|\delta^{\eps,\eta}\|_{L^\infty (\R^+,{X}^{\ell,k})} + \|g-g^\eta\|_{L^\infty ([0,T],{X}^{\ell,k})} + \|g^{\eps,\eta}-g^\eta\|_{L^\infty ([0,T],{X}^{\ell,k})} \, ,
$$
which is uniform in time if~$  g  _{{\rm in}}  $ (and hence also~$  g^\eta  _{{\rm in}}  $ if~$\eta$ is small enough, thanks to Proposition~\ref{propNSF}) generates a global solution to the limit system.
In order to achieve this goal let us now write the equation satisfied by~$ g^{\eps,\eta}-g^{\eta}  $. Our plan is   to conclude thanks to Lemma~\ref{cacciopoli+}, however there are two difficulties in this strategy. First,  linear terms   appear in the equation on~$ g^{\eps,\eta}-g^{\eta} $, whose operator norms are of the order of   norms of~$g^\eta$  which are not small -- those linear operators therefore do not satisfy the assumptions of Lemma~\ref{cacciopoli+}. 
 In order to circumvent this difficulty we shall introduce weighted Sobolev spaces, where the weight is exponentially small in~$g^\eta$ in order for the linear operator to become a contraction. The second difficulty in the ill-prepared case is that  the linear propagator~$U^\eps-U$ acting on the initial data can be decomposed into several orthogonal operators (as explained in Appendix~\ref{specBoltz}), some of which
 vanish in the well-prepared case only, and are dispersive (but not small in the energy space) in the  ill-prepared case. These terms need to be removed from~$ g^{\eps,\eta}-g^{\eta}$ if one is to apply the fixed point lemma  in the energy space.
 All these reductions are carried out in the following lemma, where we prepare the problem so as to apply Lemma~\ref{cacciopoli+}. 
\begin{lem} \label{lem:reduction}
Let~$r >4$ and~$\lambda \geq 0$ be given. With the notation introduced in Lemma~{\rm\ref{EllisPinsky}} Remark~{\rm\ref{defU}}
set
$$\overline\delta^{\eps,\eta}  (t)  :=U^\eps_{\rm{disp}}(t) g^\eta_{{\rm in}} + U^{\eps \sharp}(t)g^\eta_{{\rm in}} \quad \mbox{and}\quad\widetilde \delta^{\eps,\eta} (t)   := U^\eps(t) (g_{\rm in}^\eta - \bar g^\eta_{{\rm in}}) - \overline\delta^{\eps,\eta} (t)  \,.$$
Finally set $$\overline g^{\eps,\eta}:=g^\eta + \overline\delta^{\eps,\eta}$$ and define $h^{\e,\eta}_\lambda$ as the solution of the equation
\begin{equation}\label{equationhepsetalambda}
h^{\e,\eta}_\lambda(t) = \mathcal{D}_\lambda^{\e}(t)   +  \mathcal{L}_\lambda^{\e}(t) h^{\e,\eta}_\lambda(t) + \Phi^\e_\lambda(t)  (h^{\e,\eta}_\lambda ,h^{\e,\eta}_\lambda   ) 
 \end{equation}
where (dropping the dependence on~$\eta$ on the operators to simplify) we have written
 \begin{equation}\label{defopslambda}
\begin{aligned}
 \mathcal{D}_\lambda^\e (t)  & := e^{-\lambda \int_0^t \|\overline g^{\eps,\eta}(t')\|_{\ell,k}^r\, dt' \, 
}  \mathcal{D}^\e  (t)   \\
 \mathcal{D}^\e  (t) 
 & :=  \widetilde \delta^{\eps,\eta} + \big( U^\eps (t)   - U (t)\big)\bar g^\eta_{{\rm in}} +  \big(\Psi^\eps (t) -  \Psi(t) \big)  (g^\eta  ,g^\eta  )
\\
 &  \quad + 2 \Psi^\eps (t) \Big({g^\eta+\frac 12 \overline\delta^{\eps,\eta}}- \delta^{\eps,\eta}  , \overline\delta^{\eps,\eta}\Big) +\Psi^\eps (t) 
 \big( \delta^{\eps,\eta}  -2g^\eta ,  \delta^{\eps,\eta} \big) \\
 \mathcal{L}^\e _\lambda (t) h& :=  2\Psi_\lambda^{\e}(t)  ( \overline g^{\eps,\eta}-  \delta^{\eps,\eta}   ,h  )  \quad\mbox{with}  \\
  \Psi_\lambda^{\e }(t) (h_1 ,h_2) &:=\frac1\e 
  \int_0^t e^{- \lambda \int_{t'}^t \|\overline g^{\eps,\eta}(t'')\|_{\ell ,k}^r\, dt''
 } U^\eps(t-t')\Gamma (h_1 ,h_2) (t')\, dt'  \quad  \mbox{and}  
\\
  \Phi_\lambda^{\e }(t) (h_1,h_2) &:=\frac1\e e^{\lambda \int_0^t \|\overline g^{\eps,\eta}(t')\|_{\ell,k}^r\, dt' }
  \int_0^t e^{-2\lambda \int_{t'}^t \|\overline g^{\eps,\eta}(t'')\|_{\ell,k}^r\, dt''
 }    \\
&\qquad\qquad\qquad\qquad\qquad\qquad\times  U^\eps(t-t')\Gamma (h_1,h_2) (t')\, dt'
   \, .
 \end{aligned}
 \end{equation}
 Then to prove Theorem~{\rm\ref{mainthmwp}}, it is enough to prove the following convergence results: In the well-prepared case, for~$\lambda$ large enough
 $$
 \lim_{\eta\to 0} \lim_{\eps\to 0} \,   \big\|
h^{\eps,\eta}_\lambda
\big\|_{L^\infty([0,T],X^{\ell,k})} = 0   $$
 and in the ill-prepared case for all~$p>2/(d-1)$ and for~$\lambda$ large enough
 $$
 \lim_{\eta\to 0} \lim_{\eps\to 0} \,   \big\| h^{\eps,\eta}_\lambda
\big\|_{L^\infty([0,T],X^{\ell,k})+ {L^p(\R^+,L^{\infty,k}_v(W^{\ell,\infty}{+ H^\ell_x})(\R^d )))}} = 0  \, ,
 $$
 where the convergence is uniform in~$T$ if~$\bar g_{\rm{in}}^\eta$ gives rise to a global  unique solution.
 \end{lem}
\begin{proof}
Let us set, with notation~(\ref{eqgeta}) and~(\ref{defgepseta}),
$$
\widetilde h^{\e,\eta} := g^{\eps,\eta} - g^\eta
$$
which satisfies the following system in integral form, due to~(\ref{eqgeta}) and~(\ref{eqgepseta}) 
\begin{equation}\label{integralformtildeheps}
\widetilde h^{\e,\eta}(t) =  \mathcal{\widetilde D}^{\eps}(t)  +  \mathcal{\widetilde L}^{\e}(t) \widetilde h^{\e,\eta}+ \Psi^\e(t)  (\widetilde h^{\e,\eta}  ,\widetilde h^{\e,\eta}   ) 
\end{equation}
where
$$
\begin{aligned}
  \mathcal{\widetilde D}^{\eps}(t)  
 & :=U^\eps (t)  (g^\eta_{{\rm in}} - \bar g^\eta_{{\rm in}} )  +   \big( U^\eps (t)   - U (t)\big)\bar g^\eta_{{\rm in}}  +\Psi^\eps (t) 
  (  \delta^{\eps,\eta}  ,    \delta^{\eps,\eta}   ) 
\\&\quad\quad {- 2 \Psi^\eps(t) (g^\eta, \delta^{\eps,\eta} )} + \big(\Psi^\eps (t) -  \Psi(t) \big)  (g^\eta  ,g^\eta  )\\
  \widetilde {\mathcal{L}}^{\e}(t) h& := 2 \Psi^\e(t)  (g^\eta-   \delta^{\eps,\eta}   ,h  )   \, .
 \end{aligned}
 $$
{The conclusion of  Theorem~{\rm\ref{mainthmwp}} will be deduced from the fact that}~$\widetilde h^{\e,\eta} 
$ converges to zero in~$L^\infty([0,T],X^{\ell,k})$ (resp. in the space~$L^\infty([0,T],X^{\ell,k})+ {L^p(\R^+,L^{\infty,k}_v(W^{\ell,\infty}{+ H^\ell_x})(\R^d ))}$ in the well-prepared case (resp. in the ill-prepared case).

{In order to apply Lemma~\ref{cacciopoli+}, we would need   the linear operator~$   \widetilde {\mathcal{L}}^{\e}$ appearing in~(\ref{integralformtildeheps}) to be a contraction in~$L^\infty([0,T],X^{\ell,k})$, and 
the term~$\mathcal{\widetilde D}^{\eps}(t) $ to be small in~$L^\infty([0,T],X^{\ell,k})$. It turns out that in the $\R^2$-case, to reach this goal, we have to introduce a weight in time (note that in the references mentioned above in this context, only the three-dimensional case is treated, in which case it is not necessary to introduce that weight).
We thus introduce a function $\chi_\Omega(t)$ defined by
$$
\forall \, t \in \R^+, \quad \chi_\Omega(t) :=
 \left\lbrace \begin{array}{lcl}
1 & \text{if} &\Omega = \T^d, \, d = 2,3, \text{ or } \R^3, \\
\langle t \rangle^\frac14 &\text{if} & \Omega = \R^2 \, .
\end{array} \right. 
$$
For a given~$T>0 $ we define the associate weighted in time space
$$
\mathcal{X}_T^{\ell,k} := \Big\{f = f(t,x,v) \, / \,  f\in L^\infty(\mathds{1}_{[0,T]}(t)\chi_\Omega(t),X^{\ell,k})
\Big\}
$$
 endowed with the norm
 $$
 \|f\|_{\mathcal{X}_T^{\ell,k} }:= \sup_{t\in [0,T]}\chi_\Omega(t)\|f(t)\|_{\ell,k}\,.
 $$
In order to apply Lemma~\ref{cacciopoli+}, we   then need the term~$\mathcal{\widetilde D}^{\eps}(t) $ to be small in~$\mathcal{X}^{\ell,k}_T$.
Concerning this fact, it turns out that the first term appearing in~$  \mathcal{\widetilde D}^{\eps}(t) $  namely~$U^\eps (t)  (g^\eta_{{\rm in}} - \bar g^\eta_{{\rm in}} )
$, which is small (in fact zero) in the well-prepared case since~$g^\eta_{\rm in} = \bar g^\eta_{\rm in} $, contains in the case of ill-prepared data, a part which 
is not small in~${\mathcal X}_T^{{\ell},k}$ but in a different space: that is
$$
\overline\delta^{\eps,\eta}(t)=U^\eps_{\rm{disp}}(t) g^\eta_{{\rm in}} +
 U^{\eps \sharp}(t)g^\eta_{{\rm in}}\,.
 $$
 This is stated (among other estimates on $\bar \delta^{\eps,\eta}$) in the following lemma, which is proved in Section~\ref{proofstrichartz}.}
\begin{lem}\label{resultstildegeps}
Let $p \in (1, \infty]$. There exist a constant~$C$ such that for all $\eta \in (0,1)$ and all~$\eps \in (0,1)$,  
\begin{equation}\label{estimateLpXtildeg}
\|\overline\delta^{\eps,\eta}\|_{L^p(\R^+,{X}^{\ell,k})}
\leq C\,.
\end{equation}
Moreover there is a constant~$C$ such that for all $\eta \in (0,1)$ and all~$\eps \in (0,1)$
\begin{equation}\label{estimatedelta2}
\| U^{\eps \sharp}(t)g^\eta_{{\rm in}} 
\|_{\ell,k} \leq C e^{-\alpha {t \over \eps^2}} 
\end{equation}
where $\alpha$ is the rate of decay defined in~\eqref{rateUsharp}, and    for all $\eta \in (0,1)$ there   is a constant~$C_\eta$ such that for all~$\eps \in (0,1)$
\begin{equation}\label{estimatedelta1}
\|U^\eps_{\rm{disp}}(t) g^\eta_{{\rm in}} 
 \|_{L^{\infty,k}_v W^{\ell,\infty}_x} \leq C_\eta\Big( 1 \wedge \left(\eps \over t \right)^{\frac{d-1}{2}} \Big)
 \quad \text{and} \quad 
 {\|U^\eps_{\rm{disp}}(t) g^\eta_{{\rm in}} \|_{\ell,k} \le \frac{C_\eta}{\langle t \rangle^\frac d4}} \, \cdotp
 \end{equation}
 In particular~$\overline\delta^{\eps,\eta}$ satisfies
   for all $\eta \in (0,1)$
   $$
{\lim_{\e\to0} \|\overline\delta^{\eps,\eta}\|_{\mathcal{X}^{\ell,k}_\infty} \le C_\eta}
 \quad \text{and} \quad 
 \lim_{\e\to0}\|\overline\delta^{\eps,\eta}\|_{L^p(\R^+,L^{\infty,k}_v(W^{\ell,\infty}_x{+ H^\ell_x})(\R^d ))} = 0\, , \quad \forall \, p \in ({2}/{(d-1)} , \infty)  \, .
$$
\end{lem}
Returning to the proof of Lemma~\ref{lem:reduction}, let us set
$$
  h^{\e,\eta} :=  \widetilde h^{\e,\eta}  - \overline\delta^{\eps,\eta}  \,  , \quad \overline g^{\eps,\eta}:= g^{ \eta} +\overline \delta^{\eps,\eta}\, , 
  $$
and notice that~$  h^{\e,\eta}$ satisfies the following system in integral form
\begin{equation}\label{integralformheps}
  h^{\e,\eta}(t) = \mathcal{D}^{\e}(t)  +  \mathcal{L}^{\e}(t)   h^{\e,\eta}+ \Psi^\e(t)  (  h^{\e,\eta}  ,  h^{\e,\eta}   ) 
\end{equation}
with
$$
\begin{aligned}
 \mathcal{D}^\e  (t) 
 & :=  \widetilde \delta^{\eps,\eta} + \big( U^\eps (t)   - U (t)\big)\bar g^\eta_{{\rm in}} +  \big(\Psi^\eps (t) -  \Psi(t) \big)  (g^\eta  ,g^\eta  )
\\
 &  \quad + 2 \Psi^\eps (t) \Big({g^\eta+\frac 12 \overline\delta^{\eps,\eta}}- \delta^{\eps,\eta}  , \overline\delta^{\eps,\eta}\Big) +\Psi^\eps (t) 
 \big( \delta^{\eps,\eta}  -2g^\eta ,  \delta^{\eps,\eta} \big) \\
 \mathcal{L}^\e   (t) h& :=  2\Psi ^{\e}(t)  ( \overline g^{\eps,\eta}-  \delta^{\eps,\eta}   ,h  )  \quad\mbox{with}  \quad 
  \Psi ^{\e }(t) (h_1 ,h_2)  :=\frac1\e 
  \int_0^t  U^\eps(t-t')\Gamma (h_1 ,h_2) (t')\, dt'      \, .
 \end{aligned}
$$
In order to apply Lemma~\ref{cacciopoli+}, we need~$\mathcal{L}^{\eps}$ to be  a contraction, so we introduce
a modified space, in the spirit of~\cite{cg}, in the following way.
Since $
g^\eta$ and~$\overline\delta^{\eps,\eta}$ belong  to~$L^\infty([0,T],X^{\ell,k})$, then for all~$2\leq r \leq\infty$,  there holds
\begin{equation}\label{ginLrXellT}
\overline g^{\eps,\eta}:=  \overline\delta^{\eps,\eta} +g^\eta\in L^r([0,T],X^{\ell,k})
\end{equation}
with a norm depending on~$T$.  Moreover as recalled in  Proposition~\ref{propNSF}, if the unique solution to~(\ref{NSF}) is global in time then in particular
\begin{equation}\label{ginLrXell}
g^\eta \in L^r(\R^+,X^{\ell,k}) \, , \quad \forall \, r>4\, .
\end{equation}
So  thanks to~(\ref{estimateLpXtildeg}) we can fix~$r \in (4,\infty)$  from now on and define for  all~$\lambda> 0$
$$
h^{\e,\eta}_\lambda (t):=h^{\e,\eta}(t) \exp \Big(
-\lambda \int_0^t \|\overline g^{\eps,\eta}(t')\|_{\ell ,k}^r\, dt'
\Big)\, .
 $$
 The quantity
appearing in the exponential is finite thanks to~(\ref{estimateLpXtildeg})
 and~(\ref{ginLrXell}).
 The parameter~$\lambda>0$ will be fixed, and tuned later for~$   \mathcal{L}^{\e}$ to become a contraction. 
  Then~$h^{\e,\eta}_\lambda$ satisfies
$$h^{\e,\eta}_\lambda(t) = \mathcal{D}_\lambda^{\e}(t)   +  \mathcal{L}_\lambda^{\e}(t) h^{\e,\eta}_\lambda(t) + \Phi^{\e}_\lambda(t)  (h^{\e,\eta}_\lambda ,h^{\e,\eta}_\lambda   ) 
$$
with the notation~(\ref{defopslambda}). This concludes the proof of the lemma.
\end{proof}

 \subsection{End of the proof of Theorem~\ref{mainthmwp}} 

 The following results, together with Lemma~\ref{cacciopoli+}, are the key to the proof of Theorem~{\rm\ref{mainthmwp}}. They will be proved in the next sections. 
  \begin{prop} \label{linear}
Under the assumptions of Theorem~{\rm\ref{mainthmwp}}, there is a constant~$C$ such that  for all~$T>0$,~$\eta>0$ 
 and~$\lambda>0$
 $$
\lim_{\e \to 0}\big \|    \mathcal{L}_\lambda^{\e}(t) h\big\|_{\mathcal{X}_T^{\ell,k}} \leq  C\Big(
\frac1{\lambda^\frac1r } + \eta\Big)\| h  \|_{\mathcal{X}_T^{\ell,k}}\,.
 $$
 \end{prop}
  \begin{prop}\label{nonlinear}
Under the assumptions of Theorem~{\rm\ref{mainthmwp}}, there is a constant~$C$ such that  for all~$T>0$,~$\eta>0$,~$\e>0$ 
 and~$\lambda\geq0$
   $$
\big \|  \Phi_\lambda^{\e}(t) (f_1,f_2)\big\|_{\mathcal{X}_T^{\ell,k} } \leq C
\exp \Big(   \lambda \int_0^T {\|\bar g^{\eps,\eta}(t)\|_{\ell,k}^r}\, dt \Big)
 \| f_1  \|_{\mathcal{X}_T^{\ell,k}} \| f_2  \|_{\mathcal{X}_T^{\ell,k}} \, .
 $$
 \end{prop}
 \begin{prop}\label{limitdatawp}
 Under the assumptions of Theorem~{\rm\ref{mainthmwp}},  there holds
uniformly in~$\lambda\geq 0$  (and uniformly in~$T$ if~$\bar g^\eta_{\rm{in}}$ gives rise to a global  unique solution)
 $$
\lim_{\eta\to 0} \lim_{\eps\to 0}   \big\| \mathcal{D}_\lambda^{\e}(t) \big\|_{\mathcal{X}_T^{\ell,k}}=0\, .
$$
 \end{prop}
Assuming those results to be true, let us   apply Lemma~\ref{cacciopoli+}   to Equation~(\ref{equationhepsetalambda}) and~$X ={\mathcal X}_T^{\ell,k}$, with~$x_{0} = \mathcal{D}_\lambda^{\e}$, ${\mathcal L}  =  {\mathcal L} _\lambda ^{\e}$ and~${\mathcal B}  =  \Phi_\lambda ^{\e}$. Proposition~\ref{nonlinear}, \eqref{estimateLpXtildeg} along with~(\ref{ginLrXell})
ensure that~$ \Phi ^{\e} _\lambda$ is a bounded bilinear operator over~$\mathcal{X}_T^{\ell,k} $, uniformly in~$T$ if~$\bar g_{\rm{in}}^\eta$ gives rise to a global  unique solution. 
Moreover choosing~$\lambda$ large enough,~$\e$ small enough (depending  on~$\eta$, and on~$T$ except  if~$\bar g^\eta_{\rm{in}}$ gives rise to a global  unique solution) and~$\eta$ small enough uniformly in the other parameters, Proposition~\ref{linear} ensures that~$   {\mathcal L}_\lambda ^{\e} $ is a contraction in~${\mathcal X}_T^{\ell,k}$. Finally thanks to Proposition~\ref{limitdatawp}    the assumption  of Lemma~\ref{cacciopoli+} on~$\mathcal{D}^{\e}_\lambda $ is  satisfied as soon as~$\e$ and~$\eta$ are small enough. There is therefore a unique   solution to~{\rm(\ref{equationhepsetalambda})} in~$\mathcal{X}_T^{\ell,k}$, which satisfies, uniformly in~$T$  if~$\bar g^\eta_{\rm{in}}$ gives rise to a global  unique solution,
\begin{equation}\label{shorttimeestimate}
\lim_{\eta\to 0} \lim_{\eps\to 0}  \,    \|
h_\lambda^{\e,\eta} 
 \|_{\mathcal{X}_T^{\ell,k}}  = 0 \, .
\end{equation}
Thanks to Lemma~\ref{lem:reduction}, this ends the proof of Theorem~\ref{mainthmwp}. \qed
    
    \bigskip
    
 To conclude it remains to prove Propositions~\ref{linear} to~\ref{limitdatawp}   as well as Lemma~\ref{resultstildegeps}. {Note that the proofs of Propositions~\ref{linear} to~\ref{limitdatawp} are conducted to obtain estimates uniform in $T$, and this information is actually only useful in the case of global solutions (which is, for example, always the case in dimension $2$). Note also that, here and in what follows, we have denoted by~$A \lesssim B$ if there exists a universal constant $C$ (in particular independent of the parameters~$T,\eps,\lambda,\eta$) such that~$A \leq CB$.
 
 Before going into the proofs of Propositions~\ref{linear} to~\ref{limitdatawp}, we are going to state lemmas about continuity properties of $U^\eps(t)$ and $\Psi^\eps(t)$ in the next section that are useful in the rest of the paper.}
 
\section{Estimates on $U^\eps(t)$ and $\Psi^\eps(t)$}
 Let us mention that some of the following results (Lemmas~\ref{lem:Uepscont},~\ref{lem:decayWeps} and~\ref{lem:Psiepscont}) have already been proved in some cases (see~\cite{Bardos-Ukai}) but for the sake of completeness, we write the main steps of the proof in this paper, especially because the $\R^2$-case is not always clearly treated in previous works.  The conclusions of the following lemmas hold for $\Omega=\T^d$ or $\R^d$ with $d=2,3$ unless otherwise specified. 

 \subsection{Estimates on $U^\eps(t)$}
\begin{lem} \label{lem:Uepscont}
 Let~$\ell \ge 0$ and~$k>d/2+1$ be given.  Then for all~$\eps>0$, the operator~$ U^\eps(t) $ is a strongly continuous semigroup on~$X^{\ell,k}$ and there is a constant~$C$ such that for all~$\eps \in (0,1)$ and all~$t \geq 0$
\begin{equation} \label{boundUeps}
 \|U^\eps(t) f \|_{\ell,k} \leq C \| f\|_{\ell,k} \, , \quad \forall \, f \in X^{\ell,k} \, .
\end{equation}
 \end{lem}
 \begin{proof}
 For the generation of the semigroup, we refer for example to~\cite{Ukai,Glassey}. Concerning the estimate on $U^\eps(t)$, following Grad's decomposition~\cite{Grad}, we start by spliting the operator $L$ defined in~(\ref{defoperator}) as 
$$
L h= -\nu(v) h + Kh \, , 
$$
where
the collision frequency $\nu$ is defined through
\begin{equation}\label{defnu(v)}
\nu(v) := \int_{\R^d \times \Sp^{d-1} } |v-v_*|M(v_*)\, d\sigma \, dv_*
\end{equation}
and satisfies for some constants~$0<\nu_0<\nu_1$,
$$
 \quad \nu_0(1+|v|) \leq \nu (v)  \leq \nu_1(1+|v|) \, .
$$
The operator $K$ is then defined as
$$
Kh := Lh + \nu (v) h \, ,
$$
it is bounded from~$H^\ell_x L^2_v$ to~$X^{\ell,0}$ and from~$X^{\ell,j}$ to~$X^{\ell,j+1}$ for any~$j \ge 0$ (see~\cite{Ukai}). Then, denoting 
$$
A^\e:=-{1 \over \e^2} \left(\e v \cdot \nabla_x + \nu(v)\right) \quad \text{and} \quad B^\e:=A^\e + {1\over \e^2} K\,,
$$
we use the Duhamel formula to decompose $U^\e(t)$ as follows:
\begin{equation} \label{eq:decompUeps}
U^\e(t)=e^{tA^\e} + \int_0^t e^{(t-t')A^\e} {1 \over \e^2} K U^\e(t') \, dt'\,. 
\end{equation}
Moreover, the semigroup $e^{tA^\e}$ is explicitly given by 
\begin{equation} \label{semigroupAeps}
e^{tA^\e} h = e^{-\nu(v)\frac{t}{\e^2}} h\Big(x-v{t\over\e},v\Big)\,,
\end{equation}
so   $e^{tA^\e}$ satisfies
$$
\|e^{tA^\e} h\|_X \le e^{-\nu_0\frac{t}{\e^2}} \|h\|_X
$$
for $X=H^\ell_x  L^2_v$ or $X=X^{\ell,j}$ for $j \ge 0$. From this and the fact that $K$ is bounded from~$H^\ell_x L^2_v$ to~$X^{\ell,0}$ and from~$X^{\ell,j}$ to~$X^{\ell,j+1}$ for any $j \ge 0$, we deduce that there exists a constant $C$ such that
\begin{align*}
\|U^\e(t)f\|_{X} &\le e^{-\nu_0{\frac{t}{\eps^2}}} \|f\|_X + {C \over \eps^2}\int_0^t e^{-\nu_0 \frac{t-t'}{\eps^2}} \|U^\e(t') f\|_Y \, dt'
\end{align*}
and thus
$$
\|U^\e(t)f\|_{L^\infty(\R^+,X)}\le \|f\|_X + C \|U^\e(t) f\|_{L^\infty(\R^+,Y)} 
$$
for $(X,Y)=(X^{\ell,0},H^\ell_xL^2_v)$ or $(X,Y)=(X^{\ell,j},X^{\ell,j-1})$ for any $j \ge 1$. Reiterating the process, we obtain that 
\begin{equation} \label{reiter}
\|U^\e(t)f\|_{L^\infty(\R^+,X^{\ell,k})}\lesssim \|f\|_{\ell,k} + \|U^\eps(t) f \|_{L^\infty(\R^+,H^\ell_xL^2_v)} \, .
\end{equation}
It now remains to estimate $U^\eps(t)f$ in $H^\ell_xL^2_v$. Taking the Fourier transform in~$x$ we have   for all~$\xi $ thanks to~(\ref{decompUeps}) in Lemma~\ref{EllisPinsky} 
\begin{equation}\label{decUepsEP}
\begin{aligned}
  U^\e(t ) &= \sum_{j=1}^4    U_j^\e(t ) + U^{\e\sharp}( t) \quad \text{with} \\
 \widehat U_j^\e(t,\xi) &:= \widehat U_j\left({t \over \e^2},\e \xi\right)  \quad  \text{and} \quad\widehat U^{\e\sharp}(t,\xi) := \widehat U^\sharp\left({t \over \e^2},\e \xi\right)  
\end{aligned}
\end{equation}
and for $1\le j \le 4$
$$
\widehat U^\e_j(t,\xi) = \chi\Big(\frac{\e |\xi|}\kappa\Big) e^{t \mu_j^\e(\xi)} P_j(\e \xi)
$$
with 
$$
\mu_j^\e(\xi):={1 \over \e^2} \lambda_j(\e \xi) = i \alpha_j {|\xi| \over \e} - |\xi|^2\big(\beta_j + O(\e |\xi|)\big) \, .
$$
Denoting $\beta := \min_j \beta_j/2$ and recalling that $\alpha$ is the rate of decay defined in~\eqref{rateUsharp}, we obtain the following bound:
$$
\|\widehat U^\eps(t,\xi)\|_{L^2_v \to L^2_v}  \lesssim e^{-\beta |\xi|^2 t} + e^{-\alpha \frac{t}{\eps^2}} \, .
$$
From this and using that $k >d/2$, we deduce that 
$$
 \|U^\eps(t) f \|_{L^\infty(\R^+,H^\ell_xL^2_v)} \lesssim \|f\|_{H^\ell_xL^2_v} \lesssim \|f\|_{\ell,k}\, ,
$$
which allows us to conclude the proof thanks to~\eqref{reiter}.
 \end{proof}
We now state a lemma which provides decay estimates on $\frac1\eps U^\eps(t)$ on the orthogonal of~$\operatorname{Ker} L$. 
 \begin{lem} \label{lem:decayWeps}
 Let $\ell\ge 0$. We denote $
W^\e (t):=\frac 1\e U^\e (t)(I-\Pi_{L}  ).
$
We then have the following estimates: there exists $\sigma>0$ such that 
$$
\|W^\e(t)f\|_{H^\ell_xL^2_v} \lesssim 
\left\{
\begin{array}{lrl}
\frac{e^{-\sigma t}}{t^{\frac12}} \|f\|_{H^\ell_xL^2_v} &\text{if} &\Omega = \T^d \,, \vspace{0.1cm}\\
 {1 \over t^{\frac12}} \|f\|_{H^\ell_xL^2_v}  &\text{if} &\Omega = \R^d\, , \vspace{0.1cm} \\
{1 \over {t^{\frac12} \langle t \rangle^{d \over 4}}} (\|f\|_{H^\ell_xL^2_v} + \|f\|_{L^2_vL^1_x})  &\text{if} &\Omega = \R^d \, .
\end{array}
\right.
$$
 \end{lem}
 
 \begin{proof}
 We use again~(\ref{decUepsEP}) and we recall that $$
P_j(\e \xi)(I-\Pi_{L}  )= \e|\xi| \left(P^1_{j}\Big({\xi \over |\xi|}\right) + \e |\xi| P^2_{j}(\e |\xi|)\Big)\,. 
$$
Using results from Lemma~\ref{EllisPinsky} on $P^1_j$ and $P^2_j$, denoting $\beta := \min_j \beta_j/2$, we obtain the following bound:
\begin{equation} \label{finalboundWeps}
\|\widehat W^\e(t,\xi)\|_{L^2_v \to L^2_v} \lesssim |\xi| e^{-\beta |\xi|^2t} + {1 \over \e} e^{-\alpha \frac{t}{\e^2}}
\end{equation}
where $\alpha$ is the rate of decay defined in~\eqref{rateUsharp}. 
From this we shall deduce a bound in~$H^\ell_xL^2_v$, arguing differently according to the definition of~$\Omega$.
We first notice that for any $t \ge 0$,
\begin{equation} \label{texp}
{1 \over \e} e^{-\alpha {t\over \e^2}} \lesssim {e^{-\alpha {t\over2}}\over t^\frac12} \, \cdotp
\end{equation}

\noindent $\bullet $ $ $  {\it The case of $\T^d$.}$ $  
Since~$\xi \in \Z^d$, we have 
$$
|\xi| e^{-\beta |\xi|^2t} \lesssim \frac{e^{-\beta {t\over 2}} }{t^\frac12}\, \cdotp
$$
We can thus deduce from~\eqref{finalboundWeps} that for $\sigma:=\min(\alpha,\beta)/2>0$, and for any~$f = f(v)\in L^2_v$,
$$
\|\widehat W^\e(t,\xi)f\|_{L^2_v} \lesssim {e^{- \sigma t} \over t^\frac12} \|f\|_{L^2_v} \, . 
$$
It is then clear that for any~$f = f(x,v)\in H^\ell_xL^2_v$,
$$
\|W^\e(t) f\|_{H^\ell_xL^2_v} \lesssim {e^{-\sigma t} \over t^\frac12} \|f\|_{H^\ell_xL^2_v}\, .
$$
\noindent $\bullet $ $ $ {\it The case of $\R^d$.} $ $Note that
\begin{equation} \label{xiexp}
\forall \, t>0 \, , \quad |\xi| e^{-\beta |\xi|^2t} \lesssim {e^{-\beta |\xi|^2{t \over 2}} \over t^\frac12}.
\end{equation}
This together with~(\ref{finalboundWeps}) and~(\ref{texp})
gives  directly that
$$
\|W^\e(t) f\|_{H^\ell_xL^2_v}   \lesssim {1 \over t^\frac12 } \, \|f\|_{H^\ell_xL^2_v}\,.
$$
Finally let us prove the last estimate. We can suppose that~$t \gtrsim 1$. Then
using~\eqref{finalboundWeps} and~\eqref{texp}, we write that for any function~$f$ 
\begin{align*}
&\|W^\e(t) f\|^2_{H^\ell_xL^2_v} \\
&\quad \lesssim \int_{\R^d} \Big(|\xi|^2 e^{-2\beta t |\xi|^2} + {e^{-\alpha t} \over t}\Big)(1+|\xi|^{2\ell}) \|\widehat f (\xi, \cdot)\|^2_{L^2_v} \, d\xi \\
&\quad \lesssim \int_{\R^d} |\xi|^2 e^{-2\beta t |\xi|^2}\|\widehat f (\xi, \cdot)\|^2_{L^2_v} \, d\xi + \int_{\R^d} |\xi|^{2+2\ell} e^{-2\beta t |\xi|^2}\|\widehat f (\xi, \cdot)\|^2_{L^2_v} \, d\xi + \frac{e^{-\alpha t}}{t} \|f\|^2_{H^\ell_xL^2_v}\\
&=: I_1+I_2+I_3\,. 
\end{align*}
We treat $I_1$ using~\eqref{xiexp} and a change of variable: since~$t \gtrsim 1$ then
\begin{align*}
I_1 &\lesssim {1 \over t} \int_{\R^d} e^{-\beta t |\xi|^2} \, d\xi \, \|\widehat f \|^2_{L^2_v L^\infty_\xi} \\
&\lesssim {1 \over t^{1+\frac d2}} \, \|f\|^2_{L^2_vL^1_x} \lesssim {1 \over {t\langle t \rangle^{\frac d2}}} \, \|f\|^2_{L^2_vL^1_x}\,.
\end{align*}
The term $I_2$ is handled just by using a change of variable and we obtain (since~$t \gtrsim 1$)
$$
I_2 \lesssim  {1 \over t^{1+\ell+\frac d2}} \, \|f\|^2_{L^2_vL^1_x} \lesssim {1 \over {t\langle t \rangle^{\frac d2}}} \, \|f\|^2_{L^2_vL^1_x}\,.
$$
The decay in time of $I_3$ is even better, so in the end, we get that for any $t \ge 0$, there holds
$$
\|W^\e(t) f\|_{H^\ell_xL^2_v} \lesssim  {1 \over {t^\frac12\langle t \rangle^{\frac d4}}} \, \left(\|f\|_{H^\ell_xL^2_v} +\|f\|_{L^2_vL^1_x}\right)\,.
$$
Lemma~\ref{lem:decayWeps}
is proved.
\end{proof}

We now give some estimates on the different parts of $U^\eps(t)$ from the decomposition given in~(\ref{decompUeps}). 
\begin{lem} \label{lem:Uepsdisp}
Let $\Omega = \R^2$. Fix~$\ell \ge0$ and $k>2$  and consider~$f$ in~$X^{\ell,k}   \cap L^{2}_v L^1_x$. Then with the notation introduced in Remark~{\rm\ref{defU}} there holds for all $\eps \in (0,1)$
$$
\sup_{t \ge 0} \bigg( \langle t \rangle^{\frac12} \big\|\big(U^\eps_{\rm{disp}}(t)  +U(t) + U^{\eps \sharp}(t)\big) f\big \|_{\ell,k}\bigg) \lesssim \|f\|_{\ell,k} + \|f\|_{L^{2}_v L^1_x} \, .
$$
\end{lem}
\begin{Rmk}
We only need $f$ to be in $L^{2}_v L^1_x$ in order to estimate the terms $U^\eps_{\rm{disp}}(t)$ and~$U(t)$. Indeed, the decay in time of those terms comes from the decay of the heat flow and thus requires a loss of integrability in space. 
\end{Rmk}

\begin{proof}
Let us start with the terms $(U^\eps_{\rm disp}(t)+U(t))f$. We focus on large times $t \gtrsim 1$, the case of small times $t \lesssim 1$ can be treated in an easier way just using the continuity of the heat flow in $H^\ell_x$. We remark that given the form of $U^\eps_{\rm disp}(t)$ and~$U(t)$, we just need to estimate 
\begin{align*}
&\Big\| \mathcal{F}_x^{-1} \big(e^{-\beta_j t|\xi|^2} e^{i \alpha_j t \frac{|\xi|}{\eps}}  P^0_j{\Big({\xi \over |\xi|} \Big)} \widehat f \big) \Big\|_{\ell,k} \\
&\quad = \Big\| \Big\| e^{\beta_j t \Delta} \mathcal{F}_x^{-1} \Big(e^{i \alpha_j t \frac{|\xi|}{\eps}}  P^0_j\Big({\xi \over |\xi|} \Big) \widehat f \Big)\Big\|_{H^\ell_x} \Big\|_{L^{\infty,k}_v} 
\end{align*}
with $\alpha_j \in \R$  (and can be~0)   and $\beta_j>0$. We have
\begin{align*}
&\Big\| \mathcal{F}_x^{-1} \big(e^{-\beta_j t|\xi|^2} e^{i \alpha_j t \frac{|\xi|}{\eps}}  P^0_j\left(\xi \over |\xi| \right) \widehat f \big) \Big\|^2_{\ell,k}  \\
&\quad \lesssim \sup_v \langle v \rangle^{2k} \int_{\R^2} \langle \xi \rangle^{2\ell}  e^{-\beta_j t|\xi|^2}  \Big|  e^{i \alpha_j t \frac{|\xi|}{\eps}}  P^0_j\left(\xi \over |\xi| \right) \widehat f \Big|^2 \, d\xi \\
&\quad \lesssim \int_{\R^2} \langle \xi \rangle^{2\ell}e^{-\beta_j t|\xi|^2}  \Big\|P^0_j\Big({\xi \over |\xi|} \Big) \widehat f\Big\|^2_{L^{\infty,k}_v} \, d\xi \, .
\end{align*}
Then, using that $P_{j}^0(\xi/|\xi|)$ is bounded from $L^2_v$ into $L^{\infty,k}_v$ uniformly in $\xi$ from Lemma~\ref{EllisPinsky} and the fact that $L^{\infty,k}_v \hookrightarrow L^2_v$, we obtain:
\begin{align*}
&\left\| \mathcal{F}_x^{-1} \left(e^{-\beta_j t|\xi|^2} e^{i \alpha_j t \frac{|\xi|}{\eps}}  P^0_j\Big({\xi \over |\xi|} \Big) \widehat f \right) \right\|_{\ell,k}^2 \\
&\quad \lesssim  \int_{\R^2} \langle \xi \rangle^{2\ell} e^{-\beta_j t|\xi|^2} \| \widehat f \|^2_{L^2_v} \, d\xi \lesssim   \int_{\R^2}
\int_{\R^2} (1+|\xi|^{2\ell}) e^{-\beta_j t|\xi|^2}|\widehat f|^2\, d\xi \, dv \, .
\end{align*}
Using now the decay properties of the heat flow and bounding~$ \|f\| _{\ell-1,k} $ by~$ \|f\| _{\ell,k}$, we get:
$$
\left\| \mathcal{F}_x^{-1} \left(e^{-\beta_j t|\xi|^2} e^{i \alpha_j t \frac{|\xi|}{\eps}}  P^0_j\Big({\xi \over |\xi| }\Big) \widehat f \right) \right\|_{\ell,k}^2 \lesssim {1 \over t} \big(\|f\|^2_{L^2_v L^1_x} + \|f\|^2_{\ell,k}\big) \, .
$$
Let us now estimate the last remainder term $\|U^{\eps \sharp}(t) f\|_{\ell,k}$. From~\eqref{rateUsharp}, one can prove (see the proof of Lemma~6.2 in~\cite{Bardos-Ukai}) that 
\begin{equation} \label{decayUepssharp}
\|U^{\eps \sharp}(t) f\|_{\ell,k} \lesssim e^{-\alpha {t\over \eps^2}} \|f\|_{\ell,k}  
\end{equation}
and thus 
$$
\langle t \rangle^\frac12 \|U^{\eps \sharp}(t)f\|_{\ell,k}\lesssim \|f\|_{\ell,k}  \quad \forall \, t \lesssim 1 \, .
$$
We then notice that for $t \gtrsim 1$,
$$
t ^\frac12e^{-\alpha {t\over\eps^2}} \lesssim \eps \frac{t^\frac12}{\eps} e^{-\alpha {t\over\eps^2}} \lesssim \eps
$$
to deduce that 
$$
t^\frac12 \|U^{\eps \sharp}(t)f\|_{\ell,k}\lesssim \eps \|f\|_{\ell,k}  \quad \forall \, t \gtrsim 1 \, .
$$
Lemma~\ref{lem:Uepsdisp} is proved.
\end{proof}

\begin{lem}\label{lem:UepstoU0}
   Fix~$\ell \ge0$, $k>d/2+1$  and consider~{$f$ in~$X^{\ell,k}$}. 
   Then with the notation introduced in Remark~{\rm\ref{defU}} there holds
   \begin{equation} \label{boundUeps-U0ip}
\sup_{t \ge 0} \, \bigg(\langle t \rangle^{\frac12} \big\|  \big( U^\eps (t)   - U^\eps_{\rm{disp}}(t)  - U^{\eps \sharp}(t)- U (t)\big)  {f} \big \|_{\ell,k}\bigg) \lesssim \|f\|_{\ell,k} \, .
\end{equation}
If moreover $f \in X^{\ell+1,k}$, there holds:
\begin{equation} \label{estimatesmoothdata}
\sup_{t \ge 0} \, \bigg(\langle t \rangle^{\frac12} \big\|  \big( U^\eps (t)   - U^\eps_{\rm{disp}}(t)  - U^{\eps \sharp}(t)- U (t)\big)  {f} \big \|_{\ell,k}\bigg) \lesssim \eps \|f\|_{\ell+1,k} \, ,
\end{equation}
and if~{$f \in X^{\ell+1,k}$} is a well-prepared data in the sense of~{\rm(\ref{defwp})}, 
then
\begin{equation} \label{boundUeps-U0}
\lim_{\eps\to0}\sup_{t \ge 0} \, \bigg(\langle t \rangle^{\frac12} \big \|  \big( U^\eps (t)   - U (t)\big) {f}  \big\|_{{\ell,k}} \bigg)= 0\,.
\end{equation}
\end{lem}
\begin{proof}[Proof of Lemma~{\rm\ref{lem:UepstoU0}}]
Wes shall prove simultaneously estimates~\eqref{boundUeps-U0ip} and~\eqref{estimatesmoothdata}. 
Using the notation introduced in Appendix~\ref{specBoltz}, we   consider $1 \le j \le 4$ and we want to estimate the terms~$\langle t \rangle^{\frac12}\|U^\eps_{jm}(t)f\|_{\ell,k}$ for~$1 \le m \le 2$ as well as~$\langle t \rangle^{\frac 12}\|U^{\eps\sharp}_{j0}(t)f\|_{\ell,k}$.  We restrict ourselves to the case~$\Omega = \R^d$, the case of the torus can be treated similarly.    
We start with~$U^\eps_{j1}(t)$. We first consider small times~$t \lesssim 1$. We have 
\begin{align*}
&\|U^\eps_{j1} (t)f\|_{\ell,k}^2 \\
&\quad = \Big \| \int_{\R^d}  \langle \xi \rangle^{2\ell} \chi^2\Big(\frac{\eps|\xi|}{\kappa}\Big) e^{- 2\beta_j t |\xi|^2} \left|e^{t \frac{\gamma_j(\eps|\xi|)}{\eps^2}}-1 \right|^2\Big|P_{j}^0\left(\xi\over|\xi|\right)\widehat{f}(\xi,\cdot)\Big|^2 \, d\xi\Big \|_{L^{\infty,k}_v} \\
&\quad \lesssim \int_{\R^d}  \chi^2\Big(\frac{\eps|\xi|}{\kappa}\Big) e^{- 2\beta_j t |\xi|^2} \Big|e^{t \frac{\gamma_j(\eps|\xi|)}{\eps^2}}-1 \Big|^2\Big\|P_{j}^0\Big({\xi\over|\xi|}\Big)\langle \xi \rangle^\ell \widehat{f}(\xi,\cdot)\Big\|^2_{L^{\infty,k}_v}  \, d\xi\\
&\quad \lesssim \int_{\R^d}  \chi^2\Big(\frac{\eps|\xi|}{\kappa}\Big) e^{- 2\beta_j t |\xi|^2} \Big|e^{t \frac{\gamma_j(\eps|\xi|)}{\eps^2}}-1 \Big|^2\big\|\langle \xi \rangle^\ell \widehat{f}(\xi,\cdot)\big\|^2_{L^{2}_v} \, d\xi
\end{align*}
where we used, as in the proof of Lemma~\ref{lem:Uepsdisp}, the fact that $P_{j}^0(\xi/|\xi|)$ is bounded from~$L^2_v$ into~$L^{\infty,k}_v$ uniformly in $\xi$ to get the last inequality. Using~\eqref{estimatelambdaj} in Lemma~\ref{EllisPinsky} and the inequality $|e^a-1| \le |a| e^{|a|}$ for any~$a \in \R$, we now bound from above the term~$\Big|e^{t \frac{\gamma_j(\eps|\xi|)}{\eps^2}}-1 \Big|$:
\begin{equation}\label{estimategamma1}
\chi\Big (\frac{\eps|\xi|}\kappa\Big) e^{- \beta_j t |\xi|^2} \Big|e^{t \frac{\gamma_j(\eps|\xi|)}{\eps^2}}-1 \Big| \lesssim \chi\Big (\frac{\eps|\xi|}\kappa\Big) e^{-  \frac{\beta_j }2 t |\xi|^2}
 t   |\xi|^2 \lesssim 1\, ,
\end{equation}
and 
\begin{equation}\label{estimategamma2}
\chi\Big (\frac{\eps|\xi|}\kappa\Big) e^{- \beta_j t |\xi|^2} \left|e^{t \frac{\gamma_j(\eps|\xi|)}{\eps^2}}-1 \right| \lesssim \chi\Big (\frac{\eps|\xi|}\kappa\Big) e^{-  \frac{\beta_j }2 t |\xi|^2}
 t \eps  |\xi|^3 \lesssim \eps  |\xi| \, . 
\end{equation}
This gives, for any~$t\geq 0$,
$$
\|U^\eps_{j1}(t)f\|_{\ell,k} ^2\lesssim  \|f\|_{L^2_v H^{\ell}_x }^2 \quad \text{and} \quad \|U^\eps_{j1}(t)f\|_{\ell,k} ^2 \lesssim \eps^2 \|f\|^2_{L^2_vH^{\ell+1}_x} \, .
$$
Using that $ L^{\infty,k}_v \hookrightarrow L^2_v$, we get:
$$
\|U^\eps_{j1}(t)f\|_{\ell,k} \lesssim  \|f\|_{\ell,k} \quad \text{and} \quad \|U^\eps_{j1}(t)f\|_{\ell,k} \lesssim \eps \|f\|_{\ell+1,k} \,, \quad \forall \, t \lesssim 1\, .
$$
Now for large times $t \gtrsim 1$, we notice that using~\eqref{estimatelambdaj} in Lemma~\ref{EllisPinsky}, we can write
$$
t^{\frac12}\chi\Big (\frac{\eps|\xi|}\kappa\Big) e^{- \beta_j t |\xi|^2} \left|e^{t \gamma_j(\eps|\xi|)/\eps^2}-1 \right| \lesssim \chi\Big (\frac{\eps|\xi|}\kappa\Big) e^{-  \frac{\beta_j }2 t |\xi|^2}
 t^{\frac 32} \eps  |\xi|^3 \lesssim \eps
$$
so that as previously
\begin{align*}
t^{\frac 12} \|U^\eps_{j1}(t)f\|_{\ell,k} &\lesssim  \eps \|f\|_{\ell,k}\,, \quad \forall \, t \gtrsim 1\, .
\end{align*}
We are thus able to conclude that 
$$
\langle t \rangle^{\frac 12} \|U^\eps_{j1}(t)f\|_{{\ell,k}} \lesssim   \|f\|_{\ell,k} \quad \text{and} \quad \langle t \rangle^\frac12 \|U^\eps_{j1}(t)f\|_{\ell,k} \lesssim \eps \|f\|_{\ell+1,k}\, ,\quad \forall \, t \ge 0 \, . 
$$
For $\|U^\eps_{j2}(t)f\|_{\ell,k}$, we consider $t \lesssim 1$ and we write 
\begin{align*}
&\|U^\eps_{j2}(t)f\|_{\ell,k}^2 \\ 
&\quad =\Big\|  \int_{\R^d}  \langle \xi \rangle^{2\ell} \chi\Big (\frac{\eps|\xi|}\kappa\Big) e^{- 2\beta_j t |\xi|^2+ 2t \frac{\gamma_j(\eps|\xi|)}{\eps^2}} \eps^2 |\xi|^2 \Big|\tilde P_{j}\Big(\eps \xi,{\xi \over |\xi|}\Big)\widehat{f}(\xi,\cdot)\Big|^2 \, d\xi\Big\|_{L^{\infty,k}_v} \\
&\quad \lesssim \int_{\R^d}  \chi\Big (\frac{\eps|\xi|}\kappa\Big) e^{- 2\beta_j t |\xi|^2+ 2t \frac{\gamma_j(\eps|\xi|)}{\eps^2}} \eps^2 |\xi|^2 \Big\|\tilde P_{j}\Big(\eps \xi,{\xi \over |\xi|}\Big) \langle \xi \rangle^{\ell} \widehat{f}(\xi,\cdot)\Big\|_{L^{\infty,k}_v}^2 d\xi \, .
\end{align*}
In view of the definition of $\tilde P_{j}$ in~\eqref{deftildeP1j} and the fact that $P_{j}^1(\xi/|\xi|)$ and $P_{j}^2(\xi)$ are bounded from $L^2_v$ into $L^{\infty,k}_v$ uniformly in $|\xi| \le \kappa$ from Lemma~\ref{EllisPinsky}, we deduce that 
$$
\chi\Big (\frac{\eps|\xi|}\kappa\Big) \Big\|\tilde P_{j}\Big(\eps \xi,{\xi \over |\xi|}\Big) \langle \xi \rangle^{\ell}\widehat{f}(\xi,\cdot)\Big\|_{L^{\infty,k}_v} \lesssim \chi\Big (\frac{\eps|\xi|}\kappa\Big) \big\|  \langle \xi \rangle^{\ell}\widehat{f}(\xi,\cdot)\big\|_{L^2_v} \, . 
$$
Using again~\eqref{estimatelambdaj}, we have as long as~$\e|\xi|\leq \kappa$
$$
e^{- 2\beta_j t |\xi|^2+ 2t \frac{\gamma_j(\eps|\xi|)}{\eps^2}} \le e^{-\beta_j t |\xi|^2}\, . 
$$
Since $\chi(\e|\xi|/\kappa) \eps^2|\xi|^2 \lesssim 1$ and $\chi(\e|\xi|/\kappa) \eps^2|\xi|^2 \lesssim \eps^2|\xi|^2$, we can bound $\|U^\eps_{j2} (t)f\|_{\ell,k}$ as 
$$
\|U^\eps_{j2}(t)f\|_{\ell,k} \lesssim  \|f\|_{\ell,k}  \quad \text{and} \quad \|U^\eps_{j2}(t)f\|_{\ell,k} \lesssim  \eps \|f\|_{\ell+1,k} ,\quad \forall \, t \lesssim 1.
$$
For large times $t \gtrsim 1$, we have 
$$
t \|U^\eps_{j2}(t)f\|^2_{\ell,k} \lesssim
\eps^2  \int_{\R^d} t|\xi|^2 e^{-\beta_j t|\xi|^2} \big\|  \langle \xi \rangle^{\ell}\widehat{f}(\xi,\cdot)\big\|^2_{L^2_v} \, d\xi
$$
which implies that 
$$
t^\frac12 \|U^\eps_{j2}(t)f\|_{\ell,k}\lesssim \eps \|f\|_{\ell,k} \,, \quad \forall \, t \gtrsim 1 
$$
and thus
$$
\langle t \rangle^{\frac 12} \|U^\eps_{j2}(t)f\|_{{\ell,k}} \lesssim  \|f\|_{\ell,k}  \quad \text{and} \quad \langle t \rangle^\frac12 \|U^\eps_{j2}(t)f\|_{\ell,k} \lesssim  \eps \|f\|_{\ell+1,k} \, , \quad \forall \, t \ge 0 \, .  
$$
Finally, for $\|U^{\eps\sharp}_{j0} (t) f\|_{\ell,k}$, we proceed in the same way using the inequalities
\begin{equation}\label{estimatechi}
\Big|\chi\Big (\frac{\eps|\xi|}\kappa\Big)-1\Big| \lesssim 1 \quad \text{and} \quad \Big|\chi\Big (\frac{\eps|\xi|}\kappa\Big)-1\Big| \lesssim \eps |\xi|
\end{equation}
to get 
$$
\langle t \rangle^{\frac12} \|  U^{\eps\sharp}_{j0} (t)f\|_{{\ell,k}} \lesssim \|f\|_{\ell,k} \quad \text{and} \quad \langle t \rangle^{\frac12} \|  U^{\eps\sharp}_{j0} (t)f\|_{{\ell,k}} \lesssim\eps \|f\|_{\ell+1,k}\quad \forall \, t \ge 0 \,.
$$
This proves~(\ref{boundUeps-U0ip})-(\ref{estimatesmoothdata}).

\smallskip

\noindent
Let us now consider $f$ a well-prepared data. To prove~(\ref{boundUeps-U0}) we use the decomposition~(\ref{decompUeps}), and we notice on the one hand (see Remark~\ref{defU}) that
$$
U^\e_{30} = U_{30} \, , \quad 
U^\e_{40}=
U_{40}\quad \mbox{and}\quad U =  U_{30}+U_{40}$$
and on the other hand that from~\eqref{P012}, if $f$ is a well-prepared data, then
$$
U^\eps_{\rm{disp}}(t) f=0\, .
 $$
 This proves~(\ref{boundUeps-U0}), up to the
 fact that
 $$
\limsup_{\eps\to0} \sup_{t \ge 0} \big(\langle t \rangle^{\frac12} \|U^{\eps \sharp}(t) f\|_{\ell,k} \big)= 0 \, .
 $$
 We thus estimate this last remainder term $\|U^{\eps \sharp}(t) f\|_{\ell,k}$. The estimate for large times has already been obtained at the end of the proof of Lemma~\ref{lem:Uepsdisp}. 
For small times, in~\cite[Lemma~6.2]{Bardos-Ukai}, the authors   notice that 
$$
U^{\eps \sharp}(t )f = U^\eps(t)U^{\eps \sharp}(0 ) f= U^\eps(t) \Big[\mathcal{F}_x^{-1} \Big({\rm Id}-\chi\Big (\frac{\eps|\xi|}\kappa\Big)\sum_{j=1}^4 P_j(\eps\xi)\Big) \mathcal{F}_x{f (\xi)} \Big]
$$
so  since $f$ belongs to~$ \operatorname{Ker} L$, we have 
\begin{equation} \label{Ueps5}
U^{\eps \sharp}(t ) f = U^\eps(t) \Bigg[\mathcal{F}_x^{-1} \Bigg( \Big({\rm Id}-\chi\Big (\frac{\eps|\xi|}\kappa\Big)\Big) - \eps |\xi| \chi\Big (\frac{\eps|\xi|}\kappa\Big) \sum_{j=1}^4 \tilde P_{j} (\eps \xi)\Bigg) \mathcal{F}_x{f (\xi)}\Bigg] 
\end{equation}
with notation~(\ref{deftildeP1j}).
The $X^{\ell,k}$-norm of the first term in the right-hand side of~\eqref{Ueps5} is simply estimated using~(\ref{estimatechi}). The terms coming from the second part of the right-hand side of~\eqref{Ueps5} are estimated as the terms $U^\eps_{jm}$ for $1 \le j \le 4$ and~$1 \leq m \leq 2$. In conclusion, we obtain
$$
\|U^{\eps \sharp}(t)f\|_{\ell,k} \lesssim \eps \| f\|_{\ell+1,k}\quad \forall \, t \lesssim 1 \,.
$$
Lemma~\ref{lem:UepstoU0} is proved.
\end{proof}
The following corollary is an immediate consequence of Lemmas~\ref{lem:Uepsdisp} and~\ref{lem:UepstoU0}  along with the triangular inequality.
\begin{Cor}\label{estimateUepsd=2}
Let $\Omega = \R^2$. Fix~$\ell \ge0$ and $k>2$  and consider~$f$ in~$X^{\ell,k}   \cap L^{2}_v L^1_x$. Then  there holds for all $\eps \in (0,1)$
$$
\sup_{t \ge 0} \, \langle t \rangle^{\frac12} \big\| U^\eps (t)f\big \|_{\ell,k}  \lesssim \|f\|_{\ell,k} + \|f\|_{L^{2}_v L^1_x} \, .
$$
\end{Cor}
\subsection{Estimates on $\Psi^\eps(t)$}
Let us now give some estimates on the bilinear operator $\Psi^\eps(t)$. We also state some specific estimates in the case of $\R^2$, which  is different due to the presence of the weight in time in the definition of $\mathcal{X}^{\ell,k}_T$, when one of the two variables is $\delta^{\eps,\eta}$ which is defined in~\eqref{defgepseta} (see Lemma~\ref{lem:PsiepscontR2}). Finally, to end this section, we give another specific estimate on $\Psi^\eps(t)$ when one of the two variables is $\bar \delta^{\eps,\eta}$ (defined in Lemma~\ref{lem:reduction}) in the case of~$\R^d$, $d=2,3$, which will be useful to treat ill-prepared data. 

\begin{lem} \label{lem:Psiepscont}
 Let~$\ell>d/2$,~$k>d/2+1$ be given.  Then $\Psi^\eps(t) $ is a bilinear symmetric continuous map  from~$\mathcal{C}_b ([0,T], X^{\ell,k}) \times \mathcal{C}_b ([0,T], X^{\ell,k} )$ to~$\mathcal{C}_b ([0,T], X^{\ell,k})$, and there is a constant~$C$ such that  for all~$T\geq 0$ and all~$\eps>0$, {\begin{equation}\label{contest}
\|\Psi^\eps(t)(f_1,f_2)\|_{\mathcal{X}^{\ell,k}_T} \leq C \| f_1  \|_{\mathcal{X}^{\ell,k}_T}  \| f_2  \|_{\mathcal{X}^{\ell,k}_T} \, , \quad \forall \, f_1, f_2 \in \mathcal{X}^{\ell,k}_T\, .
\end{equation}}
  \end{lem}
 \begin{proof}
 As in~(\ref{eq:decompUeps}), we decompose $\Psi^\eps(t)$ into two parts:
 \begin{align*}
  &\quad  \Psi^{\e}(t) (f_1,f_2) \\
  &= 
    {1 \over \eps} \int_0^te^{(t-{t'})A^\e}\Gamma(f_1,f_2)({t'})  \, d{t'} + {1 \over \e}  \int_0^t \int_0^{t-{t'}} e^{(t-{t'}-\tau)A^\e} {1 \over \e^2} K U^\e(\tau)\Gamma(f_1,f_2)({t'}) \, d\tau \, d{t'} \\
&=: \Psi^{\e,1}(t)(f_1,f_2) +  {1 \over \e^2} \int_0^t e^{(t-{t'})A^\e} K \Psi^{\e}({t'})(f_1,f_2) \, d{t'} \,. 
 \end{align*}
 As in the proof of Lemma~\ref{lem:UepstoU0}, using properties of $K$, we have that
 \begin{equation}\label{estimatePsiPsi1}
 \|  \Psi^\eps (t)(f_1,f_2)\|_{{\mathcal X}_T^{\ell,k} } \lesssim \sum_{j=0}^k \|\Psi^{\e,1}(t)(f_1,f_2)\|_{{\mathcal X}_T^{\ell,j}} + \|    \Psi^\eps (t)(f_1,f_2)\|_{{\mathcal Y}_T^\ell}\,,
 \end{equation}
 where we have defined
 $$\mathcal{Y}_T^\ell := 
 \Big\{f = f(t,x,v) \, / \,  f\in L^\infty(\mathds{1}_{[0,T]}(t)\chi_\Omega(t),H^\ell_xL^2_v)\Big\}
$$
endowed with the norm
$$
 \|f\|_{\mathcal{Y}_T^\ell }:=\sup_{t\in [0,T]}\chi_\Omega(t)\|f(t)\|_{H^\ell_xL^2_v}\,.
 $$

\noindent{\it Estimates on $\Psi^{\e,1}(t)$.} Let $0 \le j \le k$ be given. We first use the explicit form of~$e^{tA^\e}$ given by~\eqref{semigroupAeps} in order to deduce that 
\begin{align*}
&\|\Psi^{\e,1}(t)(f_1,f_2)\|_{{\ell,j}} \le 
\left\| {1 \over \e} \int_0^t e^{-\nu(v)\frac{t-t'}{\e^2}} \|\Gamma(f_1 ,f_2)(t')\|_{H^\ell_x}  \, d{t'} \right\|_{L^{\infty,j}_v}\, .
\end{align*}
We have 
\begin{equation} \label{eq:Psiepsellj}
\begin{aligned}
&\|\Psi^{\e,1}(t)(f_1,f_2)\|_{{\ell,j}} \\
&\quad \le\sup_{v \in \R^d} {1 \over \e} \int_0^t e^{-\nu(v)\frac{t-{t'}}{\e^2}} \nu(v) \nu^{-1}(v) \langle v \rangle^j   \| \Gamma( f_1,f_2)  ({t'})\|_{H^\ell_x}  \, d{t'}  \\
&\quad \le  \|\Lambda^{-1} \Gamma (f_1,f_2)  \|_{\mathcal{X}_T^{\ell,k} } \left\| {1 \over \e} \int_0^t e^{-\nu(v)\frac{t-{t'}}{\e^2} } \nu(v) \, d{t'} \right\|_{L^\infty_v} \\
&\quad \lesssim \e  \|\Lambda^{-1}  \Gamma (f_1,f_2)  \|_{\mathcal{X}_T^{\ell,k} }\,,
\end{aligned}
\end{equation}
where we have defined, with notation~(\ref{defnu(v)}),
\begin{equation}\label{defLambda(v)}
\Lambda(v):=\nu(v)\,.
\end{equation}
Using~\eqref{estimUkai1}, we immediately have that 
$$
\|\Lambda^{-1} \Gamma (f_1,f_2)  \|_{\mathcal{X}^{\ell,k}_T} \lesssim \|f_1\|_{\mathcal{X}^{\ell,k}_T} \|f_2\|_{\mathcal{X}^{\ell,k}_T}\,. 
$$
From this, we conclude that for any $0 \le j \le k$,
$$
\|\Psi^{\e,1}(t)(f_1,f_2)\|_{\mathcal{X}_T^{\ell,j} } \le C\e \|f_1\|_{\mathcal{X}_T^{\ell,k} }\|f_2\|_{\mathcal{X}^{\ell,k}_T}  \,.
$$
%

\noindent {\it Estimates on $\Psi^\e(t)$ in $\mathcal{Y}^\ell_T$.} Recalling that~$
\Pi_{L}   $
is the orthogonal projector onto $\operatorname{Ker} L$ and using a weak formulation of the collision operator $\Gamma$, it can be shown   thanks to physical laws of elastic collisions that
$$
\Pi_{L}  \Gamma(f_1,f_2) = 0 \quad \forall \,f_1,f_2 
$$
so   we are going to be able to use Lemma~\ref{lem:decayWeps}. 
Let us start with the case when~$\Omega$ is not~$\R^2$  and let us define
\begin{equation} \label{def:tildechi} \tilde \chi_\Omega(t):= \left\lbrace \begin{array}{lcl}
t^{-\frac12} e^{-\sigma t} & \text{if} &\Omega = \T^d, \, \, d = 2,3 \\
t^{-\frac12} \langle t \rangle^{-\frac 34} &\text{if} & \Omega = \R^3 \, .
\end{array} \right. 
\end{equation}
We then estimate $\|\Psi^{\e}(t) (f_1,f_2)\|_{\mathcal{Y}^\ell_T}$ using the fact that thanks to~\eqref{estimUkai1}-\eqref{estimUkai2}
$$
\begin{aligned}
 \big \|  \Psi^{\e}(t) (f_1,f_2)\big\|_{H^\ell_x L^2_v} &\lesssim  \int_0^t \tilde \chi_\Omega(t-t')    \|f_1(t')\|_{\ell,k} \|f_2(t')\|_{\ell,k} \, dt'\\
& \lesssim  \int_0^t\tilde \chi_\Omega(t-t') \, dt' \| f _1 \|_{\mathcal{X}^{\ell,k}_T}\| f _2 \|_{\mathcal{X}^{\ell,k}_T}
\lesssim   \| f _1 \|_{\mathcal{X}^{\ell,k}_T}\| f _2 \|_{\mathcal{X}^{\ell,k}_T}
\end{aligned}
 $$
 since $\tilde \chi_\Omega$ is integrable over $\R^+$.  To conclude it remains  to deal with the case when~$\Omega=\R^2$.
Arguing in a similar fashion we have
 $$
 \|\Psi^\eps(t) (f_1,f_2) \|_{\mathcal{Y}^{\ell}_\infty} \lesssim   \sup_{t \ge 0} \, \langle t \rangle^{\frac14} \int_0^t \frac{1}{(t-t')^\frac12 \langle t-t' \rangle^\frac12} {1 \over \langle t' \rangle^\frac12} \, dt'   \|f_1\|_{\mathcal{X}^{\ell,k}_\infty} \|f_2\|_{\mathcal{X}^{\ell,k}_\infty}    \, ,
 $$
 so let us prove that 
 $$
 \langle t \rangle^{\frac14}   \int_0^t \frac{1}{(t-t')^\frac12\langle t-t' \rangle^\frac12} {1 \over \langle t' \rangle^\frac12} \, dt' 
 $$
 is uniformly bounded in $t \ge 0$. We
  define
 $$
 I(t,s,\tau):= \langle t \rangle^{\frac14}  \int_s^\tau \frac{1}{(t-t')^{\frac12}\langle t-t' \rangle^{\frac12}}{dt' \over \langle t' \rangle^{\frac12}}   
 $$
 and let us write~$ I(t,0,t)=  I(t,0,t/2)+ I(t,t/2,t)$ and estimate both terms separately. The second one is the easiest since
 $$
\begin{aligned} I(t,t/2,t) &= \int_{\frac t2}^t \frac{\langle t \rangle^{\frac14} }{(t-t')^{\frac12}\langle t-t' \rangle^{\frac12}}{dt' \over \langle t' \rangle^{\frac12}} \lesssim  \int_{\frac t2}^t \frac{1}{(t-t')^{\frac12}\langle t-t' \rangle^{\frac12}}{dt' \over \langle t' \rangle^{\frac14}} \\
\end{aligned} $$
so that 
$$
I(t,t/2,t) \lesssim \left\| {1 \over {t^{\frac12}\langle t \rangle^{\frac12}}} \right\|_{L^{\frac54} } \left\| {1 \over {\langle t \rangle^\frac14}} \right\|_{L^5 } < \infty\, .
$$
As to the first term we start by assuming that~$t\lesssim 1$, then
$$
\begin{aligned} 
I(t,0,t/2)
&\lesssim  \int_0^{t/2} \frac{dt'}{\sqrt{t-t'}} < \infty\, .
\end{aligned}$$
On the other hand if~$t\gtrsim 1$
then using the fact that when~$0\leq t'\leq t/2$ then~$1 \lesssim t/2\leq t-t'\leq t$ we have
$$
\frac{\langle t \rangle^{\frac14} }{(t-t')^\frac12 \langle t-t' \rangle^\frac12}\lesssim \frac{1}{(t-t')^\frac12 \langle t-t' \rangle^\frac14}
$$
so
$$
\begin{aligned} 
I(t,0,t/2)
&\lesssim \left\| {1 \over {\langle t \rangle^{\frac34}}} \right\|_{L^{\frac32} } \left\| {1 \over {\langle t \rangle^{\frac12}} }\right\|_{L^3 } < \infty\, .
\end{aligned}$$
The proof is complete.
 \end{proof}
 \begin{lem} \label{lem:PsiepscontR2}
 Let $\Omega = \R^2$, $\ell>1$ and~$k>2$. For any $f \in L^\infty(\R^+, X^{\ell,k})$, there holds
 $$
 \lim_{\eps \to 0} \|\Psi^\eps(t) (\delta^{\eps,\eta},f) \|_{\mathcal{X}^{\ell,k}_\infty} \lesssim \eta \|f\|_{L^\infty(\R^+, X^{\ell,k})}  
 $$
 where we recall that $\delta^{\eps,\eta}$ is defined in~\eqref{defgepseta}.   \end{lem}
 \begin{proof}
 Following the proof of Lemma~\ref{lem:Psiepscont}, and in particular~(\ref{estimatePsiPsi1}), it is enough to estimate~$\|\Psi^{\eps,1}(t)(\delta^{\eps,\eta},f) \|_{\mathcal{X}^{\ell,k}_\infty}$ and~$\|\Psi^{\eps}(t)(\delta^{\eps,\eta},f) \|_{\mathcal{Y}^{\ell}_\infty}$. 
 Let us notice that using  Corollary~\ref{estimateUepsd=2} we find that 
 \begin{equation} \label{eq:decaydeltaepseta}
 \|\delta^{\eps,\eta}(t) \|_{\ell,k}\lesssim {1 \over \langle t \rangle^\frac12} \big(\|\delta^\eta_{\rm in}\|_{\ell,k} + \|\delta^\eta_{\rm in}\|_{L^2_vL^1_x} \big) \lesssim \frac{\eta}{\langle t \rangle^\frac12}
 \end{equation}
 from~\eqref{deltaineta}-\eqref{deltainetaR2}. 
For the estimate of~$\|\Psi^{\eps,1}(t)(\delta^{\eps,\eta},f) \|_{\mathcal{X}^{\ell,k}_\infty}$, we notice that if $0 \le {t'} \le t/2$, then there holds~$t/2 \le t-{t'} \le t$ from which we deduce that 
$$
\langle t \rangle^{\frac14}  e^{-\nu(v)\frac{t-{t'}}{\e^2}} \lesssim e^{-\nu(v)\frac{t-{t'}}{2\e^2}}\, .
$$
If $t/2 \le {t'} \le t$, we have
$$
 \langle t \rangle^{\frac14}   \lesssim  \langle t'\rangle^{\frac14} \, .
$$
In all cases, we can thus write the following bound for~$0\leq j \leq k$:
\begin{align*}
&\langle t \rangle^{\frac14} \|\Psi^{\eps,1}(t)(\delta^{\eps,\eta},f) \|_{{\ell,k}} \\
&\quad \le\sup_{v \in \R^2} {2 \over \e} \int_0^t e^{-\nu(v)\frac{t-{t'}}{2\e^2}} \nu(v) \nu^{-1}(v) \langle v \rangle^j \langle t'\rangle^{\frac14}\| \Gamma(\delta^{\eps,\eta},f)  ({t'})\|_{H^\ell_x}  \, d{t'}  \\
&\quad \le 2\|\Lambda^{-1}  \Gamma (\delta^{\eps,\eta},h)  \|_{\mathcal{X}_\infty^{\ell,k} } \left\| {1 \over \e} \int_0^t e^{-\nu(v)\frac{t-{t'}}{2\e^2}} \nu(v) \, d{t'} \right\|_{L^\infty_v} \\
&\quad \lesssim \e  \|\Lambda^{-1}  \Gamma (\delta^{\eps,\eta},h)  \|_{\mathcal{X}_\infty^{\ell,k} }\,.
\end{align*}
Using~\eqref{estimUkai1} and~\eqref{eq:decaydeltaepseta}, we   have that 
$$
\|\Lambda^{-1} \Gamma (\delta^{\eps,\eta},h)  \|_{\mathcal{X}^{\ell,k}_\infty} \lesssim \|\delta^{\eps,\eta}\|_{\mathcal{X}^{\ell,k}_\infty} \|f\|_{L^\infty(\R^+, X^{\ell,k})} \lesssim \eta \|f\|_{L^\infty(\R^+, X^{\ell,k})}\,. 
$$
From this, we are able to conclude for the first part of the estimate. 

 As to the estimate in~$\mathcal{Y}^\ell_\infty$, we proceed as in the proof of Lemma~\ref{lem:Psiepscont} to deduce that 
 $$
 \|\Psi^\eps(t) (\delta^{\eps,\eta},f) \|_{\mathcal{Y}^{\ell}_\infty} \lesssim \eta \, \sup_{t \ge 0} \, \langle t \rangle^{\frac14} \int_0^t \frac{1}{(t-t')^\frac12 \langle t-t' \rangle^\frac12} {1 \over \langle t' \rangle^\frac12} \, dt' \|f\|_{L^\infty(\R^+,X^{\ell,k})} \, ,
 $$
 and the result follows directly as above. Lemma~\ref{lem:PsiepscontR2} is proved. \end{proof}
 
 \begin{lem} \label{lem:Psiepsbardelta}
 Let $\Omega= \R^d$, $d=2,3$, $\ell>d/2$ and~$k>d/2+1$. For any $f \in \mathcal{X}^{\ell,k}_T$, for any~$\eta>0$, there exists~$C_\eta>0$, independent of~$T$, such that 
 $$
 \|\Psi^\eps(t)(\bar \delta^{\eps,\eta},f)\|_{\mathcal{X}^{\ell,k}_T} \lesssim C_\eta \, \eps^\frac12 \|f\|_{\mathcal{X}^{\ell,k}_T} \, .
 $$
 \end{lem}
 \begin{proof}
Recall that by definition
$$
\overline\delta^{\eps,\eta}= U^\e_{\rm{disp}}(t)g^{\eta}_{\rm{in}} +U^{\e\sharp}(t) g^{\eta}_{\rm{in}}\,.
$$
Defining
\begin{equation}\label{decompositiontildeg}
\overline  \delta^{\eps,\eta}_1(t):=U^\e_{\rm{disp}}(t)g^{\eta}_{\rm{in}} \quad \mbox{and} \quad \overline \delta^{\eps,\eta}_2(t):=U^{\e\sharp}(t) g^{\eta}_{\rm{in}} \, , 
\end{equation}
we shall study separately the contributions of~$\Psi^\eps(t)(\overline  \delta^{\eps,\eta}_1,f)$ and $\Psi^\eps(t)(\overline  \delta^{\eps,\eta}_2,f)$. Following the proof of Lemma~\ref{lem:Psiepscont}, it is enough to estimate~$\|\Psi^{\eps,1}(t)(\bar\delta^{\eps,\eta},f) \|_{\mathcal{X}^{\ell,k}_T}$ and~$\|\Psi^{\eps}(t)(\bar\delta^{\eps,\eta},f) \|_{\mathcal{Y}^{\ell}_T}$. 

\smallskip

\noindent {\it Step 1: estimates in $\mathcal{Y}^\ell_T$.}
We separate the analysis according the different cases for $\Omega$.

\smallskip
\noindent $\bullet$ $ $  {\it The case of $\R^2$. } We first focus on the estimate of $\|\Psi^\eps(t) (\overline  \delta^{\eps,\eta}_1, f)\|_{\mathcal{Y}^\ell_\infty}$. We use the estimate~\eqref{eq:GammaHellxL2v} and the second estimate coming from Lemma~\ref{lem:decayWeps}. We have
\begin{align*}
\|\Psi^\eps(t) (\overline  \delta^{\eps,\eta}_1,f)\|_{\mathcal{Y}^\ell_\infty} \lesssim \sup_{t \ge 0} \, \langle t \rangle^{1 \over 4} \int_0^t {1 \over {(t-t')^{1\over 2}}} \|\overline  \delta^{\eps,\eta}_1(t')\|_{L^{\infty,k}_vW^{\ell,\infty}_x} {1 \over \langle t' \rangle^{1 \over 4}} \, dt' \|f\|_{\mathcal{X}^{\ell,k}_\infty}
\end{align*}
and thus thanks to Lemma~\ref{resultstildegeps}, estimate~(\ref{estimatedelta1})
\begin{align*}
\|\Psi^\eps(t) (\overline  \delta^{\eps,\eta}_1, f)\|_{\mathcal{Y}^\ell_\infty} &\le C_\eta  \, \eps^{1\over 2} \sup_{t \ge 0} \, \langle t \rangle^{1 \over 4} \int_0^t {1 \over {(t-t')^{1\over 2}}} {1 \over (t')^{1 \over 2}} {1 \over \langle t' \rangle^{1 \over 4}} \, dt' \|f\|_{\mathcal{X}^{\ell,k}_\infty} \\
&\le C_\eta \, \eps^{1\over 2} \sup_{t \ge 0} \, I_1(t,0,t) \|f\|_{\mathcal{X}^{\ell,k}_\infty}
\end{align*}
with 
$$
I_1(t,s,\tau) := \langle t \rangle^{1 \over 4} \int_s^\tau {1 \over {(t-t')^{1\over 2}}} {1 \over (t')^{1 \over 2}} {1 \over \langle t' \rangle^{1 \over 4}} \, dt' \, .
$$
It thus remains to verify that 
$I_1(t,0,t)$
is uniformly bounded in time. First, we notice that 
$$
I_1(t,t/2,t)
\lesssim \int_{t/2}^t  {1 \over {(t-t')^{1\over 2}}} {dt' \over (t')^{1 \over 2}} \lesssim {1 \over t^{1\over 2}} \int_{t/2}^t {dt' \over {(t-t')^{1\over 2}}} \lesssim 1 \, .
$$
Then, if $t \lesssim 1$, 
$$
I_1(t,0,t/2)
\lesssim {1 \over t^{1\over 2}} \int_0^{t/2} \frac{dt'}{(t')^{1\over 2}} \lesssim 1 \, .
$$
Finally if $t \gtrsim 1$, 
$$
I_1(t,0,t/2)
\lesssim {1 \over t^{1\over 4}} \int_0^1 \frac{dt'}{(t')^{1\over 2}} +  {1 \over t^{1\over 4}} \int_1^{t/2} \frac{dt'}{(t')^{3\over 4}}  \lesssim 1 \, ,
$$
from which we are able to conclude. 

We now turn to the estimate of $\|\Psi^\eps(t) (\overline  \delta^{\eps,\eta}_2, f)\|_{\mathcal{Y}^\ell_\infty}$. We   use the third estimate given by Lemma~\ref{lem:decayWeps} and Lemma~\ref{resultstildegeps}, estimate~(\ref{estimatedelta2}), combined with~(\ref{decayUepssharp}) and~\eqref{estimUkai1}-\eqref{estimUkai2} to get
$$
\|\Psi^\eps(t) (\overline  \delta^{\eps,\eta}_2, f)\|_{\mathcal{Y}^\ell_\infty} 
\lesssim \sup_{t \ge 0} \, \langle t \rangle^{1 \over 4} \int_0^t {1 \over {(t-t')^{1 \over 2} \langle t-t'\rangle^{1 \over 2}}} e^{- \alpha {t' \over \eps^2}} \, dt' \|f\|_{\mathcal{X}^{\ell,k}_\infty}= \sup_{t \ge 0} I_2^\eps(t,0,t) \|f\|_{\mathcal{X}^{\ell,k}_\infty}
$$
where 
$$
I_2^\eps(t,s,\tau) := \langle t \rangle^{1 \over 4} \int_s^\tau {1 \over {(t-t')^{1 \over 2} \langle t-t'\rangle^{1 \over 2}}} e^{- \alpha {t' \over \eps^2}} \, dt' \, .
$$
First, let us notice that if $t \lesssim 1$, then 
$$
I_2^\eps(t,0,t) \lesssim \int_0^t {1 \over (t-t')^{1 \over 2}} e^{- \alpha {t' \over \eps^2}} \, dt' 
$$
and thus, using Young's inequality, 
$$
I_2^\eps(t,0,t) \lesssim \left\| {1 \over t^{\frac12}}\right\|_{L^{\frac32}_t([0,1])}\Big\|e^{- \alpha {t' \over \eps^2}}\Big\|_{L^3_t} \lesssim \eps^{\frac23} \, . 
$$
Similarly, when $t \gtrsim 1$, we have 
$$
I_2^\eps(t,0,t/2) \lesssim \int_0^{t/2} {1 \over {(t-t')^{1 \over 4} \langle t-t'\rangle^{1 \over 2}}} e^{-\alpha {t' \over \eps^2}} \, dt' 
\lesssim \left\| {1 \over {t^{\frac14} \langle t \rangle^{\frac12}}}\right\|_{L^{2}_t}\Big\|e^{- \alpha {t' \over \eps^2}}\Big\|_{L^2_t} \lesssim \eps \,.
$$
Finally 
$$
I_2^\eps(t,t/2,t) \lesssim \int_{t/2}^t {1 \over {(t-t')^{1 \over 2} \langle t-t'\rangle^{1 \over 2}}} e^{-\alpha {t' \over {2\eps^2}}} \, dt' 
\lesssim \left\| {1 \over {t^{\frac12} \langle t \rangle^{\frac12}}}\right\|_{L^{\frac32}_t}\Big\|e^{-\alpha {t' \over \eps^2}}\Big\|_{L^3_t} \lesssim \eps^{\frac23} \, . 
$$
\smallskip
\noindent $\bullet$ $ $  {\it The case of $\R^3$.} The strategy of the proof is similar to the case of $\R^2$ so we skip the details. For the term $\|\Psi^\eps(t) (\overline  \delta^{\eps,\eta}_1, f)\|_{\mathcal{Y}^\ell_T}$, we notice that Lemma~\ref{resultstildegeps}, estimate~(\ref{estimatedelta1}) implies that 
$$
\|\overline  \delta^{\eps,\eta}_1(t')\|_{L^{\infty,k}_v W^{\ell,\infty}_x} \le C_\eta \left(\frac{\eps}{t'}\right)^\frac12 \, . 
$$
It is thus enough to check that the following integral is uniformly bounded in time
$$
\int_0^t {1 \over (t-t')^\frac12} {1 \over (t')^\frac12} \, dt' =J_1(0,t)\, , \quad \text{with} \quad J_1(s,t):=\int_s^t {1 \over (t-t')^\frac12} {1 \over (t')^\frac12} \, dt'  \, .
$$
We have 
$$
J_1(0,t/2) \lesssim {1 \over t^\frac12} \int_0^{t/2} {dt' \over (t')^\frac12} \lesssim 1
\quad \text{and} \quad 
J_1(t/2,t) \lesssim {1 \over t^\frac12} \int_{t/2}^t {dt' \over (t-t')^\frac12} \lesssim 1 \, ,
$$
which yields the result. 
In order to estimate $\|\Psi^\eps(t) (\overline  \delta^{\eps,\eta}_2, f)\|_{\mathcal{Y}^\ell_T}$, we just have to bound 
$$
\int_0^t {1 \over {(t-t')^\frac12 \langle t-t' \rangle^\frac34}} e^{- \alpha \frac{t}{\eps^2}} \, dt' =: J_2^\eps(0,t) \, . 
$$
Using Young's inequality, we have:
$$
J_2^\eps(0,t) \lesssim \left\| {1 \over {t^{\frac12} \langle t \rangle^{\frac34}}}\right\|_{L^{\frac32}_t}\Big\|e^{- \alpha {t' \over \eps^2}}\Big\|_{L^3_t} \lesssim \eps^{\frac23} \, .
$$

\medskip
\noindent {\it Step 2: estimates in $\mathcal{X}^{\ell,j}_T$.}
As in the proof of Lemma~\ref{lem:Psiepscont} (see estimate~\eqref{eq:Psiepsellj}) and of Lemma~\ref{lem:PsiepscontR2} for the $\R^2$-case, we have for any~$0 \le j \le k$ and any~$t\in [0,T]$
\begin{align*}
&\chi_\Omega(t)\big\|\Psi^{\eps,1}(t)\big(\overline\delta^{\eps,\eta},f\big)\big\|_{{\ell,j}} \\
&\quad \lesssim \e \Big \|\Lambda^{-1}  \Gamma\big(\overline\delta^{\eps,\eta},f\big) \Big\|_{\mathcal{X}_T^{\ell,k} } \\
&\quad \lesssim \eps \|\overline\delta^{\eps,\eta}\|_{\mathcal{X}^{\ell,k}_T} \|f\|_{L^\infty([0,T],X^{\ell,k})}
\le C_\eta \, \eps \, \|f\|_{L^\infty([0,T],X^{\ell,k})}\,,
\end{align*}
using Lemma~\ref{resultstildegeps} and this concludes the proof of Lemma~\ref{lem:Psiepsbardelta}. 
\end{proof}

\subsection{Proof of Lemma \ref{resultstildegeps}}\label{proofstrichartz} 
Recalling the notation~(\ref{decompositiontildeg})
 we note that the results on~$ \bar \delta^{\eps,\eta}_2$ follow directly  from the properties on~$U^{\eps \sharp}$ recalled in Appendix~\ref{specBoltz}, namely Lemma~\ref{EllisPinsky}.  Turning to~$ \bar \delta^{\e,\eta}_1$, we remark that we can proceed similarly as in the proof of Lemma~\ref{lem:Uepsdisp} using the heat flow to obtain
 \begin{equation} \label{eq:deltabar2}
 \|\bar \delta^{\eps,\eta}_2\|_{\ell,k} \lesssim {1 \over \langle t \rangle^{\frac d4}} \big(\|g_{\rm in}^\eta\|_{\ell,k} + \|g_{\rm in}^\eta\|_{L^2_vL^1_x} \big) \lesssim \frac{C_\eta}{\langle t \rangle^\frac d4} \, \cdotp
 \end{equation}
 Next, to prove the dispersion estimate 
 $$
 \|\bar \delta^{\eps,\eta}_1\|_{L^{\infty,k}_v W^{\ell,x}_x} \le C_\eta \Big(1 \wedge \Big(\frac{\eps}{t}\Big)^{\frac{d-1}{2}} \Big) \, ,
 $$ 
 recall that in Fourier variables, the terms inside $\widehat U^\eps_{\rm{disp}}(t,\xi)$ are of the form
 $$
\exp \Big(
i\alpha |\xi| \frac t\e - \beta t |\xi|^2 \Big)P^0\Big(\frac\xi{|\xi|}\Big) \quad \text{with} \quad \alpha, \beta >0
 $$
 and where~$P^0\big(\frac\xi{|\xi|}\big)$ can be expressed as a finite sum of functions of the form
 $$
 P^0\Big(\frac\xi{|\xi|}\Big)\hat u = a\Big(\frac\xi{|\xi|}\Big)b(v)\int c(v)\hat u (\xi,v)\,dv
 $$
 where~$a$ is a smooth function on the sphere, and $b$ and~$c$  are in~$L^{\infty,\beta}$ for all~$\beta \geq 0$.
 It follows that
 $$
 \big |U^\eps_{\rm{disp}}(t) g^\eta_{{\rm in}} (t,x,v)\big|  \lesssim   |b(v)| \Big| \int e^{ix\cdot \xi + {i\alpha |\xi| \frac t\e}} a\Big(\frac\xi{|\xi|}\Big)
  {\mathcal F}_x \widetilde g (t,\xi)\, d\xi 
   \Big|
  $$
 with
 $$
 {\mathcal F}_x \widetilde g (t,\xi):= \int c(v)e^{- \beta t |\xi|^2}\widehat g_{\rm{in}} (\xi,v)\,dv \,.
 $$
 But by~\cite{Ukai} and classical dispersive estimates on the wave operator in~$d$ space dimensions (it is here that we use the fact that~$\Omega = \R^d$) we know that 
  $$
  \Big| \int e^{ix\cdot \xi + {i\alpha |\xi| \frac t\e}} a\Big(\frac\xi{|\xi|}\Big)
  {\mathcal F}_x\widetilde g(t,\xi)\, d\xi 
   \Big|
\lesssim \Big(1+\frac\e t\Big)^{\frac{d-1}2} \big(\| \widetilde g(t)\|_{L^1} + \| \widetilde g(t)\|_{H^\ell}\big)
 $$
 and the result follows by continuity of the heat flow.
 Concerning the term coming from~$U^\eps_{\rm disp}(t)$, we can proceed as in the proof of Lemma~\ref{lem:Uepsdisp} and just use the continuity of the heat flow in~$H^\ell_x$ to get 
 $$
  \|\bar \delta^{\eps,\eta}_2\|_{\ell,k} \lesssim \|g_{\rm in}^\eta\|_{\ell,k} \lesssim 1
 $$
as previously. 
 Lemma \ref{resultstildegeps} is proved. 
 \qed

\section{Proof of Propositions~\ref{linear}, \ref{nonlinear} and~\ref{limitdatawp}}
\subsection{Proof of Proposition~\ref{linear}} The first steps of the proof follow the ones of Lemma~\ref{lem:Psiepscont}: first, we split the operator~$L$ defined in~(\ref{defoperator}) into two parts as in~(\ref{eq:decompUeps}), which provides   the decomposition~\eqref{estimateLHlk}. The last two steps are then devoted to the analysis of the terms of this decomposition.
  
 \smallskip
 \noindent
{\it {Step 1.}} $ $
From the decomposition~\eqref{eq:decompUeps}, we deduce that for any~$\lambda \geq 0$, with notation~(\ref{defopslambda}),
\begin{align*}
 \mathcal{L}_\lambda^{\e}(t)h 
&= {2 \over \eps} \int_0^te^{(t-t')A^\e}  \Gamma(\overline g^{\eps,\eta}- \delta^{\eps,\eta} ,h) (t') e^{-\lambda \int_{t'}^t \|\overline g^{\eps,\eta}(t'')\|_{\ell,k}^r\, dt'' } \, dt'\\
&    \qquad+ {2 \over \e} \int_0^t \int_0^{t-t'} e^{(t-t'-\tau)A^\e} {1 \over \e^2} K U^\e(\tau)  \Gamma(\overline g^{\eps,\eta}- \delta^{\eps,\eta} ,h) (t')e^{-\lambda \int_{t'}^t \|\overline g^{\eps,\eta}(t'')\|_{\ell,k}^r\, dt'' } \, d\tau \, dt' \\
& =: \mathcal{L}_\lambda^{\e,1}(t)h + \mathcal{L}_\lambda^{\e,2}(t)h\,.
\end{align*}
Performing a change of variables, one can notice that 
$$
 \mathcal{L}_\lambda^{\e,2}(t)h = {1 \over \e^2} \int_0^t e^{(t-t')A^\e} K \mathcal{L}^{\eps}_\lambda (t') h \, dt'\,. 
$$
Exactly as we obtained~\eqref{reiter}, we are then able to prove that
\begin{equation} \label{estimateLHlk}
\|  \mathcal{L}^{\eps}_\lambda (t)h\|_{{\mathcal X}_T^{\ell,k} } \lesssim \sum_{j=0}^k \|\mathcal{L}_\lambda^{\e,1}(t)h\|_{{\mathcal X}_T^{\ell,j}} + \|   \mathcal{L}^{\eps}_\lambda (t)h\|_{{\mathcal Y}_T^\ell} \, ,
\end{equation}
where we recall that  $$\mathcal{Y}_T^\ell := 
 \Big\{f = f(t,x,v) \, / \,  f\in L^\infty(\mathds{1}_{[0,T]}(t)\chi_\Omega(t),H^\ell_xL^2_v)\Big\}\,.
$$
In the two next steps,  we are going to estimate respectively the quantities~$\| \mathcal{L}^{\eps}_\lambda (t)h\|_{\mathcal{Y}_T^\ell}$ and~$\|\mathcal{L}_\lambda^{\e,1}(t)h\|_{{\mathcal X}_T^{\ell,j}}$ for $0 \le j \le k$. 

 \smallskip
 \noindent
{\it {Step 2.}} $ $
Let us prove that
\begin{equation}\label{estimateYTell}
 \forall \, \lambda >0 \, , \quad    \big \|    \mathcal{L}_\lambda^{\e}(t) h\big\|_{\mathcal{Y}^\ell_T}
       \lesssim\Big(  {1 \over \lambda^\frac1r} +  \eta \Big)\| h  \|_{\mathcal{X}_T^{\ell,k}}\,.
\end{equation}
As in the proof of Lemma~\ref{lem:Psiepscont}, we are going to be able to use results from Lemma~\ref{lem:decayWeps} since 
$$
\Pi_{L}  \Gamma(f_1,f_2) = 0 \quad \forall \,f_1,f_2 \, .
$$

\noindent $\bullet $ $ $  {\it The case of $\T^d$, $d = 2,3$.}$ $ 
With the definition of~ $ \mathcal{L}_\lambda^{\e}$ given in~(\ref{defopslambda}), the first estimate from Lemma~\ref{lem:decayWeps} and~\eqref{estimUkai1} we get
$$
\begin{aligned}
 \big \|    \mathcal{L}_\lambda^{\e}(t) h\big\|_{H^\ell_xL^2_v} &  \lesssim  \int_0^t \frac{e^{-\sigma(t-t')}}{(t-t')^\frac12}  e^{- \lambda \int_{t'}^t \|\overline g^{\eps,\eta}(t'')\|_{\ell,k}^r\, dt''}
\|(\overline g^{\eps,\eta}-\delta^{\eps,\eta})(t')\|_{\ell,k}\|h(t')\|_{\ell,k} \, dt'\\
&\lesssim  \int_0^t \frac{e^{-\sigma(t-t')}}{(t-t')^\frac12}  e^{- \lambda \int_{t'}^t \|\overline g^{\eps,\eta}(t'')\|_{\ell,k}^r\, dt''}
\left(\|\overline g^{\eps,\eta}(t')\|_{\ell,k} + \|\delta^{\eps,\eta}(t')\|_{\ell,k}\right) \, dt'  \,  \|h\|_{\mathcal{X}_T^{\ell,k}} \, . \\
\end{aligned}
$$
When~$\lambda >0$, writing~$\displaystyle \frac1r+\frac1{r'} = 1$ with $1<r'<4/3$ (since~$r>4$ by definition) gives thanks to~(\ref{estimatedeltaepseta0})
$$
\begin{aligned}
\big \|    \mathcal{L}_\lambda^{\e}(t) h\big\|_{H^\ell_xL^2_v}  &\lesssim \Big(\big\|   e^{- \lambda \int_{t'}^t \|\overline g^{\eps,\eta}(t'')\|_{\ell,k}^r\, dt'' } \|\overline g^{\eps,\eta}(t')\|_{\ell,k} \big\|_{L^r_{t'}} \left\|\frac{e^{-\sigma t}}{t^\frac12}\right\|_{L^{r'}_{t}} + \eta  \left\|\frac{e^{-\sigma t}}{t^\frac12}\right\|_{L^{1}_{t}}\Big)
\| h  \|_{\mathcal{X}_T^{\ell,k}}\\
& \lesssim \Big( {1 \over \lambda^\frac1r}+ \eta\Big) \| h  \|_{\mathcal{X}_T^{\ell,k}}\,.
\end{aligned}
$$
The estimate~(\ref{estimateYTell}) follows.
\medskip

\noindent $\bullet $ $ $ {\it The case of $\R^3$.} $ $  From the third estimate in Lemma~\ref{lem:decayWeps} combined with~\eqref{estimUkai1}-\eqref{estimUkai2}, we have: 
$$
\|W^\e(t) f\|_{H^\ell_xL^2_v} \lesssim  {1 \over {t^\frac12\langle t \rangle^{\frac 34}}} \, \left(\|f\|_{H^\ell_xL^2_v} +\|f\|_{L^2_vL^1_x}\right)\,.
$$
We use that from~\eqref{estimUkai2}, we also have:
$$
  \|\Gamma(f_1,f_2)\|_{{L^2_vL^1_x}}\lesssim \| f_1 \|_{{\ell,k}} \|  f_2\|_{{\ell,k}} \, .
$$
 We can thus conclude as in the case of the torus, we write $\displaystyle \frac1r+\frac1{r'} = 1$ with $1\leq r'<4/3$, then $t \mapsto t^{-\frac12} \langle t \rangle^{-\frac 34}$ is in $L^{r'}(\R^+)$ and we obtain~(\ref{estimateYTell}).

\medskip

\noindent  $\bullet $ $ ${\it The case of $\R^2$.}  
We start by noticing that 
for any $t \in \R^+$, we have thanks to~\eqref{estimUkai1}-\eqref{estimUkai2} and~\eqref{eq:decaydeltaepseta}:
\begin{equation} \label{eq:GammadeltaepsetaR2}
\forall \, h \in X^{\ell,k}, \quad \|\Gamma(\delta^{\eps,\eta},h)\|_{H^\ell_xL^2_v} + \|\Gamma(\delta^{\eps,\eta},h)\|_{L^2_vL^1_x} \lesssim \|\delta^{\eps,\eta}\|_{\ell,k} \|h\|_{\ell,k}\lesssim \frac{\eta}{\langle t \rangle^{\frac12}} \|h\|_{\ell,k} \, .\end{equation}
We also recall that the space~$\mathcal{X}_T^{\ell,k}$ involves a weight in time, namely~$\chi_\Omega(t)= \langle t \rangle^\frac14$. The part involving $\bar g^{\eps,\eta}$ is treated using~\eqref{estimUkai1} and the third estimate in Lemma~\ref{lem:decayWeps}. For the part with~$\delta^{\eps,\eta}$, we use~\eqref{estimUkai1},~\eqref{eq:GammadeltaepsetaR2} and the second estimate given in Lemma~\ref{lem:decayWeps}, 
we deduce 
$$\begin{aligned}
 &  \big \|    \mathcal{L}_\lambda^{\e}(t) h\big\|_{H^\ell_xL^2_v} \lesssim  \int_0^t \frac{1}{(t-t')^\frac12\langle t-t' \rangle^\frac12}  e^{- \lambda \int_{t'}^t \|\overline g^{\eps,\eta}(t'')\|_{\ell,k}^r\, dt''}  
\|\overline g^{\eps,\eta}\|_{\ell,k}  \, \|h(t')\|_{\ell,k} \, dt' \\
&\qquad + \int_0^t \frac{1}{(t-t')^\frac12} \frac{\eta}{\langle t' \rangle^{\frac12}} \, \|h(t')\|_{\ell,k} \, dt'\\
& \lesssim   \int_0^t \frac{1}{(t-t')^\frac12\langle t-t' \rangle^\frac12}  e^{- \lambda \int_{t'}^t \|\overline g^{\eps,\eta}(t'')\|_{\ell,k}^r\, dt''} \|\overline g^{\eps,\eta}\|_{\ell,k}
{dt' \over {\langle t' \rangle^{\frac14}}} \,  \|h\|_{\mathcal{X}_T^{\ell,k}} \\
&\qquad + { \eta \int_0^t \frac{1}{(t-t')^\frac12 \langle t-t' \rangle^\frac12 }
{dt' \over {\langle t' \rangle^{\frac34}}} \,  \|h\|_{\mathcal{X}_T^{\ell,k}} } \, .
\end{aligned}
$$
 By H\"older's inequality, writing $\displaystyle \frac1r+\frac1{r'} = 1$ with~$1<r'<4/3$, we have 
\begin{equation} \label{LlambdaHell}
\begin{aligned} 
&\langle t \rangle^{1 \over 4} \big \|    \mathcal{L}_\lambda^{\e}(t) h\big\|_{H^\ell_x L^2_v} \\
&\quad \lesssim  \Big\|   e^{- \lambda \int_{t'}^t \|\overline g^{\eps,\eta}(\sigma)\|_{\ell,k}^r\, d\sigma } \|\overline g^{\eps,\eta}\|_{\ell,k} \Big\|_{L^{r}_{t'}} \|h\|_{\mathcal{X}_T^{\ell,k}}\,   I_1(t,0,t)^{\frac1{r' }}\,  {+ \eta \, \|h\|_{\mathcal{X}_T^{\ell,k}}\, I_2(t,0,t)}
\end{aligned}
 \end{equation}
 with
 $$
 I_1(t,s,\tau):= \langle t \rangle^{\frac{r'}4} \int_s^\tau \frac{1}{(t-t')^{\frac{r'}2}\langle t-t' \rangle^{\frac{r'}2}}{1 \over \langle t' \rangle^{\frac{r'}4}} \, dt'
 $$
 and 
 $$
 {I_2(t,s,\tau) := \langle t \rangle^{\frac 14} \int_s^\tau \frac{1}{(t-t')^{\frac{1}2}\langle t-t' \rangle^\frac12}{1 \over \langle t' \rangle^{\frac{3}4}} \, dt' \, .}
 $$
 Let us now show that~$ I_1(t,0,t)$  is uniformly bounded in $t \ge 0$. We first notice that 
 \begin{align*}
 I_1(t,0,t/2) \lesssim \int_{t/2}^t \frac{1}{(t')^{\frac{r'}2}\langle t' \rangle^{\frac{r'}4}} {1 \over \langle t-t' \rangle^{\frac{r'}4}} \,  dt' \, .
 \end{align*}
 If $t \lesssim 1$, we can write the following bound: 
 $$
 I_1(t,0,t/2)
 \lesssim \int_0^1 \frac{dt'}{(t')^{\frac{r'}2}} < \infty
 $$
 because $r'<2$ and if $t \gtrsim 1$, we have:
$$
I_1(t,0,t/2) \lesssim \left\| {1 \over {\langle t \rangle^{\frac{3r'}4}}} \right\|_{L^{\frac43}} \left\| {1 \over {\langle t \rangle^{\frac{r'}4}}} \right\|_{L^4 } < \infty
$$
 since $r'>1$. 
On the other hand, we also have that 
 \begin{align*}
I_1(t,t/2,t) \lesssim \int_0^\infty \frac{dt'}{(t-t')^{\frac{r'}2}\langle t-t' \rangle^{\frac{r'}2}} < \infty
 \end{align*}
 because $r' \in (1,2)$. {Concerning $I_2$,  let us also write~$ I_2(t,0,t)=  I_2(t,0,t/2)+ I_2(t,t/2,t)$ and estimate both terms separately. The second one is the easiest since
 $$
\begin{aligned} I_2(t,t/2,t) &= \int_{ t/2}^t \frac{\langle t \rangle^{\frac14} }{(t-t')^{\frac12} \langle t-t' \rangle^\frac12}{dt' \over \langle t' \rangle^{\frac34}} &\lesssim  \int_{ t/2}^t \frac{1}{(t-t')^{\frac12}}{dt' \over \langle t' \rangle^{\frac12}} 
\lesssim {1 \over t^{\frac12}} \int_{0}^{t/2} \frac {dt'}{(t')^{\frac12}} \lesssim 1 \\
\end{aligned} $$
where we used the fact that $t/2 \le t' \le t$.
For the first term, we start by assuming that~$t\lesssim 1$, then
$$
\begin{aligned} 
I_2(t,0,t/2)
&\lesssim  \int_0^{t/2} \frac{dt'}{\sqrt{t-t'}} < \infty\, .
\end{aligned}$$
On the other hand if~$t\gtrsim 1$
then using the fact that when~$0\leq t'\leq t/2$ then~$t/2\leq t-t'\leq t$ we have
$$
 I_2(t,0,t/2) \lesssim {1 \over  t^{\frac14}} \int_0^{t/2} \frac{dt'}{\langle t'\rangle^\frac34} \lesssim 1 \, . 
$$
Coming back to~\eqref{LlambdaHell}, we thus conclude to~(\ref{estimateYTell}).

\smallskip
\noindent{\it{Step 3.}} $ $
Due to~(\ref{estimateLHlk}), it remains to estimate $\|\mathcal{L}_\lambda^{\e,1}(t)h\|_{\mathcal{X}_T^{\ell,j} }$ for~$0 \le j \le k$  given. The proof is similar to the one of Lemma~\ref{lem:Psiepscont}: we use the explicit form of~$e^{tA^\e}$ given by~\eqref{semigroupAeps} in order to deduce that 
\begin{align*}
&\chi_\Omega(t)\|\mathcal{L}_\lambda^{\e,1}(t)h\|_{{\ell,j}} \\
&\le 
\chi_\Omega(t)\left\| {2 \over \e} \int_0^t e^{-\nu(v)\frac{t-t'}{\e^2}} \|\Gamma( \overline g^{\eps,\eta}-\delta^{\eps,\eta} ,h)(t')\|_{H^\ell_x} e^{-\lambda \int_{t'}^t \|\overline g^{\eps,\eta}(t'')\|^r_{\ell,k} \, dt''} \, d{t'} \right\|_{L^{\infty,j}_v}\, .
\end{align*}

\noindent  $\bullet $ $ $ {\it The case of $\T^d$, $d = 2,3$ and $\R^3$.} As in the proof of Lemma~\ref{lem:Psiepscont} (see~\eqref{eq:Psiepsellj}), we obtain that for all~$\lambda \geq 0$
\begin{align*}
&\|\mathcal{L}_\lambda^{\e,1}(t)h\|_{{\ell,j}} 
\lesssim \e  \|\Lambda^{-1}  \Gamma (\overline g^{\eps,\eta}-\delta^{\eps,\eta},h)  \|_{\mathcal{X}_T^{\ell,k} }\,.
\end{align*}
Using~\eqref{estimUkai1}, we   have that 
\begin{equation} \label{eq:Lambda-1Gamma}
\|\Lambda^{-1} \Gamma (\overline g^{\eps,\eta}-\delta^{\eps,\eta},h)  \|_{\mathcal{X}^{\ell,k}_T} \lesssim \|\overline g^{\eps,\eta}-\delta^{\eps,\eta}\|_{L^\infty([0,T], X^{\ell,k})} \|h\|_{\mathcal{X}^{\ell,k}_T}\,. 
\end{equation}
From this, we conclude that for any $0 \le j \le k$,
$$
\|\mathcal{L}_\lambda^{\e,1}(t)h\|_{\mathcal{X}_T^{\ell,j} } \le C\e \|\overline g^{\eps,\eta}-\delta^{\eps,\eta}\|_{L^\infty([0,T], X^{\ell,k})}  \|h\|_{\mathcal{X}_T^{\ell,k} }\,.
$$
Thanks  to Lemma~\ref{resultstildegeps}-\eqref{estimateLpXtildeg} and~(\ref{estimatedeltaepseta0}), recalling that~$\overline g^{\eps,\eta} =   g^{\eta}+\overline \delta^{\eps,\eta}$, we deduce that
\begin{align*}
\|\mathcal{L}_\lambda^{\e,1}(t)h\|_{\mathcal{X}_T^{\ell,j} } &\le C\e \big(\|{g^\eta } \|_{L^\infty([0,T], X^{\ell,k})} +C\big) \|h\|_{\mathcal{X}_T^{\ell,k} }\,.
\end{align*}

\noindent $\bullet $ $ $  {\it The case of $\R^2$.}
We have $\chi_\Omega(t) = \langle t \rangle^{\frac14} $. Exactly as in the proof of Lemma~\ref{lem:PsiepscontR2}, we can write the following bound for~$0\leq j \leq k$:
\begin{align*}
&\langle t \rangle^{\frac14} \|\mathcal{L}_\lambda^{\e,1}(t)h\|_{{\ell,j}} 
\lesssim \e  \|\Lambda^{-1}  \Gamma (\overline g^{\eps,\eta}-\delta^{\eps,\eta},h)  \|_{\mathcal{X}_T^{\ell,k} }\,.
\end{align*}
As previously, we conclude that for any $0 \le j \le k$,
\begin{align*}
\|\mathcal{L}_\lambda^{\e,1}(t)h\|_{\mathcal{X}_T^{\ell,j} } 
&\le C\e \big(\|{g^\eta} \|_{L^\infty([0,T], X^{\ell,k})}  +C\big) \|h\|_{\mathcal{X}_T^{\ell,k} }\,,
\end{align*}
this concludes the proof of Proposition~\ref{linear}. \qed

 \subsection{Proof of Proposition~\ref{nonlinear}}   The proof  of Proposition~\ref{nonlinear}
 is an immediate consequence of the computations leading to Lemma~\ref{lem:Psiepscont}, bounding the exponential~$e^{\lambda \int_0^t \|\overline g^{\eps,\eta}(t')\|_{\ell,k}^r\, dt' }
$ by a constant.  %
 \qed

\subsection{Proof of Proposition~\ref{limitdatawp}} \label{sec:Deps}
Let us write
$$
 \mathcal{D}_\lambda^{\e} :=e^{-\lambda \int_0^t \|\overline g^{\eps,\eta}(t')\|_{\ell,k}^r\, dt' \, 
} \sum_{j=1}^4 {\mathcal D} ^{\eps,j} \, , 
$$
with
$$
\begin{aligned}
 {\mathcal D} ^{\eps,1}(t) &  :=   \widetilde \delta^{\eps,\eta} + \big( U^\eps (t)   - U (t)\big){\bar g_{{\rm in}}^\eta} \, , \\
 {\mathcal D} ^{\eps,2}(t)  & :=    \big(\Psi^\eps (t) -  \Psi(t) \big)  (g^\eta   ,g^\eta    ) \, ,\\
 {\mathcal D} ^{\eps,3}(t) &  :=2 \Psi^\eps (t) \Big({g^\eta+\frac 12 \overline\delta^{\eps,\eta}}  , \overline\delta^{\eps,\eta}\Big)  \, , \\
 {\mathcal D} ^{\eps,4}(t) &  :=- 2\Psi^\eps (t) (\delta^{\e,\eta} , \overline\delta^{\eps,\eta}) +\Psi^\eps (t) 
 \big( \delta^{\eps,\eta} -2{g^\eta} , \delta^{\eps,\eta} 
 \big) 
\, .
\end{aligned}
$$ 
We shall prove that
$$
\forall \,j \in [1,4] \, , \quad \lim_{\eta\to0} \lim_{\eps\to0} \| {\mathcal D} ^{\eps,j}\|_{\mathcal{X}_T^{\ell,k}} = 0 \, , 
$$
uniformly in~$T$ if~$\bar g^\eta_{\rm{in}}$ generates a global solution. 

\medskip
The result on~$ {\mathcal D}^{\eps,1}$ 
follows from Lemma~\ref{lem:UepstoU0}.
Indeed, recalling that by definition
$$
\widetilde \delta^{\eps,\eta}= U^\eps(t) (g_{\rm in}^\eta - \bar g^\eta_{{\rm in}}) - \overline\delta^{\eps,\eta}
\quad \mbox{and}\quad \overline\delta^{\eps,\eta} = U^\eps_{\rm{disp}}(t) g^\eta_{{\rm in}} + U^{\eps \sharp}(t)g^\eta_{{\rm in}} $$
 we have
$$
 {\mathcal D} ^{\eps,1}(t) =  \big( U^\eps (t)   - U^\eps_{\rm{disp}}(t)  - U^{\eps \sharp}(t)- U (t)\big)  g^\eta_{{\rm in}}  + U (t) (g_{\rm in}^\eta - \bar g^\eta_{{\rm in}}) 
$$
and we conclude that
$$
\lim_{\eps\to0} \| {\mathcal D} ^{\eps,1}\|_{\mathcal{X}_\infty^{\ell,k}} = 0 
$$
 from the fact that
{$$
 U (t) (g_{\rm in}^\eta - \bar g^\eta_{{\rm in}})  = 0 \, ,
$$
which comes from~(\ref{formulaU(0}) and~(\ref{U(t)U(0}) and thanks to~(\ref{bargeta})-(\ref{bargetageta}) which imply that $\bar g^\eta_{{\rm in}} = U(0) g^\eta_{{\rm in}}$.} 

\medskip
Now let us concentrate on~$  {\mathcal D} ^{\eps,2}$, the control of which follows from the following lemma.
\begin{lem}\label{estimatePsiepslimitPsi0Xellk}
   Let~$\ell>d/2$ and $k>d/2+1$ be given and consider a function~$g$ solving the limit system on~$[0,T]$ with initial data in~$X^{\ell,k}$ then  
 $$
 \begin{aligned}
\lim_{\e \to 0}\| \Psi^\eps(t) (g,g)- \Psi(t) (g,g)\|_{{\mathcal X}_T^{\ell,k}} = 0\, ,
 \end{aligned}$$
 uniformly in~$T$ if~$g$ is a global solution.
\end{lem}
\noindent {\it Proof of Lemma~{\rm\ref{estimatePsiepslimitPsi0Xellk}}.}
A large part of the proof is dedicated to the case~$\Omega = \R^2$ which is   the most intricate one. Then, we conclude the proof by describing the slight changes that need to be made to address the other cases. 

\noindent $\bullet $ $ $ {\it The case of $\R^2$.} We recall that {$\chi_\Omega(t) = \langle t \rangle^{\frac14} $}. 
We start with the decomposition~(\ref{decompositionPsiappendix}) and we deal with each term  in succession, the most delicate one being~$\Psi^\eps_{j0}$ for~$j\in\{1,2\}$. So let us set~$j\in\{1,2\}$.  Defining
$$
H_j(t,t',x):=  \mathcal{F}_x^{-1} \left(e^{- \beta_j (t-t') |\xi|^2} |\xi|P_{j}^1\Big({\xi \over |\xi|}\Big) \widehat {\Gamma}(g,g)(t',\xi)\right)
$$
an integration by parts in time provides (recalling that~$\alpha_j>0$)
\begin{equation}\label{decomposepsij0}
\begin{aligned}
\mathcal{F}_x \left( \Psi^\eps_{j0}(t)(g,g)\right)(\xi) =\frac{\eps}{i \alpha_j|\xi|} &\Big(\int_0^t e^{i\alpha_j |\xi|{{t-t'}\over \eps}}\partial_{t'} \widehat H_j(t,{t'},\xi)\,d{t'}\\
&\quad
- \widehat H_j(t,t,\xi)+ e^{i\alpha_j |\xi| {t \over \eps}}  \widehat H_j(t,0,\xi)
\Big)\,.
\end{aligned}
\end{equation}
The first term on the right-hand side may be split into two parts:
$$
\frac{\eps}{i \alpha_j|\xi|}  \int_0^t e^{i\alpha_j |\xi|\frac{t-t'}{\eps}}\partial_{t'}  \widehat H_j(t,{t'},\xi)\,d{t'}= \widehat H^1_j(t,\xi)+ \widehat H^2_j(t,\xi)
$$
with
$$
\begin{aligned}
H^1_j(t,x)&:= \mathcal{F}_x^{-1} \left(\frac{\eps\beta_j}{i \alpha_j } \int_0^t e^{i\alpha_j |\xi|\frac{t-t'}{\eps}} |\xi|\widehat H_j(t,{t'},\xi)\, d{t'}\right)
\quad \mbox{and}\\
H^2_j(t,x)&:= \mathcal{F}_x^{-1} \left(\frac{\eps }{i \alpha_j } \int_0^t e^{i\alpha_j |\xi|\frac{t-t'}{\eps} - \beta_j (t-{t'}) |\xi|^2}P_{j}^1\Big({\xi \over |\xi|}\Big)\partial_{t'}\widehat {\Gamma}(g,g)({t'},\xi)
\, d{t'}\right)\,.
\end{aligned}
$$
Since~$P_{j}^1(\xi/|\xi|)$  is bounded from $L^2_v$ into $L^{\infty,k}_v$ uniformly in~$\xi$ from Lemma~\ref{EllisPinsky}, we have 
\begin{align*}
\|H^1_j(t)\|_{\ell,k}^2&\lesssim\eps^2 \int_{\R^2}  \langle \xi \rangle^{2\ell} \left( \int_0^t e^{- \beta_j (t-{t'}) |\xi|^2}|\xi|^2 \big\| \widehat {\Gamma}(g,g)({t'},\xi)\big\|_{L^2_v}\,d{t'}\right)^2 \, d\xi
\,.
\end{align*}
Using now Young's inequality in time, we infer that for all~$t\lesssim 1$
\begin{align*}
\|H^1_j(t)\|^2_{\ell,k} 
&\lesssim \eps^2  \int_{\R^2} \langle \xi \rangle^{2\ell} \big\| \widehat {\Gamma}(g,g)({t},\xi)\big\|_{L^\infty_{t}L^2_v} ^2  \, d\xi  \lesssim \eps^2
\end{align*}
from Lemma~\ref{lem:gamma&tildegamma}-$(i)$. 
For the case $t \gtrsim 1$, we write 
\begin{align*}
&{t^\frac12}\|H^1_j(t)\|^2_{\ell,k} \\
&\quad \lesssim \eps^2 \int_{\R^2}  \langle \xi \rangle^{2\ell} \left(\int_0^{t/2} {(t-{t'})^{\frac14} |\xi|^\frac12} e^{-\beta_j|\xi|^2 \frac{t-t'}{2}}  e^{-\beta_j|\xi|^2 \frac{t-t'}{2}} {|\xi|^\frac12}|\xi|  {\gamma_1}({t'},\xi) \, d{t'} \right)^2  d\xi\\
&\qquad + \eps^2 \int_{\R^2} \langle \xi \rangle^{2 \ell}\left(\int_{t/2}^{t} e^{-\beta_j(t-{t'})|\xi|^2} |\xi|^2 {{t'}^\frac14}  {\gamma_1}({t'},\xi) \, d{t'} \right)^2  d\xi
\, , \end{align*}
where to simplify notation we have set
$$
\gamma_1(t,\xi):=\big\| \widehat {\Gamma}(g,g)({t},\xi)\big\|_{L^2_v}\, .
$$
Then
we use Young's inequality in time as well as the fact that $t^{\frac14} |\xi| ^{\frac12} e^{- \beta_j t |\xi|^2}$ is uniformly bounded to obtain
\begin{align*}
&\int_{\R^2}  \langle \xi \rangle^{2\ell} \left(\int_0^{t/2} {(t-{t'})^{\frac14} |\xi|^\frac12} e^{-\beta_j|\xi|^2 \frac{t-t'}{2}}  e^{-\beta_j|\xi|^2 \frac{t-t'}{2}} {|\xi|^\frac12}|\xi|  {\gamma_1}({t'},\xi) \, d{t'}\right)^2 d\xi\\
&\quad \lesssim  \int_{\R^2} \langle \xi \rangle^{2\ell} \big\|  |\xi| {\gamma_1}(t, \xi)\big\|_{L^\frac43_t}^2 \, d\xi\\
&\quad \lesssim 
{\big\| |\xi| \gamma_1(t,\xi) \big\|_{L^\frac43_t L^2_\xi(\langle \xi \rangle^\ell)}^2}
\end{align*}
by Minkowski's inequality. Similarly
$$  \int_{\R^2} \langle \xi \rangle^{2 \ell}\left(\int_{t/2}^{t}  |\xi|  e^{-\beta_j(t-{t'})|\xi|^2} {{t'}^\frac14}   |\xi| {\gamma_1}({t'},\xi) \, d{t'} \right)^2  d\xi   \lesssim   { \big\|{t^{\frac14}} |\xi| \gamma_1(t,\xi)\big\|^2_{L^2_t L^2_\xi(\langle \xi \rangle^\ell)} }\, ,
$$
from which we conclude, using Lemma~\ref{lem:gamma&tildegammaR2}-$(i)$, that 
$$
{t^{\frac14}} \|H^1_j(t)\|_{\ell,k} \lesssim\eps\,  .
$$
Concerning $H^2_j(t)$, again since~$P_{j}^1(\xi/|\xi|)$  is bounded from $L^2_v$ into $L^{\infty,k}_v$ uniformly in~$\xi$, there holds
 $$
\begin{aligned}
\|H^2_j(t)\|^2_{\ell,k} &\lesssim\eps^2 \int_{\R^2}  \langle \xi \rangle^{2\ell} \left( \int_0^t e^{- \beta_j (t-{t'}) |\xi|^2} \big\| \partial_{t'} \widehat {\Gamma}(g,g)({t'},\xi)\big\|_{L^2_v}\,d{t'} \right)^2 d\xi 
\, .
\end{aligned}
$$
For $t \lesssim 1$, we separate low and high frequencies. We use  again Young's inequality in time: defining for simplicity
$$
{\gamma_2}({t},\xi) :=  \|  \partial_{t}  \widehat  {\Gamma}(g,g)(t,\xi) \|_{L^2_v} \, ,
$$
there holds
\begin{align*}
\|H^2_j(t)\|^2_{\ell,k} 
&\lesssim \eps^2  \int_{|\xi| \le 1} \left( \int_0^t |\xi|^\frac12  e^{- \beta_j (t-{t'}) |\xi|^2} |\xi|^{-\frac12}  \big\| \partial_{t'} \widehat {\Gamma}(g,g)({t'},\xi)\big\|_{L^2_v} \, d{t'} \right)^2 d\xi \\
&\quad + \eps^2 \int_{|\xi| \ge 1} \langle \xi \rangle^{2\ell} \left( \int_0^t |\xi| e^{- \beta_j (t-{t'}) |\xi|^2} |\xi|^{-1}   \big\| \partial_{t'} \widehat {\Gamma}(g,g)({t'},\xi)\big\|_{L^2_v} \, d{t'} \right)^2  d\xi \\
&\lesssim \eps^2 {\Big(\big\|   |\xi|^{-\frac12}\gamma_2(t,\xi) \big\|_{L^2_\xi L^\frac43_t} ^2 +\big\| \gamma_2(t,\xi)\big\|^2_{L^2_tL^2_\xi(\langle \xi \rangle^{\ell-1})} \Big)} \, .
\end{align*}
From Minkowski's inequality followed by the Sobolev embedding~$L^\frac43(\R^2)\subset \dot H^{-\frac12}(\R^2)$, we have 
$$
\begin{aligned}
  \big\| |\xi|^{-\frac12}   \gamma_2(t,\xi) \big\|_{L^2_\xi L^\frac43_t} & \lesssim \big\|  |\xi|^{-\frac12}\gamma_2(t,\xi)\big\|_{L^\frac43_tL^2_{\xi} }\\
   & \lesssim \big\| \mathcal F_x^{-1}(\gamma_2)\|_{L^\frac43_{t,x}}
\end{aligned}
$$
which is bounded from Lemma~\ref{lem:gamma&tildegammaR2}-$(ii)$,  as well as {$\big\|        \gamma_2(t,\xi)\big\|_{L^2_tL^2_\xi(\langle \xi \rangle^{\ell-1})} $} from Lemma~\ref{lem:gamma&tildegamma}-$(ii)$.
We now focus on the case~$t \gtrsim 1$, separating again the integral into low and high frequencies and separating times in $(0,t/2)$ and in $(t/2,t)$: 
there holds
\begin{align*}
&{t^\frac12}\|H^2_j(t)\|^2_{\ell,k}  \\
&\quad \lesssim \eps^2 {t^\frac12}\int_{\R^2} \left(\int_0^t e^{-\beta_j(t-{t'})|\xi|^2}{\gamma_2}({t'},\xi) \, d{t'} \right)^2  d\xi \\
& \quad \lesssim \eps^2 \int_{|\xi|\le 1}  \left(\int_0^{t/2} {(t-{t'})^{\frac14} |\xi|^\frac12} e^{-\beta_j(t-{t'})|\xi|^2} {|\xi|^{-\frac12}}{\gamma_2}({t'},\xi)  \, d{t'} \right)^2  d\xi \\
&\qquad + \eps^2 \int_{|\xi| \ge 1} \langle \xi \rangle^{2\ell} \left(\int_0^{t/2} (t-{t'})^{\frac14} e^{-\beta_j\frac{t-t'}{2}} e^{-\beta_j|\xi|^2 \frac{t-t'}{2}} |\xi| |\xi|^{-1} {\gamma_2}({t'},\xi)   \, d{t'} \right)^2  d\xi \\
&\qquad + \eps^2 \int_{|\xi|\le 1} \left(\int_{t/2}^{t} e^{-\beta_j(t-{t'})|\xi|^2} {|\xi|^{\frac12} |\xi|^{-\frac12}{t'}^{\frac14}}  {\gamma_2}({t'},\xi) \, d{t'} \right)^2  d\xi \\
&\qquad + \eps^2 \int_{|\xi| \ge 1}\langle \xi \rangle^{2\ell}  \left(\int_{t/2}^{t} e^{-\beta_j(t-{t'})|\xi|^2} |\xi| |\xi|^{-1} {t'}^{\frac14} {\gamma_2}({t'},\xi) \, d{t'} \right)^2 d\xi =: I_1+I_2+I_3+I_4. 
\end{align*}
For $I_1$, we introduce $1/2<b\le3/4$. From the Cauchy-Schwarz inequality in time and using the fact that $t^{\frac14} |\xi| ^{\frac12} e^{- \beta_j t |\xi|^2}$ is uniformly bounded, we get: 
$$
I_1 \lesssim \eps^2 \|\langle t \rangle^{b} {\mathcal F}_x^{-1} (\gamma_2) \|^2_{L^2_t \dot H^{-\frac12}_x} \lesssim \eps^2 \|\langle t \rangle^{b}  {\mathcal F}_x^{-1} (\gamma_2) \|^2_{L^2_t L^{\frac43}_x}
$$ 
using again the Sobolev embedding~$L^\frac43(\R^2)\hookrightarrow \dot H^{-\frac12}(\R^2)$. We deduce that $I_1 \lesssim \eps^2$ from Lemma~\ref{lem:gamma&tildegammaR2}-$(ii)$. 
The second term can be bounded as follows using Young's inequality in time, as well as the fact that~$t^{\frac14}  e^{- \beta_j t }$ is uniformly bounded:
$$
I_2 \lesssim \eps^2 \int_0^\infty \int_{\R^2}  \langle \xi \rangle^{2(\ell-1)} \gamma_2^2(t',\xi) \, d\xi \, dt'
\lesssim \eps^2{ \big \| \gamma_2(t,\xi) \big\|^2_{L^2_t L^2_\xi(\langle \xi \rangle^{\ell-1})}}
$$
and thus $I_2 \lesssim \eps^2$ from Lemma~\ref{lem:gamma&tildegammaR2}-$(ii)$. For $I_3$, we use first Young's inequality in time to get
$$
I_3 \lesssim \eps^2 \int_{\R^2} \Big\| |\xi|^{-\frac12} t'^{\frac14}  {\gamma_2}(t',\xi) \Big\|_{L^{\frac43}_{t'}}^2 d\xi \, .
$$
Then, again from Minkowski's inequality and the Sobolev embedding~$L^\frac43(\R^2)\hookrightarrow \dot H^{-\frac12}(\R^2)$, we obtain 
$$
I_3 \lesssim \eps^2 \bigg(\int_0^\infty t^{\frac13}{\big \| \mathcal{F}_x^{-1}(\gamma_2) (t,x)\big \|_{\dot H^{-\frac12}_x}^{\frac43}} \, dt \bigg)^\frac32\lesssim \eps^2 {\big\| t^{\frac14} \mathcal{F}_x^{-1}(\gamma_2)(t,x)\big\|^2_{L^{\frac43}_{t,x}}}
$$
so that $I_3 \lesssim \eps^2$ still from Lemma~\ref{lem:gamma&tildegammaR2}-$(ii)$. For the last term $I_4$, we use Young's inequality in time to obtain:
$$
I_4\lesssim \eps^2 \int_0^\infty \int_{\R^2}  \langle \xi \rangle^{2(\ell-1)} \langle t' \rangle^\frac12 \gamma_2^2(t',\xi) \, d\xi \, dt'
\lesssim \eps^2{\big  \|\langle t \rangle^{\frac14}\gamma_2(t,\xi) \big \|_{L^2_t L^2_\xi(\langle \xi \rangle^{\ell-1})}}
$$
and thus $I_4 \lesssim \eps^2$ from Lemma~\ref{lem:gamma&tildegammaR2}-$(ii)$.

Now let us turn to the two other contributions in~(\ref{decomposepsij0}). There holds
$$
\begin{aligned}
& \frac{\eps}{| \alpha_j|\, |\xi|}   \Big( |\widehat H_j(t,t,\xi)|+ \big|e^{i \alpha_j|\xi|{t \over \eps}} \widehat H_j(t,0,\xi) \big| \Big) \\
&\quad \lesssim 
 \eps  \Big| P_{j}^1\Big({\xi \over |\xi|}\Big) \widehat {\Gamma}(g,g)(t,\xi) \Big| +\eps \, e^{- \beta_j t |\xi|^2}    \Big| P_{j}^1\Big({\xi \over |\xi|}\Big)  \widehat {\Gamma}(g,g)(0,\xi) \Big| \\
 &\quad \lesssim \eps  \left(\Big| P_{j}^1\Big({\xi \over |\xi|}\Big) \widehat {\Gamma}(g,g)(t,\xi) \Big| +  e^{- \beta_j t |\xi|^2} \Big| P_{j}^1\Big({\xi \over |\xi|}\Big) \widehat {\Gamma}(g,g)(0,\xi) \Big| \right)\\
 & \quad =: \eps \left( \big|\widehat {\tilde H_j} (t,\xi)\big| + e^{- \beta_j t |\xi|^2} \big|\widehat {\tilde H_j}(0,\xi)\big|\right) \,. 
 \end{aligned}
$$
We have for any $t \lesssim 1$, still using that~$P_{j}^1(\xi/|\xi|)$  is bounded from $L^2_v$ into $L^{\infty,k}_v$:
\begin{align*}
\|   {\tilde H_j}  (t) \|^2_{\ell,k} 
&\lesssim \sup_v \, \langle v \rangle^{2k} \int_{\R^2}  \langle \xi \rangle^{2\ell} \Big| P_{j}^1\Big({\xi \over |\xi|}\Big) \widehat {\Gamma}(g,g)(t,\xi) \Big|^2  d\xi \\
&\lesssim \int_{\R^2}  \langle \xi \rangle^{2\ell} \Big\| P_{j}^1\Big({\xi \over |\xi|}\Big) \widehat {\Gamma}(g,g)(t,\xi) \Big\|^2_{L^{\infty,k}_v} \, d\xi \\
&\lesssim \|\gamma_1(t,\xi)\|_{L^\infty_t L^2_\xi(\langle \xi \rangle^\ell)} \, , 
\end{align*}
and this quantity in uniformly bounded in time from Lemma~\ref{lem:gamma&tildegamma}-$(i)$. In the case when~$t \gtrsim 1$, we simply write
$$
t^{\frac12} \|\tilde H_j(t ) \|^2_{\ell,k} \lesssim t^{\frac12} {\|\gamma_1(t,\xi) \|_{L^2_\xi(\langle \xi \rangle^\ell)}^2}
$$
which is uniformly bounded in time thanks to Lemma~\ref{lem:gamma&tildegammaR2}-$(iii)$. Then, similarly, we write that 
$$
t^{\frac12}\|e^{- \beta_j t |\xi|^2}\tilde H_j(0) \|^2_{\ell,k} \lesssim t^{\frac12}  \int_{\R^2} e^{-2\beta_j t |\xi|^2} \langle \xi \rangle^{2\ell} \gamma_1^2(0,\xi) \, d\xi \, .
$$
Using the fact that $t^{\frac12} |\xi| e^{- 2\beta_j t |\xi|^2}$ is uniformly bounded, we obtain
\begin{align*}
t^{\frac12}\|e^{- \beta_j t |\xi|^2}\tilde H_j(0) \|^2_{\ell,k} &\lesssim \int_{\R^2} |\xi|^{-1} \langle \xi \rangle^{2\ell}  \gamma_1^2(0,\xi) \, d\xi \\
&\lesssim {\big\| \mathcal{F}_x^{-1}(\gamma_1)(0,x) \big\|_{\dot H^{-\frac12}_x} ^2+ \big\| \gamma_1(0,\xi)\big\|_{L^2_\xi(\langle \xi \rangle^\ell)}^2}
\end{align*}
from which we deduce, using again the Sobolev embedding~$L^\frac43(\R^2)\hookrightarrow \dot H^{-\frac12}(\R^2)$, that
$$
t^{\frac12}\|e^{- \beta_j t |\xi|^2}\tilde H_j(0,\xi) \|^2_{\ell,k} \lesssim{\big \|\mathcal{F}_x^{-1}(\gamma_1)(t,x)\big\|_{L^\infty_t L^{\frac43}_x}^2 +\big \| \gamma_1(t,\xi)\big\|_{L^\infty_t L^2_\xi(\langle \xi \rangle^\ell) }^2}\, .
$$
The last inequality yields the expected result thanks to Lemma~\ref{lem:gamma&tildegammaR2}-$(iii)$.

The terms~$ \Psi^{\eps\sharp}_{j0}$, $ \Psi^\eps_{j1}$ and~ $ \Psi^\eps_{j2}$ for $1 \le j \le 4$ are dealt with in a similar, though easier way. Indeed for~$ \Psi^{\eps\sharp}_{j0}$ we simply notice that thanks to~(\ref{estimatechi})
$$
\begin{aligned}
\big |\widehat\Psi^{\eps\sharp}_{j0}(t)(g,g)(\xi)\big| & \lesssim\Big|\chi\Big (\frac{\eps|\xi|}\kappa\Big)-1\Big| \int_0^te^{- \beta_j (t-t') |\xi|^2}|\xi|   \Big| P_{j}^1\Big({\xi \over |\xi|}\Big) \widehat  {\Gamma}(g,g)(t',\xi)  \Big| \, dt'\\
 & \lesssim \eps  \int_0^te^{- \beta_j (t-t') |\xi|^2}|\xi| ^2  \Big| P_{j}^1\Big({\xi \over |\xi|}\Big) \widehat  {\Gamma}(g,g)(t',\xi)  \Big| \, dt'
 \end{aligned}
$$
and the same estimates as above (see the term $H^1_j$) provide
$$
\big\| \Psi^{\eps\sharp}_{j0}(t)(g,g)\big\|_{\mathcal{X}_\infty^{\ell,k}}\lesssim  \eps\,.
$$
The term~$\Psi^\eps_{j1}$ is directly estimated by~(\ref{estimategamma2}):
$$
\big |\widehat\Psi^{\eps}_{j1}(t)(g,g)(\xi)\big|   \lesssim \eps  \int_0^te^{- \beta_j |\xi|^2{{t-t'} \over 4}}|\xi| ^2  \Big| P_{j}^1\Big({\xi \over |\xi|}\Big) \widehat  {\Gamma}(g,g)(t',\xi)  \Big| \, dt'
$$
so again
$$
\big\| \Psi^{\eps}_{j1}(t)(g,g)\big\|_{\mathcal{X}_\infty^{\ell,k}}\lesssim  \eps\, .
$$
Finally~$\Psi^\eps_{j2}$ is controlled in the same way thanks to the fact that~$P_{j}^2(\xi)$ is bounded from $L^2_v$ into $L^{\infty,k}_v$ uniformly in $|\xi| \le \kappa$ so there also holds
$$
\big\| \Psi^{\eps}_{j2}(t)(g,g)\big\|_{\mathcal{X}_\infty^{\ell,k}}\lesssim \eps\, .
$$
To end the proof of the proposition it remains to estimate~$ \Psi^{\eps \sharp}(t)(g,g)$ but this again is an easy matter. 
Indeed, {Lemma~\ref{lem:gamma&tildegamma}-$(i)$} and estimate~\eqref{decayUepssharp} imply that 
$$
\big\| \Psi^{\eps\sharp}(t)(g,g)\big\|_{\mathcal{X}_\infty^{\ell,k}} \lesssim \eps\, .
$$

\smallskip
\noindent  $\bullet $ $ $ {\it The cases of $\R^3$ and $\T^3$.} We recall that in those cases, $\chi_\Omega(t)=1$. All the terms can be treated using the same estimate as in the $\R^2$ case for $t \lesssim 1$ --- note that the Sobolev embedding~$L^\frac43(\R^2)\hookrightarrow \dot H^{-\frac12}(\R^2)$ must be replaced by the use of Lemma~\ref{lem:gamma&tildegamma}-$(ii)$.

\smallskip
\noindent  $\bullet $ $ $ {\it The case of $\T^2$.} We also have $\chi_\Omega(t)=1$ here. As previously, the terms $ \Psi^{\eps\sharp}_{j0}$, $ \Psi^\eps_{j1}$ and~ $ \Psi^\eps_{j2}$ for~$1 \le j \le 4$ can be treated exactly in the same way as for the small times case~$t \lesssim 1$ in the~$\R^2$ case. Concerning $\Psi^\eps_{j0}$ for~$j\in\{1,2\}$, the term $H^1_j$ can still be handled using the same estimate as in the $\R^2$ case for~$t \lesssim 1$ as well as the remaining terms $\tilde{H_j}(t)$ and $\tilde H_j(0)$. The only difference lies in the treatment of $H^2_j$ due to the special case of~$\xi=0$. Using Young's  inequality in time for the non-zero frequencies, we have 
\begin{align*}
&\|H^2_j(t)\|^2_{\ell,k} \\
&\quad \lesssim \eps^2 \sum_{\xi \in \Z^d} \langle \xi \rangle^{2\ell} \left( \int_0^t e^{-\b_j(t-t')|\xi|^2} {\gamma_2}({t'},\xi)\, dt' \right)^2 \\
&\quad \lesssim \eps^2 \left( \int_0^t {\gamma_2}({t'},0)\, dt' \right)^2 + \eps^2 \sum_{\xi \in \Z^d \setminus \{0\}} \langle \xi \rangle^{2\ell} \left( \int_0^t e^{-\b_j(t-t')|\xi|^2} |\xi| |\xi|^{-1} {\gamma_2}({t'},\xi)\, dt' \right)^2 \\
&\quad \lesssim \eps^2 \left( \int_0^t {\gamma_2}({t'},0)\, dt' \right)^2 + \eps^2 \sum_{\xi \in \Z^d \setminus \{0\}} \langle \xi \rangle^{2(\ell-1)}  \int_0^t   \gamma_2^2({t'},\xi)\, dt'  \, .
\end{align*}
For the first term in the right-hand side, we notice that 
$$
| {\gamma_2}({t'},0)| \lesssim 
{\|\mathcal{F}_x^{-1}(\gamma_2)(t)\|_{L^1_x}}
$$
which implies
$$
\eps^2  \left( \int_0^t {\gamma_2}({t'},0)\, dt' \right)^2\lesssim \eps^2
$$
from Lemma~\ref{lem:gamma&tildegammaT2}. 
The second term is also bounded by $\eps^2$ thanks to Lemma~\ref{lem:gamma&tildegamma}-$(ii)$, and this concludes the proof of Lemma~\ref{estimatePsiepslimitPsi0Xellk}.
\qed

\medskip

{Now let us turn to~$  {\mathcal D} ^{\eps,3}$. It vanishes in the well-prepared case and we control it thanks to Lemma~\ref{lem:Psiepsbardelta} in the ill-prepared case. The latter implies that for all~$\eta\in(0,1)$ 
$$
\lim_{\eps\to 0}  \Big\|\Psi^\eps(t)\Big(g^\eta + {1 \over 2} \overline\delta^{\eps,\eta},\overline\delta^{\eps,\eta}\Big)\Big\|_{\mathcal{X}^{\ell,k}_{T}} =0 \, . 
$$
Indeed, in both the $\R^2$ and $\R^3$ cases, we have that $g^\eta \in \mathcal{X}^{\ell,k}_T$ from Proposition~\ref{propNSF} and also that~$\lim_{\e\to0} \|\overline\delta^{\eps,\eta}\|_{\mathcal{X}^{\ell,k}_\infty} \le C_\eta$ from Lemma~\ref{resultstildegeps}. 

\smallskip
 Finally, in the cases $\Omega=\T^d$, $d=2,3$ and $\Omega=\R^3$ (where $\chi_\Omega(t)=1$),~$  {\mathcal D} ^{\eps,4}$ is very easily estimated thanks to the continuity bounds provided in Lemma~\ref{lem:Psiepscont}, along with~(\ref{estimatedeltaepseta0}),~\eqref{estimateLpXtildeg} and the fact that $g^\eta$ is uniformly bounded in time in $X^{\ell,k}$. Concerning the case $\Omega=\R^2$, we have to be more careful since $\delta^{\eps,\eta}$ is not bounded in $\mathcal{X}^{\ell,k}_\infty$. The result is a consequence of Lemma~\ref{lem:PsiepscontR2} combined with~\eqref{estimatebardeltaepseta0},~\eqref{estimatedeltaepseta0} and~\eqref{estimateLpXtildeg}.
}

\medskip

\noindent
 This ends the proof of Proposition~\ref{limitdatawp}. \qed

 \appendix
\section{Spectral decomposition for the linearized Boltzmann operator}\label{specBoltz}
 In this section we present a crucial spectral decomposition result  for the semigroup~$ U^\e(t)$ associated with the operator 
$$
  B^\e:={1 \over \e^2}(-\e v  \cdot \nabla_x +L)\, ,
$$
recalling that
$$ Lg  =M^{-\frac12}\big(
Q(M,M^\frac12 g) + Q(M^\frac12 g,M)\big)\,.
$$This theory is a key point to study the limits of $U^\eps(t)$ and $\Psi^\eps(t)$ as $\eps$ goes to~$0$.
We start by recalling a result from~\cite{Ellis-Pinsky} in which a Fourier analysis in $x$ on the semigroup~$U^1$ is carried out.  Roughly speaking, this result shows that the spectrum of the whole linearized operator can be seen as a perturbation of the homogeneous one.

Denoting $\mathcal{F}_x$ the Fourier transform in $x \in \R^d$ (resp. $x \in \T^d$) with $\xi \in \R^d$ (resp. $\xi \in \Z^d$) its dual variable, we write 
$$
U^\e(t) = \mathcal{F}_x^{-1} \widehat U^\e(t ) \mathcal{F}_x
$$
where $\widehat U^\e$ is the semigroup associated with the operator 
$$
\widehat B^\e:={1 \over \e^2}(-i\e \xi \cdot v +\widehat L)\, . 
$$
 In the following we denote by~$\chi$ a fixed, compactly supported function of the interval~$(-1,1)$, equal to one on~$[-\frac12 , \frac12]$.
 \begin{lem}[\cite{Ellis-Pinsky}] \label{EllisPinsky}
Let~$\ell\in \R$ and~$k>d/2+1$ be given.  
There exists $\kappa>0$ such that one can write
 \begin{equation} \label{decompUeps}\begin{aligned}
  U^\eps(t) &=\sum_{j=1}^4   U^\eps_j(t)  +   U^{\eps\sharp}(t)  \\
& \text{with} \quad \widehat U_j^\e(t,\xi): = \widehat U _j\Big({t \over \e^2},\e \xi\Big) \quad \text{and}\quad \widehat U^{\eps\sharp}(t,\xi):= \widehat U^{\sharp}\Big({t \over \e^2},\e \xi\Big) \, ,
\end{aligned}
\end{equation}
where for $1 \le j \le 4$,  
$$
\widehat U _j(t,\xi) := \chi\Big(\frac{|\xi|}\kappa\Big) \, e^{t \lambda_j(\xi)} P_j(\xi)
$$
with $\lambda_j \in \mathcal{C}^\infty( B(0,\kappa))$ satisfying
\begin{equation} \label{estimatelambdaj}
\begin{aligned}
\lambda_j(\xi) = i \alpha_j |\xi| - \beta_j|\xi|^2 + \gamma_j(|\xi|) \quad \text{as} \quad |\xi|\to 0
\, , \\
\alpha_1>0\, , \quad \alpha_2>0\, , \quad  \alpha_3=\alpha_4=0\,  , \quad \beta_j>0 \, , 
\\
\gamma_j(|\xi|) =O(|\xi|^3) \quad \text{and} \quad \gamma_j(|\xi|) \le \beta_j |\xi|^2/2\quad \text{for} \quad |\xi|\le \kappa\,  ,
\end{aligned}
\end{equation}
and 
$$
P_j(\xi) = P_{j}^0\left({\xi \over |\xi|}\right) +|\xi| P_{j}^1\left({\xi \over |\xi|}\right) + |\xi|^2 P_{j}^2(\xi)\, ,
$$
with  $P_{j}^n$     bounded linear operators on $L^2_v$ with operator norms uniform for $|\xi| \le \kappa$. {We also have that $P_{j}^n(\xi/|\xi|)$ is bounded from
$L^2_v$ into $L^{\infty,k}_v$ uniformly in~$\xi$. Moreover, if $j \neq n$, then we have that $P^0_j P^0_n=0$. We also have that the orthogonal  projector~$\Pi_{L}   $ onto $\operatorname{Ker} L$ satisfies
$$\Pi_{L}  = \sum_{j=1}^4 P_{j}^0\Big({\xi \over |\xi|}\Big)
$$
 and  is independent of $\xi/|\xi|$.}  Finally~$ \widehat U^\sharp$ satisfies
\begin{equation} \label{rateUsharp}
\| \widehat U^\sharp\|_{L^2_v \to L^2_v} \le C e^{-\alpha t}
\end{equation}
for some positive constants $C$ and $\alpha$ independent of $t$ and $\xi$.
\end{lem}
\begin{proof}
The decomposition of~$\widehat U^\eps(t)$ follows that of~$\widehat U^1(t)$: we recall that according to~\cite{Ellis-Pinsky}, one can write
$$
\widehat U^1(t,\xi) = \sum_{j=1}^4 \widehat U _j(t,\xi)  + \widehat U^{\sharp}(t,\xi)  \, , 
$$
where for $1 \le j \le 4$,  
$$
\widehat U _j(t,\xi) := \chi\Big(\frac{|\xi|}\kappa\Big) \, e^{t \lambda_j(\xi)} P_j(\xi)
$$
and~$\lambda_j(\xi) \in \bC$   are the eigenvalues of $\widehat B^1$ with associated eigenprojections $P_j(\xi)$ on $L^2_v$, satisfying the properties stated in the lemma.
 The properties of the projectors come from~\cite{Ellis-Pinsky,Bardos-Ukai}. 
   \end{proof}
\begin{Rmk} \label{defU} Denoting
\begin{equation} \label{deftildeP1j}
\tilde P_j \Big(\xi, {\xi \over |\xi|}\Big) := P_{j}^1\Big({\xi \over |\xi|}\Big) + |\xi| P_{j}^2(\xi)
 \end{equation}
 for $1 \le j \le 4$ we can   further split $\widehat U^\eps_j(t)$ into four parts (a main part and three remainder terms):
$$
  U^\eps_j =   U^\eps_{j0} +  U^{\eps\sharp}_{j0} +   U^\eps_{j1} +   U^\eps_{j2}
$$
where 
\begin{align*}
&\widehat U^\eps_{j0}(t,\xi) := e^{i\alpha_j |\xi|{t \over \eps} - \beta_j t |\xi|^2} P_{j}^0\Big({\xi \over |\xi|}\Big), \\
&\widehat U^{\eps\sharp}_{j0}(t,\xi) := \Big(\chi\Big (\frac{\eps|\xi|}\kappa\Big)-1\Big) e^{i\alpha_j |\xi|{t \over \eps} - \beta_j t |\xi|^2}P_{j}^0\Big({\xi \over |\xi|}\Big)\\
&\widehat U^\eps_{j1}(t,\xi) := \chi\Big (\frac{\eps|\xi|}\kappa\Big) e^{i\alpha_j |\xi|{t \over \eps} - \beta_j t |\xi|^2} \Big(e^{t \frac{\gamma_j(\eps|\xi|)}{\eps^2}}-1 \Big)P_{j}^0\Big({\xi \over |\xi|}\Big)\, , \\
&\widehat U^\eps_{j2}(t,\xi) := \chi\Big (\frac{\eps|\xi|}\kappa\Big) e^{i\alpha_j |\xi|{t \over \eps} - \beta_j t |\xi|^2+ t \frac{\gamma_j(\eps|\xi|)}{\eps^2}} \eps |\xi|\tilde P_{j} \Big(\eps \xi, {\xi \over |\xi|}\Big)
\, .
\end{align*}
In the following we set
$$
U^\e_{\rm{disp}}:=U^\eps_{10} + U^\eps_{20} \, . 
$$
One can notice that $U_{30}:=U^\eps_{30}$ and $U_{40}:=U^\eps_{40}$ do not depend on $\eps$ since $\alpha_3=\alpha_4=0$. We set 
$$
U := U_{30} + U_{40}\, . 
$$
It is proved in~\cite{Bardos-Ukai} -- and in this paper (see Lemma~{\rm\ref{lem:UepstoU0}}) --  that the operator~$U(t)$ is a limit of~$U^\e(t)$.
\end{Rmk}

\begin{prop} We have that $U(0)$ is the projection on the subset of $\operatorname{Ker} L$ consisting of functions $f$ satisfying $\Div u_f = 0$ and also~$\rho_f+ \theta_f = 0$ where we recall that 
\begin{equation}\label{formulaU(0}
U(0) =\mathcal{F}_x^{-1} \bigg(P_{3}^0\Big({\xi \over |\xi|}\Big) + P_{4}^0\Big({\xi \over |\xi|}\Big)\bigg) \mathcal{F}_x
\end{equation}
and with the notations
$$
\begin{aligned}
\rho_{f}(x)= \int_{\R^d} f (x,v)M^\frac12 (v) \, dv \, , \quad  u_{f} (x)= \int_{\R^d} v\,  f (x,v)M^\frac12 (v) \, dv\, , \\
 \theta_{f} (x)= {1 \over d} \int_{\R^d} (|v|^2-d) f (x,v)M^\frac12 (v) \, dv \, .
\end{aligned}
$$
We also have
\begin{equation}\label{U(t)U(0}
U(t) f = U(t) U(0) f \, , \quad \forall \, t \ge 0, \quad \forall \, f \in X^{\ell,k} \, 
\end{equation}
and
\begin{equation} \label{P012}
\Div u_f = 0 \text{ and } \rho_f+ \theta_f = 0 \Rightarrow P_{j}^0\Big({\xi \over |\xi|}\Big) f =0 \text{ for } j=1,2. 
\end{equation}
\end{prop}
 
 \begin{proof}
 The first part of the proof can be deduced from the form of the projectors $P_3^0(\xi/|\xi|)$ and $P_4^0(\xi/|\xi|)$ given in~\cite{Ellis-Pinsky}. We point out that many authors in previous works have omitted some factor that was present in~\cite{Ellis-Pinsky}, for the sake of clarity, we thus recall that 
 $$
 P_{3}^0\Big({\xi \over |\xi|}\Big) \hat f (\xi) = \frac{2}{d+2} \Big(-1+ \frac 12 (|v|^2-d)\Big) M^\frac12 \int_{\R^d} \Big(-1+ \frac 12 (|v|^2-d)\Big) M^{\frac12} \hat f \, dv 
 $$ 
 and also the fact that $P_4^0(\xi/|\xi|)$ is a projection onto the $(d-1)$-dimensional space spanned by~$v-\Big(v \cdot \frac{\xi}{|\xi|}\Big) \frac{\xi}{|\xi|}$ for any $\xi$.
 The second part of the proposition directly comes the forms of~$U(t)$ and $U(0)$ and from the results of Lemma~\ref{EllisPinsky} on projectors. For completeness, we here also recall the exact formulas for the projectors $P_1^0$ and $P^0_2$:
 $$
 \begin{aligned}
 &P_{1,2}^0\Big({\xi \over |\xi|}\Big) \hat f (\xi) \\
 &\quad = \frac{d}{2(d+2)} \Big(1 \pm \frac{\xi}{|\xi|} \cdot v + \frac 1d (|v|^2-d)\Big) M^\frac12 \int_{\R^d} \Big(1 \pm \frac{\xi}{|\xi|} \cdot v + \frac 1d (|v|^2-d)\Big) M^{\frac12} \hat f \, dv \, .
 \end{aligned}
 $$ 
 \end{proof}
In this paper, we call {well-prepared data} this class of functions $f$ that write:
\begin{equation}\label{defwp}
\begin{aligned}
&f (x,v) = M^\frac12(v)\Big( \rho_{f} (x) +  u_{f}(x) \cdot v + \frac12(|v|^2- d)  \theta_{f} (x)\Big) \\
& \hspace{5cm} \text{with} \quad  \Div u_f = 0 \quad \text{and} \quad \rho_f+ \theta_f = 0 \, .
\end{aligned}
 \end{equation}

 \begin{lem}\label{estimatePsiepslimitPsi0}
Let~$\ell>d/2$ and~$k>d/2+1$ be given.  
The following decomposition holds
$$    \Psi^\e   = \sum_{j=1}^4 \Psi^\eps_j +  \Psi^{\eps \sharp}  $$
with
$$
\widehat \Psi^{\eps \sharp}(t)(f,f):={1 \over \eps}  \int_0^t \widehat U^{\eps\sharp}(t-t') \widehat  \Gamma \big(f(t'),f(t')\big) \, dt'
$$
and
\begin{equation}\label{decompositionPsiappendix}
\Psi^\e_j = \Psi^\eps_{j0}  + \Psi^{\eps \sharp}_{j0}  + \Psi^\eps_{j1}  + \Psi^\eps_{j2} \end{equation}
where denoting~$F(t) := \Gamma \big(f(t),f(t)\big)$
\begin{align*}
&\mathcal{F}_x \left( \Psi^\eps_{j0}(t)(f,f)\right)(\xi) := \int_0^t e^{i\alpha_j |\xi|{t-t' \over \eps} - \beta_j (t-t') |\xi|^2} |\xi|P_{j}^1\Big({\xi \over |\xi|}\Big) \widehat F(t') \, dt' \,, \\
&\mathcal{F}_x \left( \Psi^{\eps\sharp}_{j0}(t)(f,f)\right)(\xi):= \Big(\chi\Big (\frac{\eps|\xi|}\kappa\Big)-1\Big) \int_0^te^{i\alpha_j |\xi|{t-t' \over \eps} - \beta_j (t-t') |\xi|^2}|\xi|P_{j}^1\Big({\xi \over |\xi|}\Big) \widehat F(t') \, dt' \, ,\\
&\mathcal{F}_x \left( \Psi^\eps_{j1}(t)(f,f)\right)(\xi) \\
&\quad := \chi\Big (\frac{\eps|\xi|}\kappa\Big) \int_0^te^{i\alpha_j |\xi|{t-t' \over \eps} - \beta_j (t-t') |\xi|^2} \Big(e^{(t-t') \frac{\gamma_j(\eps|\xi|)}{\eps^2}}-1 \Big)|\xi|P_{j}^1\Big({\xi \over |\xi|}\Big) \widehat F(t') \, dt'\, , \\
&\mathcal{F}_x \left( \Psi^\eps_{j2}(t)(f,f)\right)(\xi) \\
&\quad := \chi\Big (\frac{\eps|\xi|}\kappa\Big) \int_0^te^{i\alpha_j |\xi|{t-t' \over \eps} - \beta_j (t-t') |\xi|^2+ (t-t') \frac{\gamma_j(\eps|\xi|)}{\eps^2}} \eps |\xi|^2P^2_j(\eps\xi) \widehat F(t') \, dt'
\, .
\end{align*}

\end{lem}

\begin{proof} 
Recall that
$$
 \Psi^\e(t)  (f,f) =\frac1\e \int_0^t U^\e (t-t') \Gamma \big(f(t'),f(t')\big) \, dt' \,.
$$
Following the decomposition of $  U^\e(t,\xi)$ in~\eqref{decompUeps}, we can split the Fourier transform of~$\Psi^\eps(t)(h,h)$ into five parts:
$$
\mathcal{F}_x\left({\Psi^\e(t)  (f,f)}\right)(\xi) = \sum_{j=1}^4 {1 \over \eps} \int_0^t \widehat U^\eps_j(t-t') \widehat F(t') \, dt'
+{1 \over \eps}  \int_0^t \widehat U^{\eps\sharp}(t- t') \widehat F(t') \, dt'\, .
$$
Remark that 
$$
F=\Gamma(f,f) \in (\operatorname{Ker} L)^\perp.
$$
From that, since for $1 \le j \le 4$, $P_j^0(\xi/|\xi|)$ is a projection onto a subspace of $\operatorname{Ker} L$, we deduce that 
$$
P_j(\e \xi) \widehat F = \e |\xi| \Big(P_j^1 \Big({\xi \over |\xi|}\Big)  + \e |\xi| P_j^2 \left(\e|\xi|\right)\Big)\widehat F
=\e |\xi| \tilde P_j\Big(\eps \xi, {\xi \over |\xi|}\Big) \widehat F\, , \quad \forall \, 1 \le j \le 4 \, .
$$
It implies that 
\begin{align*}
&\mathcal{F}_x\left({\Psi^\e(t)  (f,f)}\right)(\xi)  \\
&\quad = \sum_{j=1}^4 {1 \over \eps} \int_0^t \chi\Big(\frac{\eps |\xi|}  \kappa\Big)e^{i\alpha_j|\xi|{{t-t'} \over \eps}-\beta_j(t-t')|\xi|^2+(t-t')\frac{\gamma_j(\eps|\xi|)}{\eps^2}}  P_j(\eps \xi) \widehat F(t') \, dt' \\
&\qquad +{1 \over \eps}  \int_0^t \widehat U^{\eps\sharp}(t-t') \widehat F(t') \, dt' \\
&\quad = \sum_{j=1}^4 \int_0^t \chi\Big(\frac{\eps |\xi|}  \kappa\Big)e^{i\alpha_j|\xi|{{t-t'} \over \eps}-\beta_j(t-t')|\xi|^2+(t-t')\frac{\gamma_j(\eps|\xi|)}{\eps^2}}  |\xi| \tilde P_j\Big(\eps \xi, {\xi \over |\xi|}\Big) \widehat F(t') \, dt' \\
&\qquad +{1 \over \eps}  \int_0^t \widehat U^{\eps\sharp}(t-t') \widehat F( t') \, dt' \\
&\quad =: \sum_{j=1}^4\Psi^\eps_j(t)(f,f) + \Psi^{\eps \sharp}(t)(f,f) \, .
\end{align*}
The rest of the proof follows from the decomposition given in Remark~{\rm\ref{defU}}.
\end{proof}
\begin{Rmk}\label{defPsi}
Let us notice that  as in Remark~{\rm\ref{defU}} there holds
$$
 \Psi^\eps_{30} =  \Psi_{30} \quad \mbox{and}\quad  \Psi^\eps_{40}  =  \Psi_{40} 
$$
and we set
$$
 \Psi := \Psi_{30}+ \Psi_{40} \,.
$$
It is proved in~\cite{Bardos-Ukai} that given~$\bar g_{{\rm in}} \in X^{{\ell},k}$   of the form~{\rm(\ref{defg0dataperiodi})}, the function~$g$ defined in~{\rm(\ref{defg0solution})} satisfies
$$
g(t) =U (t)  \bar g_{{\rm in}}    +    \Psi(t)  (g  ,g  ) \, .
$$
\end{Rmk}

\section{Results on the Cauchy problem for the Boltzmann and Navier-Stokes-Fourier equations} 
\subsection{Functional spaces}
The spaces $\tilde L^\infty([0,T],H^s(\R^d))$ and $\tilde L^\infty([0,T],H^s(\T^d))$ are defined through their norms (see~\cite{cheminlerner})
$$
 \|f\|_{\tilde L^\infty([0,T],H^s(\R^d))}^2:= \int_{\R^d} \langle \xi \rangle^{2s} \sup_{t \in [0,T]} |\widehat f(t,\xi)|^2 \, d\xi
$$
and
$$
  \|f\|_{\tilde L^\infty([0,T],H^s(\T^d))}^2:=\sum_{\xi \in \Z^d} \langle \xi \rangle^{2s} \sup_{t \in [0,T]} |\widehat f(t,\xi)|^2 \, .
 $$ 
Let us now recall two elementary inequalities 
 \begin{equation} \label{lem:algebra}
\forall \, \ell>d/2 \, ,\quad \|f_1 f_2\|_{ \tilde L^\infty_tH^\ell} \lesssim \|f_1 \|_{ \tilde L^\infty_tH^\ell}  \|f_2\|_{ \tilde L^\infty_tH^\ell} \,,
\end{equation}
and
 \begin{equation}\label{lem:prodHell}
\forall \, \ell>d/2 \, ,   \forall \, m\ge 0 \, , \quad
\|f_1 f_2\|_{H^{m} } \lesssim \|f_1\|_{H^{m} } \|f_2\|_{H^\ell} + \|f_1\|_{H^\ell } \|f_2\|_{H^{m} } \, ,
\end{equation}
as well as the classical product rule
 \begin{equation}\label{prodrules}
\forall \, (s,t) \in \bigg(-\frac d2,\frac d2\bigg) \,, \quad s+t>0 \, , \quad \|f_1 f_2\|_{\dot H^{s+t-\frac d2}  } \lesssim \|f_1\|_{\dot H^{s} } \|f_2\|_{\dot H^t}   \, .
\end{equation}


We also define the space $L^2(\langle v \rangle^k)$ through
$$
\|f\|_{L^2_v(\langle v \rangle^k)}^2:=\int |f(v)|^2 \langle v \rangle^{2k}\,dv\,.
$$

\subsection{Results on the Boltzmann equation}\label{appendixbotlz}
\subsubsection{The Cauchy problem} The Cauchy problem for the classical Boltzmann equation (equation~\eqref{scaledboltzmann} with $\eps=1$) has been widely studied in the last decades. Let us perform a very brief review of the results concerning the Cauchy theory of this equation in our framework of strong solutions in a close to equilibrium regime. Those results are based on a careful study of the associate linearized problems around   equilibrium.
Such studies started with Grad~\cite{Grad} and Ukai~\cite{Ukai} who developed Cauchy theories in~$X^{\ell,k}$ type spaces  (see~\eqref{Xellk}), proving the following type of result (see~\cite{Ukai,Bardos-Ukai,Glassey} for example).
\begin{prop}\label{propBoltzmannGradUkai}
Let~$\ell>d/2$ and~$k>d/2+1$ be given. For any initial data~$g_{\rm in} \in X^{\ell,k}$ there is a time~$T>0$ and a unique solution~$g$ to~\eqref{scaledboltzmann} with $\eps=1$, in the space~$C([0,T]; X^{\ell,k})$.
\end{prop}
 It has then been extended to larger spaces of the type $H^\ell_{x,v}(M^{-\frac12})$ thanks to hypocoercivity methods (see for example the paper by Mouhot-Neumann~\cite{Mouhot-Neumann}). More recently, thanks to an ``enlargement argument'', Gualdani, Mischler and Mouhot in~\cite{GMM} were able to develop a Cauchy theory in spaces with polynomial or stretched exponential weights instead of the classical weight prescribed by the Maxwellian equilibrium. We also refer the reader to the review~\cite{Ukai-Yang} by Ukai and Yang in which several results are presented. 

The study of the case $\eps=1$ is justified by rescaling or changes of physical units. However, if one wants to capture the hydrodynamical limit of the Boltzmann equation, one has to take into account the Knudsen number and obtain explicit estimates with respect to it. 

\subsubsection{Nonlinear estimates on the Boltzmann collision operator}
We here give simple estimates on the Boltzmann collision operator $\Gamma$ defined in~\eqref{defoperator}. Those estimates are not optimal in terms of weights but are enough for our purposes.  

\begin{lem} \label{lem:GammaL2v}
For $f_1=f_1(v)$ and $f_2=f_2(v)$, there holds
$$
\|\Gamma(f_1,f_2)\|_{L^2_v} \lesssim \|f_1\|_{L^2_v(\langle v \rangle^k)} \|f_2\|_{L^2_v(\langle v \rangle)} + \|f_1\|_{L^2_v(\langle v \rangle)} \|f_2\|_{L^2_v(\langle v \rangle^k)}\, , \quad k>d/2 \, .
$$
\end{lem}

\begin{proof}
For   simplicity, in what follows, we denote $\mu:= {M^\frac12}$. 
We recall that 
$$
\Gamma(f_1,f_2)= \frac12\mu^{-1}
\big(
Q(\mu f_1,\mu f_2) + Q(\mu f_2,\mu f_1)\big)=:  \frac12\Big( \Gamma_1(f_1,f_2) + \Gamma_2(f_1,f_2)\Big) \, .
$$ 
By symmetry, we focus on the first term $\Gamma_1(f_1,f_2)$.
We then split the collision operator~$Q$ into two parts (the gain and   loss terms): 
\begin{align*}
\Gamma_1(f_1,f_2) &=\frac12 \mu^{-1}\int_{\R^d \times {\mathbb S}^{d-1}} |v-v_*| \, \mu'_*f_{1*}' \mu'f_2' \, d\sigma \, dv_* - \frac12\mu^{-1} \int_{\R^d} |v-v_*| \, \mu_* f_{1*} \, dv_* \, \mu f_2 \\
&=: \Gamma_1^+(f_1,f_2) - \Gamma_1^-(f_1,f_2) \, . 
\end{align*}
We also use the notation
\begin{equation}\label{defGammapm}
 \Gamma_1^+ (f_1,f_2)= \mu^{-1}
Q^+(\mu f_1,\mu f_2) \quad \mbox{and}\quad \Gamma_1^-(f_1,f_2) = \mu^{-1}
Q^-(\mu f_1,\mu f_2)\,.
\end{equation}
We thus have
$$
\|\Gamma_1(f_1,f_2))\|_{L^2_v} \le \|\Gamma_1^+(f_1,f_2)\|_{L^2_v} + \|\Gamma_1^-(f_1,f_2)\|_{L^2_v} \, .
$$
We first focus on the simplest term, the loss one. There holds
\begin{align*}
\|\Gamma_1^-(f_1,f_2)\|_{L^2_v} &= \frac12\sup_{\|\varphi\|_{L^2_v} \le 1} \int_{\R^d \times \R^d} |v-v_*| \, \mu_* f_{1*} f_2 \varphi \, dv_* \, dv \\
&\le\frac12 \sup_{\|\varphi\|_{L^2_v} \le 1} \|f_1\|_{L^1_v(\langle v \rangle \mu)} \int_{\R^d} \langle v \rangle |f_2| |\varphi| \, dv \\
&\le\frac12 \|f_1\|_{L^2_v} \|f_2\|_{L^2_v(\langle v \rangle)}\, . 
\end{align*}
For the gain term, we first recall that from~\cite[Theorem~2.1]{Mouhot-Villani}, there holds
$$
\|Q^+(f_1,f_2)\|_{L^2_v} \lesssim \|f_1\|_{L^1_v} \|f_2\|_{L^2_v} \lesssim  \|f_1\|_{L^2_v(\langle v \rangle^k)} \|f_2\|_{L^2_v}\,, \quad k>d/2 \, . 
$$
Then, using the equality $\mu \mu_* = \mu' \mu'_*$ (and bounding~$\mu_*$ by 1), we get  \begin{align*}
\|\Gamma_1^+(f_1,f_2)\|_{L^2_v} &= \frac12 \sup_{\|\varphi\|_{L^2_v} \le 1}\int_{\R^d \times \R^d \times {\mathbb S}^{d-1}} \mu^{-1} |v-v_*| \, \mu'_*f_{1*}' \mu'f_2' \varphi \, d\sigma \, dv_* \\
&\le  \frac12\sup_{\|\varphi\|_{L^2_v} \le 1}\int_{\R^d \times \R^d \times {\mathbb S}^{d-1}} |v-v_*| \mu_* |f_{1*}' | |f_2'| |\varphi| \, d\sigma \, dv_* \, dv \\
&\le  \frac12 \sup_{\|\varphi\|_{L^2_v} \le 1}\int_{\R^d} Q^+(|f_1|,|f_2|) |\varphi| \, dv \\
&\le \frac12 \|Q^+(|f_1|,|f_2|)\|_{L^2_v} \lesssim \|f_1\|_{L^2_v(\langle v \rangle^k)} \|f_2\|_{L^2_v}\,, \quad k>d/2 \, . 
\end{align*}
The lemma is proved since $\Gamma_2(f_1,f_2)= \Gamma_1(f_2,f_1)$.
\end{proof}

\begin{lem} \label{lem:GammaLinftyv}
If $f_1=f_1(v)$ and $f_2=f_2(v)$ grow polynomially in $v$, then for any $k \ge 0$,
$$
\Gamma(M^\frac12 f_1,M^\frac12 f_2) \in L^{\infty,k}_v \,. 
$$
\end{lem}
\begin{proof}
From the definition of the collision operator $\Gamma$ in~\eqref{defoperator}, we have:
\begin{align*}
&\Gamma(M^\frac12 f_1,M^\frac12 f_2) =\\
&  \frac{M^{-\frac12}}2 \int_{\R^d \times \Sp^{d-1}} |v-v_*| \left(M'_*f'_{1*} M' f_2' + M'_*f'_{2*} M' f_1' - M_*f_{1*} M f_2 -  M_*f_{2*} M f_1 \right) \, dv_* \, d\sigma \, .
\end{align*}
Then, we use that $M'M'_*=MM_*$ to get 
 \begin{align*}
&|\Gamma(M^\frac12 f_1,M^\frac12 f_2)|  \\
&\quad \lesssim M^\frac12 \int_{\R^d \times \Sp^{d-1}} |v-v_*| M_* \left(|f'_{1*}| |f_2'| + |f'_{2*}| |f_1'| + |f_{1*}| |f_2| +  |f_{2*}| |f_1| \right) \, dv_* \, d\sigma \, .
\end{align*}
Finally, since
$
\langle v' \rangle + \langle v'_* \rangle \lesssim \langle v \rangle \langle v_* \rangle 
$
and $f_1$, $f_2$ are polynomial in $v$, we obtain a bound of the form 
 \begin{align*}
&|\Gamma(M^\frac12 f_1,M^\frac12 f_2)|  \\
&\quad \lesssim M^\frac12 \langle v \rangle^q \int_{\R^d} M_* \langle v_* \rangle^q \, dv_* \le C 
\end{align*}
for some $q \ge 0$.  The lemma follows.
\end{proof}
\begin{lem}
We have, for~$\ell>d/2$ and~$k>d/2+1$,
\begin{equation} \label{eq:GammaHellxL2v}
\|\Gamma(f_1,f_2)\|_{H^\ell_xL^2_v} \lesssim \|f_1\|_{L^{\infty,k}_v W^{\ell,\infty}_x} \|f_2\|_{\ell,k}\,,
\end{equation}
\begin{equation} \label{estimUkai1}
\|\Gamma(f_1,f_2)\|_{H^\ell_xL^2_v} \lesssim \|\Lambda^{-1}\Gamma(f_1,f_2)\|_{X^{\ell,k}}\lesssim \|f_1\|_{\ell,k} \|f_2\|_{\ell,k} 
\end{equation}
and
\begin{equation} \label{estimUkai2}
   \|\Gamma(f_1,f_2)\|_{L^2_vL^1_x} \lesssim \|f_1\|_{\ell,k} \|f_2\|_{\ell,k} \, . 
\end{equation}
\end{lem}
\begin{proof}
Let us recall that $\Lambda^{-1} X^{\ell,k} \hookrightarrow H^\ell_xL^2_v$. Then, we write 
$$
\|\Lambda^{-1} \Gamma(f_1,f_2)\|_{\ell,k} = \sup_v \left( \langle v \rangle^{k-1} \|\Gamma(f_1,f_2)\|_{H^\ell_x} (v)\right)
$$
Using the quadratic form of the gain and the loss terms of the collision operator $\Gamma$ given in~(\ref{defGammapm}) and the fact that it is local in $x$, we notice that 
$$
\|\Gamma^{\pm}(f_1,f_2)\|_{H^\ell_x} \lesssim \Big|\Gamma^{\pm}\big(\|f_1\|_{W^{\ell,\infty}_x}, \|f_2\|_{H^{\ell}_x}\big)\Big| \, .
$$
As a consequence, we get 
\begin{align*}
\|\Gamma(f_1,f_2)\|_{H^\ell_xL^2_v} &\lesssim \big\|\Gamma^{+}\big(\|f_1\|_{W^{\ell,\infty}_x}, \|f_2\|_{H^{\ell}_x}\big)\big\|_{L^{\infty,k-1}_v} + \big\|\Gamma^{-}\big(\|f_1\|_{W^{\ell,\infty}_x}, \|f_2\|_{H^{\ell}_x}\big)\big\|_{L^{\infty,k-1}_v}  \\
&\lesssim \|f_1\|_{L^{\infty,k}_v W^{\ell,\infty}_x} \|f_2\|_{L^{\infty,k}_v H^\ell_x} \, .
\end{align*}
The other estimates are taken from~\cite[Lemma~4.5.1]{Ukai}. 
\end{proof}

\subsection{Results on the limit equation}\label{appendixNS}
\subsubsection{The Navier-Stokes-Fourier system}
The results used in this paper are summarized in the following statement, elements of proofs are given below.  Note that we
make no attempt at exhaustivity in this presentation, nor do we state the optimal results present in the literature (we refer among other references to~\cite{cheminSIAM,BCD,lemarie1,lemarie2} for more on the subject). 
\begin{prop}\label{propNSF}
Let~$\ell \geq d/2-1$.
Given~$ (\rho_{\rm in},u_{\rm in},\theta_{\rm in}) $ in~$ H^\ell(\Omega)$, there is  a unique maximal time~$T^*>0$ and  a unique solution~$ (\rho ,u ,\theta ) $ to~{\rm(\ref{NSF})} in~$ L^\infty([0,T],H^\ell(\Omega)) \cap L^2([0,T],H^{\ell+1}(\Omega))$ for all times~$T<T^*$.  It satisfies
\begin{equation}\label{estimateH1/2}
\begin{aligned}
& \|(\rho ,u ,\theta ) \|_{\tilde L^\infty\big([0,T],H^{\frac d2-1}(\Omega)\big) } +   \|(\nabla\rho ,\nabla u ,\nabla\theta ) \|_{L^2\big([0,T],H^{\frac d2-1}(\Omega)\big) }   \lesssim \| (\rho_{\rm in},u_{\rm in},\theta_{\rm in})\|_{H^{\frac d2-1}(\Omega)) }   \, ,
\end{aligned}
\end{equation}
and if~$\ell>d/2-1$
\begin{equation}\label{estimateHell}
\begin{aligned}
& \|(\rho ,u ,\theta ) \|_{\tilde L^\infty([0,T],H^\ell(\Omega)) } +   \|(\nabla\rho ,\nabla u ,\nabla\theta ) \|_{L^2([0,T],H^\ell(\Omega)) } \\
&\quad\qquad\quad \lesssim \| (\rho_{\rm in},u_{\rm in},\theta_{\rm in})\|_{H^\ell(\Omega)} \times \exp C   \|\nabla u\|^2_{L^2\big([0,T], H^{\frac d2-1}(\Omega)\big) } \, .
\end{aligned}
\end{equation}
Moreover if~$d=2$ then~$T^*=\infty$, and  if~$ (\rho_{\rm in},u_{\rm in},\theta_{\rm in}) $ lies in~$H^\ell\cap L^1(\Omega)$ then
  for any~$t >0$, 
 $$
 \begin{aligned}
 \|(\rho ,u ,\theta )(t)\|_{L^q(\Omega) } &\lesssim {1 \over\langle t \rangle^{1-\frac1q}} \, , \quad \forall \, 2 \le q < \infty
\\
 \big\|(D^\alpha \rho,D^\alpha u,D^\alpha \theta)(t)\big \|^2_{L^2(\Omega) }& \lesssim {1 \over \langle t \rangle^{1+|\alpha|}} \, , \quad \forall \, \alpha \in \N^2 \, , \quad  |\alpha| \le \ell \, ,
 \end{aligned}
 $$ 
 with~$D = \sqrt{-\Delta}$.
Similarly   if~$d=3$, if~$T^*=\infty$ and  if~$ (\rho_{\rm in},u_{\rm in},\theta_{\rm in}) $ lies in~$H^\ell\cap L^1(\Omega)$ 
then   for any~$t >0$, 
$$
\|(\rho ,u ,\theta ) (t)\|_{H^\ell(\R^3)}   \lesssim {1 \over \langle t \rangle^{\frac14}}\,  \cdotp
$$
Furthermore  if~$d=3$, there is a constant~$c>0$ such that if
$$
\|u_{\rm in}\|_{\dot H ^{\frac 12 }(\Omega)} \leq c
$$
then~$T^*=\infty$.    

\smallskip

\noindent Finally~{\rm(\ref{NSF})} is stable in the sense that if~$ (\rho_{\rm in},u_{\rm in},\theta_{\rm in}) $ in~$ H^\ell(\Omega)$ generates a   unique solution on~$[0,T]$ then there is~$c'>0$ (independent of~$T$ if~$ (\rho_{\rm in},u_{\rm in},\theta_{\rm in}) $ generates a global solution) such that any initial data in a ball of~$ H^\ell(\Omega)$  centered at~$ (\rho_{\rm in},u_{\rm in},\theta_{\rm in}) $ and of radius~$c'$ also generates a unique solution on~$[0,T]$.
\end{prop}
\noindent {\it Sketch of proof.} $ $  Let us start by considering the Navier-Stokes system. It is known since~\cite{fujitakato,cheminSIAM} that
  given~$ u_{\rm in} $ in~$ H^\ell(\Omega)$ with~$\ell \geq \displaystyle \frac d2-1$, there is a unique maximal time~$T^*>0$ and a unique associate solution~$u$  to the Navier-Stokes equations  in~$ L^\infty([0,T],H^\ell(\Omega)) \cap L^2([0,T],H^{\ell+1}(\Omega))$ for all times~$T<T^*$ (see~\cite{cheminlerner} for the case of~$\tilde L^\infty([0,T],H^\ell(\Omega)) \cap L^2([0,T],H^{\ell+1}(\Omega))$) satisfying~(\ref{estimateH1/2})-(\ref{estimateHell}).  Moreover, the solution $u$ is global in time if~$u_{\rm in}$ is small   in~$ \dot H^{\frac d2-1}(\Omega)$, and it satisfies  in that case $$
  \|u\|_{L^\infty \big(\R^+, H^{\frac d2-1}(\Omega) \big)} +   \|\nabla u\|_{L^2 \big(\R^+, H^{\frac d2-1}(\Omega) \big) } \leq   \|u_{\rm in}\|_{H^{\frac d2-1}(\Omega)} \, .
  $$
   On the other hand   if~$d=2$
  then global existence and uniqueness in~$ L^\infty(\R^+,L^2(\Omega)) \cap L^2(\R^+,\dot H^{1}(\Omega))$ (where the space~$\dot H^{1}(\Omega)$ is the homogeneous Sobolev space) holds unconditionnally (see~\cite{leray2D}).      
  \medskip
  
  Let us turn to the time decay properties. In the periodic case the mean free assumption implies that global solutions have exponential decay in time so let us consider the whole space case. 
 In two space dimensions it is proved in~\cite{wiegner} (see also~\cite{Schonbek24}) that if the initial data lies in~$L^2\cap L^1(\R^2)$ (whatever its size) then the solution decays in~$H^\ell(\R^2)$ as~$\langle t \rangle^{-\frac12}$, and moreover for any~$t >0$ and for all~$ \alpha \in \N^2 $ such that~$ |\alpha| \le \ell$, 
 $$
 \|D^\alpha u(t)\|^2_{L^2} \lesssim {1 \over \langle t \rangle^{1+|\alpha|}} \,  \cdotp
 $$ 
 The time decay in three space dimensions is due to the fact that    any global solution in~$\dot H^{\frac 12}(\Omega)$, regardless of the size of the initial data, decays to zero   in large times in~$\dot H^\frac 12 (\Omega)$  (see~\cite{GIP}). On the other hand it is known (\cite{Schonbek}) that
some Leray-type~\cite{leray} weak solutions associated with~$L^2$ initial data decay in~$H^\ell$ with the rate~$\langle t\rangle^{-\frac14}$, so weak strong uniqueness gives the result.
 The stability of solutions for short times is an easy computation, for large times it follows from the fact that large solutions become small in large times    (see~\cite{GIP,wiegner}).  

 \medskip
 
 Finally it is an easy matter to prove that the temperature~$\theta$ and the density~$\rho$, which solve a linear transport-diffusion equation, enjoy the same properties as~$u$. 
     \qed
     
        \subsubsection{The limit   equation}
Let us start by noticing   that if~$(\rho,u,\theta)$ belongs to~$ L^\infty([0,T],H^\ell(\Omega)) $ and if~$\nabla (\rho,u,\theta)$ belongs to~$L^2([0,T],H^{\ell }(\Omega))$, then clearly 
$$
g
(t,x,v) := M^\frac12(v)\big(\rho (t,x) + u (t,x) \cdot v + \frac12(|v|^2- d) \theta(t,x) \big)
$$
belongs to~$ L^\infty([0,T],X^{\ell,k}) $ and~$\nabla g$ belongs to~$L^2([0,T],X^{\ell ,k})$ for all~$k\geq 0$. Similarly all the results stated in Proposition~\ref{propNSF} are easily extended to~$g $ in the space~$X^{\ell,k}$.
Moreover it will be useful in the following to
 remark that~$g$ is of the following form:
\begin{equation} \label{eq:gp12}
g(t,x,v) = \sum_{p=1}^{d+2} g_{p }(t,x) \widetilde g _{p }(v) M^\frac12(v)
\end{equation}
where $\widetilde g _{p}(v)$ is polynomial in $v$. 

\medskip
   In what follows, when it is not mentioned, the Lebesgue norms in time (denoted~$L^p_t$) are taken on $\R^+$ if the solutions of~\eqref{NSF} are global in time or in $[0,T]$ for any $T<T^*$ if $T^*$ is the maximal time of existence of solutions.

\medskip

The following statement is an immediate consequence of Proposition~\ref{propNSF}.
    \begin{lem} \label{lem:propNS}
Let $\Omega = \T^d$ or $\R^d$ with $d = 2,3$, and set~$\ell\geq d/2-1$ For any $1 \le p \le d+2$,   there holds
\begin{enumerate}[(i)]
\item $g_p \in \tilde L^\infty_t H^\ell $,
\item $\nabla g_p \in L^2_t  H^{\ell }$,
{\item $g_p \in L^2_t  H^{\ell+1} $ if $\Omega = \T^d$ from the Poincar\'e inequality}.
\end{enumerate}
Moreover if $\Omega=\R^2$ then for any $1 \le p \le 4$ and any $t \ge 1$,   there holds
$$
\|g_p(t,\cdot)\|_{L^q} \lesssim {1 \over\langle t \rangle^{1-\frac1q}} \, , \quad \forall \, 2 \le q < \infty
$$
and 
$$
\|D^\alpha g_p(t,\cdot)\|^2_{L^2} \lesssim {1 \over \langle t \rangle^{1+|\alpha|}} \, , \quad \forall \, \alpha \in \N^2 \, , \quad \, |\alpha| \le \ell. 
$$
\end{lem}

\medskip

The properties recalled above on the Navier-Stokes equations imply in particular the following results. In the rest of this section we consider~$\ell>d/2$. We define
$$
\gamma_1(t,\xi):=\big\| \widehat {\Gamma}(g,g)({t},\xi)\big\|_{L^2_v}\quad \text{and} \quad 
{\gamma_2}({t},\xi) :=  \|  \partial_{t}  \widehat  {\Gamma}(g,g)(t,\xi) \|_{L^2_v} \, .
$$
Let us prove the following lemma.
\begin{lem} \label{lem:gamma&tildegamma} 
Let $\Omega= \T^d$ or $\R^d$ for $d= 2,3$. There holds
{ \begin{enumerate}[(i)]
\item  $  \gamma_1   \in \tilde L^\infty_tL^2_\xi(\langle \xi \rangle^\ell) $ and $\|\Gamma(g,g)\|_{\ell,k} \in L^\infty_t$, 
 \item $\gamma_2\in L^2_t L^2_\xi(\langle \xi \rangle^{\ell-1})$ and
 $\mathcal{F}_x^{-1} \gamma_2 \in  L^\frac43_{t} \dot H^{-\frac12}_{x}$ if~$d=3$.
 \end{enumerate}}
\end{lem}
\begin{proof} 
For $(i)$, using the form of $g$ given in~\eqref{eq:gp12} together with Lemma~\ref{lem:GammaL2v}, we remark that
 \begin{equation}\label{lem:hatGamma(i)}
\gamma_1(t,\xi) = \big\| \widehat {\Gamma}(g,g)({t},\xi)\big\|_{L^2_v}  \lesssim \sum_{p,q=1}^{d+2} |\mathcal{F}_x({g_{p }g_{q }})(t,\xi)|  \, ,
\end{equation}
and similarly,
\begin{equation} \label{lem:hatGamma(i)bis}
\|\Gamma(g,g)(t)\|_{\ell,k}\lesssim \sum_{p,q=1}^{d+2} \|(g_pg_q)(t)\|_{H^\ell_x} \|\Gamma(M^{\frac12} \widetilde g_p, M^{\frac12} \widetilde g_q)\|_{L^{\infty,k}_v} \lesssim \sum_{p,q=1}^{d+2} \|(g_pg_q)(t)\|_{H^\ell_x} 
\end{equation} 
from Lemma~\ref{lem:GammaLinftyv}. 
So to prove $(i)$, it is enough to prove that $g_pg_q \in \tilde L^\infty_tH^\ell $ for any $p$ and~$q$. And this is actually immediate using~(\ref{lem:algebra}) and the fact that every~$g_{p} $ is in $\tilde L^\infty_tH^\ell$ from Lemma~\ref{lem:propNS}. {The second part is then obvious since $g_pg_q \in \tilde L^\infty_tH^\ell $ implies that $g_pg_q \in L^\infty_tH^\ell$.}

\smallskip

Now let us turn to $(ii)$. From
 \begin{equation}\label{lem:hatGamma(ii)}
 \big\|\partial_t\widehat {\Gamma}(g,g)({t},\xi)\big\|_{L^2_v}  \lesssim\sum_{p,q=1}^{d+2} |\mathcal{F}_x({(\partial_t g_{p }) g_{q }})(t,\xi)| 
\end{equation}
 it is enough to prove estimates on $(\partial_t g_p)g_p$ for any~$p$ and $q$. 
Using the equation satisfied by~$g_p$ and recalling that~${\mathbb P}$ is the Leray projector onto divergence free vector fields
we find that
$$
\sum_{p,q=1}^{d+2} \|(\partial_t g_{p }) g_{q }\|_{H^{\ell-1}} \lesssim\sum_{p,q,r=1}^{d+2}\Big(   \|\Delta g_p g_q\|_{H^{\ell-1} }  + \|{\mathbb P}(g_p \cdot\nabla g_r )g_q\|_{H^{\ell-1} }\Big)\,.
$$
On the one hand  we have 
\begin{equation} \label{eq:Deltagg}
\begin{aligned}
&\|\Delta g_p g_q\|_{H^{\ell-1} } \\
&\quad \lesssim \|\nabla g_p \nabla g_q\|_{H^{\ell-1} } + \|\nabla g_p g_q\|_{H^\ell} \\
&\quad \lesssim \|\nabla g_p\|_{H^{\ell-1} } \|\nabla g_q\|_{H^\ell} + \|\nabla g_p\|_{H^\ell} \|\nabla g_q\|_{H^{\ell-1} } + \|\nabla g_p\|_{H^\ell} \|g_q\|_{H^\ell}\, , 
\end{aligned}
\end{equation}
where we used~(\ref{lem:prodHell}) to get the second inequality. 
We conclude for this term using that for any~$p$, we have $g_p \in L^\infty_t H^\ell$ and $\nabla g_p \in L^2_t H^\ell $. Then, the terms of the form~${\mathbb P}(g_p \cdot\nabla g_r )g_q$ are treated crudely bounding the $H^{\ell-1} $ norm by the $H^\ell$ one. We thus have
$$
\begin{aligned}
\|{\mathbb P}(g_p \cdot\nabla g_r )g_q\|_{L^2_t H^{\ell-1} } &\lesssim \|g_p\cdot\nabla g_r\|_{L^2_t H^\ell }  \|g_q\|_{L^\infty_t H^\ell}\\
 &\lesssim \| g_p\|_{L^\infty_t H^\ell} \| \nabla g_r\|_{L^2_t H^\ell}  \|g_q\|_{L^\infty_t H^\ell}  \, ,
\end{aligned}
$$
which ends the first part of the proof. To prove the second one we start by noticing that thanks to~(\ref{prodrules})
$$
\begin{aligned}
\| \Delta g_p g_q\|_{L^\frac43_{t} \dot H^{-\frac12}_{x} }& \lesssim  \|\nabla g_p \nabla g_q\|_{L^\frac43_{t} \dot H^{-\frac12}_{x}} + \|\nabla g_p g_q\|_{L^\frac43_{t} \dot H^{ \frac12}_{x} }\\
& \lesssim  \|\nabla g_p\|_{L^2_{t} \dot H^{ \frac12}_{x}}   \| \nabla g_q\|_{L^4_{t} \dot H^{ \frac12}_{x}}  +   \|\nabla g_p\|_{L^2_{t} \dot H^{ \frac76}_{x}}   \|   g_q\|_{L^4_{t} \dot H^{ \frac56}_{x}}
\end{aligned}
$$
which is bounded thanks to Lemma~\ref{lem:propNS}. Similarly
$$
\begin{aligned}
\|{\mathbb P}(g_p \cdot\nabla g_r )g_q\|_{L^\frac43_{t} \dot H^{-\frac12}_{x}  } & \lesssim  \|  {\mathbb P}(g_p \cdot\nabla g_r )\|_{L^\frac85_{t} L^2_{x}}     \|    g_q\|_{L^8 _{t} \dot H^{ 1}_{x} }\\
& \lesssim   \|  \nabla g_r\|_{L^2_{t} \dot H^{ \frac12}_{x}}   \|  g_p \|_{L^8 _{t} \dot H^{ 1}_{x} }  \|  g_q \|_{L^8 _{t} \dot H^{ 1}_{x} }\,,
\end{aligned}
$$
which is also bounded thanks to Lemma~\ref{lem:propNS}.  \end{proof}

\begin{lem} \label{lem:gamma&tildegammaT2}
Let $\Omega = \T^2$, then {$\mathcal{F}_x^{-1}(\gamma_2)  \in L^1_{t,x}$}. 
\end{lem} 
\begin{proof}
As previously, it is enough to get estimates on norms of $(\partial_t g_p)g_q$ for any $p$, $q$. First, using the Cauchy-Schwarz inequality, we have:
$$
\|(\partial_t g_{p} ) g_{q}  \|_{L^1_{t,x}} \lesssim \|\partial_t g_{p} \|_{L^2_{t,x}} \| g_{q} \|_{L^2_{t,x}} \, .
$$
{Moreover, since $\Omega = \T^2$,  we have that $g_q \in L^2_{t,x}$ for any $q$ from Lemma~\ref{lem:propNS}}. Concerning the~$L^2_{t,x}$-norm of $\partial_t g_{p} $, 
as above we use the fact that we can replace a control on~$\partial_t g_{p}$ by a control on~$\Delta g_p$ and on~${\mathbb P}(g_p \cdot\nabla g_r )$. The term~$\Delta g_{p} $ is clearly in $L^2_{t,x}$ since $\nabla g_{p}  \in L^2_t H^{\ell}$ with~$\ell>d/2=1$ from Lemma~\ref{lem:propNS}. The terms ${\mathbb P}(g_p \cdot\nabla g_r )$ can be treated as follows since $H^\ell$ is an algebra:
$$
\|{\mathbb P}(g_p \cdot\nabla g_r )\|_{L^2_{t,x}} \lesssim \|g_p \nabla g_r\|_{L^2_{t} H^\ell} 
\lesssim \|g_p\|_{L^\infty_t H^\ell} \|\nabla g_r\|_{L^2_t H^\ell } \, .
$$
The lemma follows.
\end{proof}

\begin{lem} \label{lem:gamma&tildegammaR2}
Let $\Omega = \R^2$. Then, for any $1 \le p, q \le 4$  there holds
{\begin{enumerate}[(i)]

\item For any $b \le 1/2$, $\langle t \rangle^b  |\xi| \gamma_1 \in L^2_tL^2_\xi(\langle \xi \rangle^\ell) $ and $ |\xi| \gamma_1  \in L^{\frac43}_tL^2_\xi(\langle \xi \rangle^\ell)$,
\item For any $b\le 1/2$, $\langle t \rangle^b  \gamma_2  \in L^2_t L^2_\xi(\langle \xi \rangle^{\ell-1}) $, for any $b \le 3/4$, $\langle t \rangle^b  \mathcal{F}_x^{-1}(\gamma_2) \in L^2_t L^{\frac43}_x$ and for any~$b<1/2$, $\langle t \rangle^b  \mathcal{F}_x^{-1}(\gamma_2)  \in L^{\frac43}_{t,x}$,
\item For any $b \le 1$, $\langle t \rangle^b  \gamma_1 \in  L^\infty_tL^2_\xi(\langle \xi \rangle^\ell)$ and $ \mathcal{F}_x^{-1}(\gamma_1)  \in L^\infty_t L^{\frac43}_x$. 
\end{enumerate}}
\end{lem}
\begin{proof}
Similarly to above, it is enough to get estimates on $g_p \nabla g_q$ for $(i)$, on $(\partial_t g_p) g_q$ for $(ii)$ and on $g_pg_q$ for $(iii)$ for any $p$ and $q$. The proof mainly relies on Lemma~\ref{lem:propNS}. 

\smallskip

For the point~$(i)$, since $H^\ell $ is an algebra, we have for any $p$, $q$:
\begin{align*}
&\|\langle t \rangle^b g_p \nabla g_q\|^2_{L^2_tH^\ell } \\
&\quad \lesssim \int_0^1 \|g_p(t,\cdot) \|^2_{H^\ell } \|\nabla g_q(t,\cdot)\|^2_{H^\ell } \, dt + \int_1^\infty \langle t \rangle^{2b} \|g_p(t,\cdot)\|^2_{H^\ell } \|\nabla g_q(t,\cdot)\|^2_{H^\ell } \, dt \\
&\quad =: I_1+I_2.
\end{align*}
The term $I_1$ is finite since for any $p$, we have $g_p \in L^\infty_t H^\ell$ and $\nabla g_p \in L^2_t H^\ell$ from Lemma~\ref{lem:propNS}. For $I_2$, from Lemma~\ref{lem:propNS}, we have:
$$
I_2 \lesssim\Big( \sup_t {\langle t \rangle^{2b -1}} \|\nabla g_q\|_{L^2_t H^\ell}^2\Big) \Big( \sup_t \langle t \rangle \|  g_p(t)\|_{ H^\ell}^2\Big)  
$$
which is finite since~$b \leq 1/2,$
from which we can conclude. 
For the second part of $(i)$, using that $H^\ell$ is an algebra and H\"older's inequality in time, we can write that for any $p$, $q$:
$$
\|g_p \nabla g_q\|_{L^{\frac43}_t H^\ell} \lesssim \|g_p\|_{L^4_t H^\ell} \|\nabla g_q\|_{L^2_t H^\ell}
$$
which gives the result still using Lemma  \ref{lem:propNS}.
 
\smallskip

Concerning $(ii)$, we use the same strategy keeping in mind that norms on~$\partial_t g_p$ can be controled by  the same norms on~$\Delta g_p$ and ${\mathbb P}(g_p \cdot\nabla g_r )$. Moreover, for any~$1 \le p,q,r \le 4$, using the second inequality of~\eqref{eq:Deltagg}, we have 
\begin{align*}
& \int_1^\infty \langle t \rangle^{2b} \left(\|\Delta g_p(t,\cdot) g_q(t,\cdot)\|^2_{H^{\ell-1} }+\|{\mathbb P}(g_p \cdot\nabla g_r )  g_q(t,\cdot)\|^2_{H^{\ell-1} }\right) \, dt \\
&\quad \lesssim \int_1^\infty \langle t \rangle^{2b} \Big(\|\nabla g_p\|^2_{H^{\ell-1} } \|\nabla g_q\|^2_{H^\ell} + \|\nabla g_p\|^2_{H^\ell} \|\nabla g_q\|^2_{H^{\ell-1}} + \|\nabla g_p\|^2_{H^\ell} \|g_q\|^2_{H^\ell} \\
&\hskip 10cm + \|g_p\|^2_{H^\ell} \|g_q\|^2_{H^\ell} \|\nabla g_r\|^2_{H^\ell}\Big) \, dt \, .
\end{align*}
So recalling that from Lemma~\ref{lem:propNS}, for any $p$,
$$
\|\nabla g_p\|^2_{H^{\ell} } \in L^1_t\, ,  \quad  \|g_p(t)\|^2_{H^\ell} \lesssim \frac1{\langle t \rangle}\quad \mbox{and}\quad  \|\nabla g_p(t)\|^2_{H^{\ell-1} }\lesssim {1 \over \langle t \rangle^2} \, ,
$$
we get the result as soon as~$b\le1/2$. For the second part of $(ii)$, still using Lemma~\ref{lem:propNS}, we notice that thanks to H\"older's inequality, for~$t \gtrsim 1$, 
$$
\|\Delta g_p g_q\|^2_{L^{\frac43}} \lesssim \|\Delta g_p \|^2_{L^2} \|g_q\|^2_{L^4} \lesssim  \|\nabla g_p\|^2_{H^{\ell} } {1 \over {\langle t \rangle^{\frac32}}}
$$
and 
$$
\begin{aligned}
\|{\mathbb P}(g_p \cdot\nabla g_r ) g_q\|^2_{L^{\frac43}}& \lesssim \||{\mathbb P}(g_p \cdot\nabla g_r ) \|^2_{L^\frac85}  \|g_q\|^2_{L^8}
\\& \lesssim \|\nabla g_r\|^2_{L^2} \|g_p\|^2_{L^8} \|g_q\|^2_{L^8} \\& 
\lesssim  \|\nabla g_r\|^2_{H^{\ell} }  {1 \over {\langle t \rangle^{\frac72}}}\, \cdotp
\end{aligned}
$$
Consequently, for any $b\le 3/4$, we have 
$$
\langle t \rangle^b (\partial_t g_p) g_q \in L^2_t \big(\R^+,L^{\frac43}\big) \, .
$$
For the last part of $(ii)$, we notice that 
$$
\langle t \rangle^b \|\Delta g_p g_q\|_{L^{\frac43}} \lesssim \|\nabla g_p\|_{H^1} \langle t \rangle^b \|g_q\|_{L^4}
$$
and from Lemma~\ref{lem:propNS}, we have $\langle t \rangle^b \|g_q\|_{L^4} \in L^4_t$ as soon as $b<1/2$. Similarly, 
$$
\begin{aligned}
\|\langle t \rangle^b {\mathbb P}(g_p \cdot\nabla g_r ) g_q\|_{L^{\frac43}_{t,x}}
& \lesssim
 \big\|\langle t \rangle^{\frac b2}  {\mathbb P}(g_p \cdot\nabla g_r )\big\|_{L^\frac85_{t,x}} 
 \big\|\langle t \rangle^{\frac b2} g_q\big\|_{L^8_{t,x}} 
\\
& \lesssim \|\nabla g_r\|_{L^2_{t,x}} \big\|\langle t \rangle^{\frac b2} g_p\big\|_{L^8_{t,x}}\big\|\langle t \rangle^{\frac b2} g_q\big\|_{L^8_{t,x}} 
\end{aligned}
$$
which is also bounded if $b<1/2$ from Lemma~\ref{lem:propNS}. 

\smallskip

Finally, the first part of $(iii)$ is clear from Lemma~\ref{lem:propNS} and the second one comes from the   H\"older inequality
$$
\|g_pg_q\|_{L^{\frac43}} \le \|g_p\|_{L^2} \|g_q\|_{L^4}
$$
which is uniformly bounded in time from Lemma~\ref{lem:propNS}. 
This concludes the proof of   Lemma~\ref{lem:gamma&tildegammaR2}.
\end{proof}


\begin{thebibliography}{99}
 
\bibitem{BMN}  A. Babin, A. Mahalov, B. Nicolaenko:
Regularity and integrability of $3$D Euler and Navier-Stokes equations for
rotating fluids,
{\it Asymptot. Anal.}, {\bf 15}, (1997), no. 2, pages~103-150.
 
 
 \bibitem{BCD}   H. Bahouri, J.-Y. Chemin and R. Danchin,
{\it Fourier Analysis and Nonlinear Partial Differential Equations}, Grundlehren der mathematischen Wissenschaften, {\bf 343}, Springer-Verlag Berlin Heidelberg, 2011. 
 
\bibitem{bgl1} C. Bardos, F. Golse, C.D. Levermore, Fluid Dynamic Limits of the Boltzmann
Equation I, {\it J. Stat, Phys.} {63} (1991), 323-344.

\bibitem{bgl2} C. Bardos, F. Golse, C.D. Levermore, Fluid Dynamic Limits of Kinetic Equations
II: Convergence Proofs for the Boltzmann Equation,  {\it Comm. Pure   Appl. Math} {46} 
46 (1993), 667-753. 


\bibitem{Bardos-Ukai} C. Bardos and  S. Ukai, The classical incompressible Navier-Stokes limit of the Boltzmann equation.
{\it Math. Models Methods Appl. Sci.} 1 (1991), no. 2, 235-257. 



\bibitem{Briant} M. Briant,  From the Boltzmann equation to the incompressible Navier-Stokes equations on the torus: a quantitative error estimate, {\it Journal of Differential Equations}, Vol. 259, Issue 11 pp 6072-6141 (2015)

\bibitem{Briant-Mouhot-Merino} M. Briant, S. Merino-Aceituno and C. Mouhot, From Boltzmann to incompressible Navier-Stokes in Sobolev spaces with polynomial weight to appear in {\it Analysis and Applications}.

\bibitem{Caflisch} R.E. Caflisch,The fluid dynamic limit of the nonlinear Boltzmann equation, {\it Comm. on Pure and Appl. Math.} {\bf 33} (1980), 651--666.

\bibitem{ChapEns} S. Chapman, T.G. Cowling, The mathematical theory of non-uniform gases: An account of the kinetic theory of viscosity, thermal conduction, and diffusion in gases. {\it Cambridge University Press}, New York, 1960.

\bibitem{cheminSIAM} J.-Y. Chemin, Remarques sur l'existence globale pour le syst\`eme de Navier-Stokes incompressible, {\it SIAM J. Math. Anal.} {\bf 23} (1992), no. 1, 20-28.


\bibitem{CDGG} J.-Y. Chemin, B. Desjardins, I. Gallagher and
 E. Grenier, Fluids with anisotropic viscosity, {\it Mod\'elisation
 Math\'ematique et Analyse Num\'erique}, {\bf 34}, 2000, pages 315-335.

\bibitem{cg1} J.-Y. Chemin and I. Gallagher,  On  the global wellposedness  of the 3-D 
   Navier-Stokes equations with large initial data, 
{\it Annales Scientifiques de l'\'Ecole Normale Sup\'erieure de Paris}, {\bf 39} (2006), 679-698.

\bibitem{cg} J.-Y. Chemin and I. Gallagher, Large, global solutions to the 
 Navier-Stokes  equations slowly varying in one direction,  {\it Transactions of the American Mathematical Society} {\bf 362}  (2010),  no. 6, 2859-2873.


\bibitem{cheminlerner} J.-Y. Chemin and N. Lerner,  Flot de champs de vecteurs non lipschitziens et \'equations de Navier-Stokes, {\it
J. Differential Equations} {\bf 121} (1995), no. 2, 314-328. 


\bibitem{DeMasi-Esposito-Lebowitz} A. DeMasi, R. Esposito, J. L. Lebowitz, Incompressible Navier-Stokes and Euler limits of the Boltzmann equation. {\it Comm. Pure Appl. Math.} 42 (1989), no. 8, 1189-1214. 


\bibitem{DiPerna-Lions}  R. DiPerna and P.-L. Lions, On the Cauchy problem for Boltzmann equations: global existence and weak stability, {\it Ann. of Math.} {\bf 130} (1989), no. 2, 321-366.


\bibitem{Ellis-Pinsky} R. S. Ellis and M. A. Pinsky, The first and second fluid approximations to the linearized Boltzmann equation. {\it J. Math. Pures Appl.} (9) 54 (1975), 125-156. 

\bibitem{fujitakato}  H. Fujita et T. Kato, On the Navier-Stokes initial value problem I,
{\it  Archive for Rational Mechanics and Analysis}, 
{16}  (1964),  269-315.

\bibitem{GIP} I. Gallagher, D. Iftimie and F. Planchon, Asymptotics
   and stability for global solutions to the
   Navier-Stokes equations,  {\it  Annales de l'Institut Fourier},
   {\bf 53}, 5 
(2003), pages 1387-1424. 


\bibitem{Glassey} R. T. Glassey, The Cauchy problem in kinetic theory, {\it Society for Industrial and Applied Mathematics (SIAM), Philadelphia, PA}, (1996) xii+241 pp.


\bibitem{Grenier} E. Grenier,
Oscillatory perturbations of the Navier-Stokes equations, {\it J.
Math. Pures Appl.}, {\bf 76}, (1997), no. $6$, pages~477-498.


\bibitem{GSR1} F. Golse and L. Saint-Raymond, The Navier-Stokes limit of the Boltzmann equation
for bounded collision kernels, {\it Invent. Math.} {155} (2004), no. 1, 81-161.

\bibitem{GSR2} F. Golse and L. Saint-Raymond, The incompressible Navier-Stokes limit of the Boltzmann equation for hard cutoff potentials, {\it J. Math. Pures Appl.} {91} (2009), no. 5, 508-552.

\bibitem{Grad} H. Grad,  Asymptotic equivalence of the Navier-Stokes and nonlinear Boltzmann equations, {\it AEC Research and Development report}   (1964).
\bibitem{Gradhydro} H. Grad. Asymptotic theory of the Boltzmann equation II Rarefied Gas Dynamics {\it Proc. of the 3rd Intern. Sympos. Palais de l'UNESCO}, (Paris, 1962) Vol. I, 26--59.

\bibitem{GMM} M. P. Gualdani, S. Mischler and C. Mouhot, Factorization of non-symmetric operators and exponential H-theorem. {\it M\'em. Soc. Math. Fr. (N.S.)} No. 153 (2017), 137 pp. ISBN: 978-2-85629-874-9



\bibitem{Hilbert} D. Hilbert, {\it Sur les probl\`emes futurs des math\'ematiques}, Compte-Rendu du 2\`eme Congr\'es
International de Math\'ematiques, tenu \`a Paris en 1900 : Gauthier-Villars (Paris, 1902), p. 58-114.

\bibitem{Lachowicz} M. Lachowicz, On the initial layer and the existence theorem for the nonlinear
Boltzmann equation. {\it Math. Methods Appl. Sci.} {\bf 9} (1987), 342--366.

\bibitem{lemarie1} P.-G. Lemari\'e, {\it  Recent developments in the Navier-Stokes problem}, CRC Press.

\bibitem{lemarie2} P.-G. Lemari\'e, {\it   The Navier-Stokes problem in the 21st century}, CRC Press.

\bibitem{leray} J. Leray, Sur le mouvement d'un liquide visqueux emplissant l'espace,  {\it Acta Math.} {\bf 63} (1934), 193-248.

\bibitem{leray2D} J. Leray,
\'Etude de diverses \'equations int\'egrales non lin\'eaires et
de quelques probl\`emes que pose l'hydrodynamique.
{\it Journal de  Math\'ematiques Pures et  Appliqu\'ees}, {\bf 12} (1933),  pages 1-82.

\bibitem{Levermore-Masmoudi} D. Levermore and N. Masmoudi, From the Boltzmann equation to an incompressible Navier-Stokes-Fourier system. {\it Arch. Ration. Mech. Anal. 196} (2010), no. 3, 753-809.

\bibitem{lionsmasmoudi} P.-L. Lions, N. Masmoudi, From Boltzmann equation to the Navier-Stokes and
Euler equations II, {\it Archive Rat. Mech.  Anal. } {158} (2001), 195-211.

\bibitem{luzhang} X. Lu and Y. Zhang, On the nonnegativity of solutions of the Boltzmann equation, {\it Transport Theory and Statistical Physics} {30}(7) (2001), pages  641-657.
 
\bibitem{Mouhot-Neumann} C. Mouhot and L. Neumann, Quantitative perturbative study of convergence to equilibrium for collisional kinetic models in the torus, {\it Nonlinearity} 19 (2006), no. 4, 969--998.

\bibitem{Mouhot-Villani} C. Mouhot and C. Villani, Regularity theory for the spatially homogeneous Boltzmann equation with cut-off. {\it Arch. Ration. Mech. Anal.} 173 (2004), no. 2, 169--212.

\bibitem{Nishida} T. Nishida, Fluid dynamical limit of the nonlinear Boltzmann equation to the level of the compressible Euler equation. {\it Comm. Math. Phys.} {\bf61} (1978), 119--148.

\bibitem{Saint-Raymond} L. Saint-Raymond, {\it Hydrodynamic limits of the Boltzmann equation},  Lecture Notes in
Mathematics, vol. 1971, (2009) Springer, Berlin.



\bibitem{Schochet} S.
Schochet,  Fast singular limits of hyperbolic  PDEs, {\sl
Journal of Differential Equations}, {\bf 114}, (1994),
pages~476-512.

\bibitem{Schonbek} M.
Schonbek,   Large time behaviour of solutions to the Navier-Stokes equations in $H^m$ spaces, {\it Comm. PDEs} {\bf 20} (1995), 103-117.

\bibitem{Schonbek24} M.
Schonbek,   $L^2$ decay for weak solutons of the  Navier-Stokes equations,  {\it Arch. Rational Mech. Anal.} {\bf 88} (1985), no. 3, 209-222.  



\bibitem{Ukai} S. Ukai,  Solutions of the Boltzmann equation, {\it Patterns and waves, Stud. Math. Appl., } {\bf 18}, North-Holland, Amsterdam, (1986), 37-96.

\bibitem{Ukai-Yang} S. Ukai and T. Yang,  {\it Mathematical theory of the Boltzmann equation}, Lecture Notes Ser No 8, Liu Bie Ju Center for Math. Sci., City University of Hong-Kong, 2006.
 
\bibitem{wiegner}  M. Wiegner,  Decay results for weak solutions of the Navier-Stokes equations on $\R^n$, {\it J. London Math. Soc.} {\bf 35} (1987), no. 2, 303-313. 
\end{thebibliography}
\end{document}